\documentclass[11pt]{article}
\usepackage{amsfonts}
\usepackage{amssymb}
\usepackage{amsmath,amsthm}
\usepackage{amsfonts}
\usepackage{graphicx}
\usepackage{bm}
\usepackage{xcolor}

\def\beq{\begin{equation}}
\def\eeq{\end{equation}}
\def\rit{\mathbb{R}}

\newtheorem{theorem}{Theorem}

\newtheorem{lemma}[theorem]{Lemma}

\newtheorem{remark}[theorem]{Remark}

\begin{document}

\centerline{\Large \bf Bifurcations of viscous boundary layers in the half space}

\bigskip

\centerline{D. Bian\footnote{School of Mathematics and Statistics, Beijing Institute of Technology, Beijing 100081, China. Email: biandongfen@bit.edu.cn and emmanuelgrenier@bit.edu.cn}, 
E. Grenier$^1$,
%M. Haragus\footnote{FEMTO-ST Institute, University of Franche Comt\'e,  25030 Besan\c con cedex, France. Email: mharagus@univ-fcomte.fr},
G. Iooss\footnote{Laboratoire J.A.Dieudonné, I.U.F.,
Université Côte d’Azur,
Parc Valrose
06108 Nice Cedex 02, France}}

%%%%%%%%%

\subsubsection*{Abstract}

%%%%%%%%%

It is well-established that shear flows are linearly unstable provided the viscosity is small enough, 
when the horizontal Fourier wave number lies in some interval, 
between the so-called lower and upper marginally stable curves.
In this article, we prove that, under a natural spectral assumption, 
shear flows undergo a Hopf bifurcation near their upper marginally stable curve.
In particular, close to this curve, there exists space periodic traveling waves solutions of the full incompressible
Navier-Stokes equations.
For the linearized operator, the occurrence of an essential spectrum containing the entire negative real axis causes certain difficulties which are overcome. 
Moreover, if this Hopf bifurcation is super-critical, these time and space periodic solutions are linearly
and nonlinearly asymptotically stable.

%%%%%%%%%%%%%%%%%%%%%%%%%%%%%%%%%%%%%%%%%%%%%%%%%%%%%%%%%%%%%%%%%%%%%%%%%

\section{Introduction}

%%%%%%%%%%%%%%%%%%%%%%%%%%%%%%%%%%%%%%%%%%%%%%%%%%%%%%%%%%%%%%%%%%%%%%%%%

In this paper, we consider the incompressible Navier-Stokes equations in the half space $\Omega = \rit \times \rit_+$
 \beq \label{NS1} 
\partial_t u^\nu + (u^\nu \cdot \nabla) u^\nu - \nu \Delta u^\nu + \nabla p^\nu = f^\nu,
\eeq
\beq \label{NS2}
\nabla \cdot u^\nu = 0 ,
\eeq
together with the Dirichlet boundary condition
\beq \label{NS3} 
u^\nu = 0 \quad \hbox{when} \quad y = 0
\eeq
and address the classical question of the linear and nonlinear stability of shear flows for these equations.

A shear flow is a stationary solution of (\ref{NS1},\ref{NS2},\ref{NS3}) of the form
$$
U(y) = (U_s(y),0), \qquad f^\nu = (- \nu \Delta U_s,0),
$$
where $U_s(y)$ is a smooth function, vanishing at $y = 0$ and converging exponentially fast
at infinity to some constant $U_+$. 

The question of the linear stability of such shear flows is one of the most classical questions
in Fluid Mechanics, which has been intensively studied since the pioneering work of Lord Rayleigh
at the end of the $19^{th}$ century. The situation has been progressively understood in the $20^{th}$ century
thanks to the works of L. Prandtl, Orr, Sommerfeld, Schlichting and C.C. Lin to only quote a few names.
We in particular refer to \cite{Reid,Reid2,Schmidt} for a detailed presentation of the approach in physics.

\begin{figure}[th]
\begin{center}
\includegraphics[width=7cm]{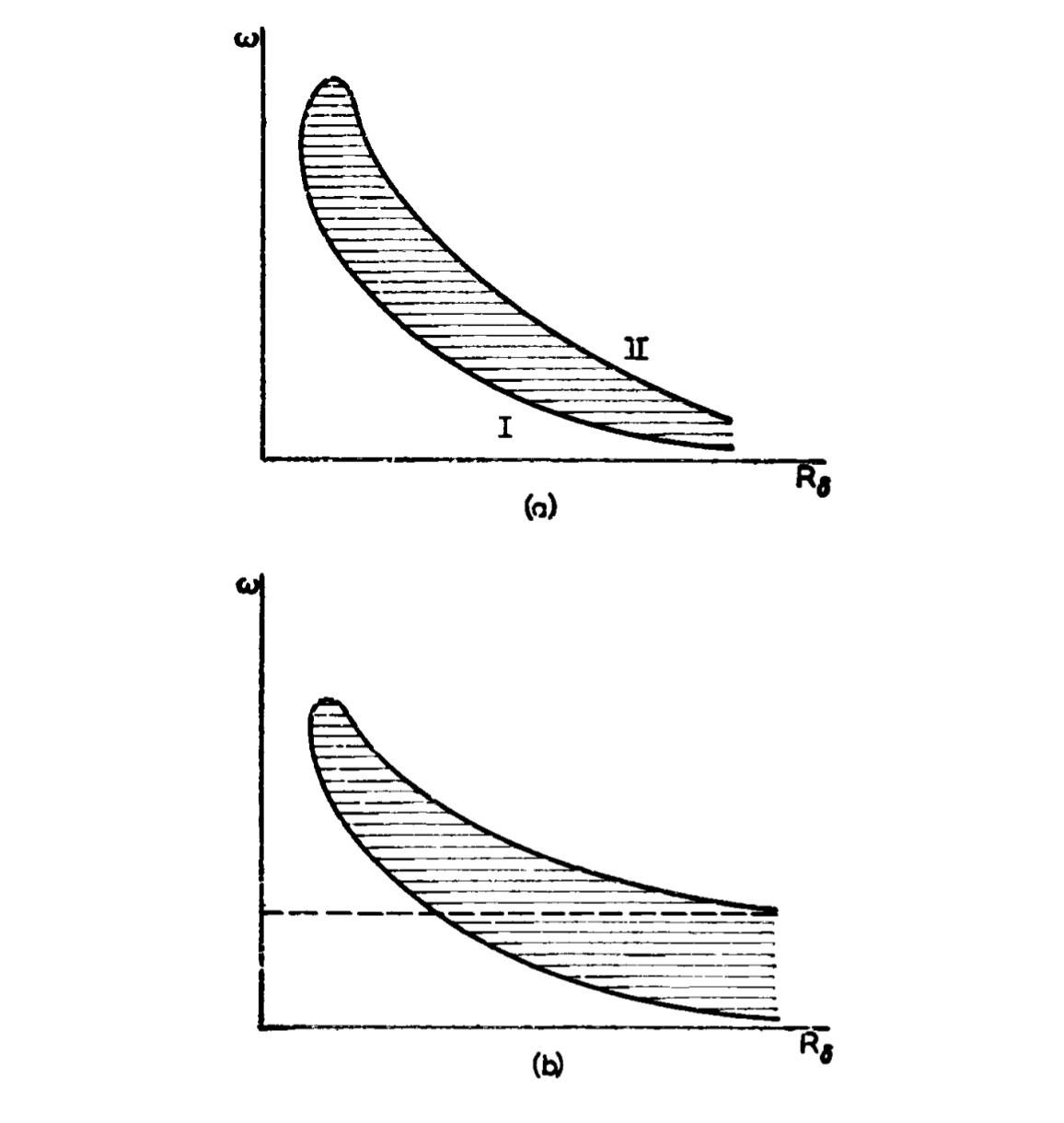}
\end{center}
\caption{Stability of shear flows: 
horizontally, the Reynolds number (inverse of $\nu$), vertically the horizontal wave length
$\alpha$ of the perturbation. In grey, the unstable area. 
Top sub-figure: Euler-stable profile. Bottom sub-figure: Euler-unstable profile.
From \cite{Landau}.}
\label{unstable}
\end{figure}

Roughly speaking, shear flows can be classified into two categories:

\begin{itemize}
\item Some are spectrally unstable for the Euler equations, namely in the case $\nu = 0$ 
(see the bottom sub-figure of Figure \ref{unstable}).
According to Rayleigh criterium, such shear flows have inflection points. 

\item Other are spectrally stable for the Euler equations, which is the case if they have no inflection points,
for instance if they are convex or concave (see the top sub-figure of Figure \ref{unstable}).

\end{itemize}

In both cases, provided the Reynolds number is large enough, namely provided the viscosity $\nu$ is small enough,
the shear layer is linearly unstable with respect to perturbations whose horizontal wavenumber $\alpha$ 
lies in some interval $[\alpha_-(\nu),\alpha_+(\nu)]$ for some increasing functions $\alpha_\pm(\nu)$. 
The function $\alpha_-(\nu)$
(respectively $\alpha_+(\nu)$) is called the lower (respectively upper) marginal stability curve.

Let us fix $\nu = \nu_0$ small enough, which corresponds to a large enough Reynolds number.
Then if $\alpha > \alpha_{+}(\nu_0)$, we note that the harmonics of $\alpha$, namely all the positive
multiple of $\alpha$, remain larger than the basic harmonic $\alpha_{+}(\nu_0)$ and thus are all stable.
As a consequence, we expect that the shear layer $U$ is linearly and nonlinear stable with respect
to perturbations which are $2 \pi / \alpha$ periodic.

However, if $\alpha < \alpha_{+}(\nu_0)$ and close to $\alpha_{+}(\nu_0)$, a linear instability appears
and, following \cite{Schmidt} (section $5.3$), we expect a bifurcation: small perturbations will
grow exponentially till the nonlinear term saturates them, leading to a bifurcated solution,
which is a traveling wave of small amplitude.

\medskip

The aim of this article is to investigate mathematically this physical scenario.
We first formalize the spectral situation depicted on Figure \ref{unstable},
which leads to the following set of assumptions and then study the bifurcation arising at 
$\alpha_+(\nu)$.

\medskip

We take the Fourier transform in the
$x$ variable, with dual Fourier variable $\alpha$. We denote by $Sp_{\alpha,\nu}$ the
spectrum of the linearized Navier-Stokes equations after this Fourier transform, and by
$\lambda(\alpha,\nu) \in Sp_{\alpha,\nu}$ the eigenvalue (if it exists and is unique) with the largest real part.

In this article, we assume that, 
when the viscosity $\nu$ is small enough,
there exist two smooth and increasing functions $\alpha_\pm(\nu)$, 
defined for $\nu$ small enough, with $\alpha_-(\nu) < \alpha_+(\nu)$, such that

\begin{itemize}

\item[(A1)] for $\alpha_-(\nu) < |\alpha| < \alpha_+(\nu)$, $\lambda(\alpha,\nu)$ exists, is unique,
and satisfies $$
\Re \lambda(\alpha, \lambda) > 0. 
$$
Moreover, $\lambda$ is simple, with corresponding  eigenvector $\zeta(x,y)$ of the form
$$
\zeta(x,y) = \nabla^{\perp} \Bigl[ e^{i \alpha x} \psi_{\alpha,\nu}(y) \Bigr],
$$
%$$
%\overline{\zeta}(x,y) = \nabla^{\perp} \Bigl[ e^{-i \alpha x} \overline{\psi}(y) \Bigr],
%$$
for some smooth stream function $\psi_{\alpha,\nu}(y)$ which is exponentially decreasing at infinity,

\item[(A2)] for $|\alpha| > \alpha_+(\nu)$, close to $\alpha_+(\nu)$,
$\lambda(\alpha,\nu)$ exists, is unique and satisfies
$$
\Re \lambda(\alpha,\nu) < 0,
$$

\item[(A3)] for $| \alpha | = \alpha_+(\nu)$, $\lambda(\alpha,\nu)$ is well defined and purely imaginary.
Moreover, defining $\nu_0(\alpha)$ by $\alpha=\alpha_+(\nu_0)$
$$
(\partial_\nu \Re \lambda) (\alpha,\nu) |_{\nu = \nu_0(\alpha)} > 0.
$$
\end{itemize}

The mathematical proof of the existence of an unstable
mode for some values of $\alpha$
has been initiated in \cite{Guo}, where it is 
established that any shear flow is spectrally
unstable for the incompressible Navier-Stokes equations provided the viscosity is small enough,
a first result later improved by \cite{Zhang} and \cite{BG4}. In this last paper, 
we proved that strictly convex or concave flows satisfy the set of assumptions (A1), (A2) and (A3)
under the additional assumptions $U_s'(0) \ne 0$ and $U_+ \ne 0$.

When we cross the upper marginally stable curve $\alpha = \alpha_+(\nu)$, two eigenvalues cross
the imaginary axis (corresponding to $\pm \alpha$): we are exactly in the situation of an Hopf bifurcation.
This paper focuses on the study of this bifurcation.

Let us now state our main result in an informal way. We refer to Theorems \ref{maintheo1}, \ref{maintheo2} 
and \ref{maintheo3} for detailed statements, including precise definitions of the function
spaces.

\begin{theorem}\label{maintheo}
Let us assume (A1), (A2) and (A3).
Let $\nu_0 > 0$ and let $\alpha = \alpha_+(\nu_0)$, let us consider
perturbations which are $2 \pi / \alpha$ periodic in $x$.
Then, at $\nu=\nu_0$, the system undergoes a Hopf bifurcation:
\begin{itemize}

\item for $\nu < \nu_0$, the shear flow $U$ is linearly and nonlinearly stable,

\item if the Hopf bifurcation is subcritical, there exists a time and space periodic solution
of (\ref{NS1},\ref{NS2},\ref{NS3}) for $\nu < \nu_0$ sufficiently close to $\nu_0$,
and this time and space periodic solution is linearly unstable,

\item if the Hopf bifurcation is supercritical, there exists a time and space periodic solution
of (\ref{NS1},\ref{NS2},\ref{NS3}) for $\nu > \nu_0$, sufficiently close to $\nu_0$, and this
time and space periodic solution is linearly and nonlinearly stable.

\end{itemize}

\end{theorem}

There exist many works dealing with Hopf bifurcations on Navier-Stokes equations (see \cite{Haragus}
for detailed references), however, all these works consider bounded domains or periodic domains
(for Couette-Taylor flow \cite{CI} or Bénard-Rayleigh convection for example).
Recently, some works  deal with parallel flows, like 
for example Poiseuille flow in \cite{Charru} 
at section $6.4$, where a sub-critical Hopf bifurcation of a traveling wave is completely treated
(see also \cite{BHG} for the Hopf bifurcation of a shear flow between two plates).

None of these works deal with an half plane domain ($x\in \mathbb{R},y\in \mathbb{R}^+$),
 assuming periodicity only in the $x$ direction, the $y$ direction being unbounded,
%Up to the best of our knowledge, this result is the first %bifurcation result for the genuine Navier-Stokes equations
%in a half space,
which is a problem of serious physical interest. 
 As the domain is unbounded in the $y$ direction, the essential spectrum of the linearized operator
goes up to zero and there is no spectral gap. The classical tools of bifurcation theory cannot
be applied and we have to design a new approach.
Moreover, in cases where the shear flow $U$ is stable, perturbation do not decay exponentially
fast, but only like polynomially fast, as $t^{-1}$.

\medskip

The super-criticality or sub-criticality of the Hopf bifurcation depends on the sign of a coefficient
which unfortunately can only be numerically studied. In the case of exponential profiles
$U_s(y) = 1 - e^{-y}$, numerical evidence \cite{BG5} indicates that the Hopf bifurcation is supercritical.

\medskip

The paper is organized as follows: in section $2$, we introduce the various function spaces, the
various linear operators and investigate their spectra and their related semi-groups.
In section $3$ we prove the nonlinear stability of the shear flow $U$ above the upper marginal stability curve.
In section $4$ we investigate the bifurcation and prove the existence of a traveling wave, time periodic solution to the
Navier Stokes equations. In section $5$, we focus on the linear and non linear stability of this
periodic solution.
The proofs of the various lemmas are detailed in the Appendix.

%%%%%%%

\subsubsection*{Notations}

%%%%%%

In all this paper, $\eta$ will be a sufficiently small positive number.
We define the one-dimensional periodic torus of period $2 \pi / \alpha$ to be $\mathbb{T} = \mathbb{R} / (2 \pi / \alpha) \mathbb{Z}$.
We define the Banach space $C^0_\eta(\mathbb{R}^+)$ by its norm
$$
\|u\|_{C_{\eta}^{0}} = \underset{y\in \mathbb{R}^{+}}{\sup }|u(y) e^{\eta y}|.
$$
The components of a two dimensional vector field $v$ will be denoted by $v = (v^x,v^y)$.
We denote by $D$ the derivative with respect to $y$.
We denote by $u_n(y)$ the $n^{th}$ Fourier component of a function $u(x,y)$.

By a slight abuse of language, we will also denote the vector $(U_+,0)$ by $U_+$.

%%%%%%%%%%%%%%%%%%%%%%%%%%%%%%%%%%%%%%%%%%%%%%%%%%%%%%%%%%%%%%%%%%%%%%%%%%%%%%%%%

\section{Preliminaries}

%%%%%%%%%%%%%%%%%%%%%%%%%%%%%%%%%%%%%%%%%%%%%%%%%%%%%%%%%%%%%%%%%%%%%%%%%%%%%%%%%

%%%%%%%%%%%%%%%%%%%%%%%%%%%%%%%%%%%%%%%%%%%%

\subsection{Function spaces}

%%%%%%%%%%%%%%%%%%%%%%%%%%%%%%%%%%%%%%%%%%%%

We first observe that if $v = (v^x,v^y)$ is a two dimensional divergence free vector field which is independent on $x$ 
and vanishes at $y = 0$, then, using the incompressibility condition, $v^y = 0$.

Let us define the Banach space
\begin{equation*}
\mathcal{X}_{\eta}=\mathring{L}_{\eta}^{2}\oplus \mathring{L}^{\infty },
\end{equation*}%
so that $v\in \mathcal{X}_{\eta}$ can be written as%
\begin{eqnarray*}
v &=&\widetilde{v}+v_{0},\text{ }\widetilde{v}\in \mathring{L}_{\eta}^{2}, \\
v_{0} &=&(v_{0}^{x},0),\text{ }v_{0}^{x}\in L^{\infty }(\mathbb{R}^{+}),
\end{eqnarray*}%
where we define%
\begin{equation*}
\mathring{L}_{\eta}^{2}=\Bigl\{ \widetilde{v}\in L_\eta^{2}; 
\, \nabla \cdot \widetilde{v}=0, \, \widetilde{v}^{y}|_{y=0}=0 \Bigr\} ,
\end{equation*}%
with%
\begin{equation*}
L_{\eta}^{2}=\Bigl\{ \widetilde{u}\in [L^{2}[\mathbb{T},C_{\eta}^{0}(\mathbb{R}^{+})]]^{2};
\, \int_{0}^{2\pi/\alpha }\widetilde{u}(x,y) \, dx=0,\text{ }\|\widetilde{u}\|_{L_{\eta}^{2}}<\infty \Bigr\},
\end{equation*}%
\begin{equation*}
\mathring{L}^{\infty } = \Bigl\{ (v_{0}^{x},0); \, v_{0}^{x}\in L^{\infty }(\mathbb{R}^{+})\Bigr\},
\end{equation*}%
\begin{eqnarray*}
\|\widetilde{u}\|_{L_{\eta}^{2}}^{2} &=&\sum_{|n|\geq
1}\|u_{n}\|_{C_{\eta}^{0}}^{2},\\
\|v_{0}\|_{\mathring{L}^{\infty }} &=&\|v_{0}^{x}\|_{L^{\infty }}.
\end{eqnarray*}%
Note that we do not assume any decay in $y$ on the 0-Fourier mode $v_0$.
By analogy, we define the Banach spaces $H_{\eta}^{2},$ $\mathring{H}_{\eta}^{2}$ and
\begin{equation*}
\mathcal{Z}_{\eta}=\mathring{H}_{\eta}^{2}\oplus \mathring{W}^{2,\infty },
\end{equation*}%
where%
\begin{equation*}
\mathring{H}_{\eta}^{2}= \Bigl\{\widetilde{v}\in H_{\eta}^{2}; \, \nabla \cdot \widetilde{v}%
=0, \, \widetilde{v}|_{y=0}=0, \, \|\widetilde{v}\|_{\mathring{H}_{\eta}^{2}}<\infty \Bigr\},
\end{equation*}%
\begin{eqnarray*}
H_{\eta}^{2} &=&\Bigl\{ \widetilde{u}\in \lbrack H^{2}[\mathbb{T},C_{\eta}^{2}(\mathbb{R}^{+})]]^{2}; \,
 \int_{0}^{2\pi/\alpha }\widetilde{u}(x,y) \, dx=0, \, \|\widetilde{u}%
\|_{H_{\eta}^{2}}<\infty \Bigl\} ,
\end{eqnarray*}%
\begin{equation*}
\mathring{W}^{2,\infty }= \Bigl\{ (v_{0}^{x},0); \, v_{0}^{x}|_{y=0}=0, \, v_{0}^{x} 
\in W^{2,\infty }(\mathbb{R}^{+})\Bigr\},
\end{equation*}%
with the corresponding norms
\begin{equation*}
\|\widetilde{v}\|_{H_{\eta}^{2}}^{2}=\sum_{|n|\geq
1}\|D^{2}v_{n}\|_{C_{\eta}^{0}}^{2}+n^{2}\|Dv_{n}\|_{C_{\eta}^{0}}^{2}+n^{4}\|v_{n}\|_{C_{\eta}^{0}}^{2}
\end{equation*}%
and
\begin{equation*}
\|v_{0}\|_{\mathring{W}^{2,\infty }}=\|v_{0}^{x}\|_{W^{2,\infty }},
\end{equation*}%
so that $\mathcal{Z}_{\eta}\mathcal{\hookrightarrow X}_{\eta}$ densely.

We also define the intermediary Banach space
\begin{equation*}
\mathcal{Y}_{\eta}=\mathring{H}_{\eta}^{1}\oplus \mathring{W}^{1,\infty
}\hookrightarrow \mathcal{X}_{\eta},
\end{equation*}%
with 
\begin{eqnarray*}
\mathring{H}_{\eta}^{1} &=&\Bigl\{ \widetilde{v}\in \lbrack H^{1}[\mathbb{T},C_{\eta}^{1}( \mathbb{R}^{+})]]^{2}; \\
&&\quad \int_{0}^{2\pi/\alpha }\widetilde{v}(x,y) \, dx=0,\text{ } \nabla \cdot 
\widetilde{v}=0,\widetilde{v}^{y}|_{y=0}=0,\text{ }\|\widetilde{v}\|_{%
\mathring{H}_{\eta}^{1}}<\infty \Bigr\} ,
\end{eqnarray*}%
\begin{equation*}
\mathring{W}^{1,\infty }= \Bigl\{(v_{0}^{x},0); \, v_{0}^{x}\in W^{1,\infty }(\mathbb{R}^{+})\Bigr\}
\end{equation*}%
and the norm%
\begin{equation*}
\|\widetilde{v}\|_{\mathring{H}_{\eta}^{1}}^{2}=\sum_{|n|\geq
1}n^{2}\|v_{n}\|_{C_{\eta}^{0}}^{2}+\|Dv_{n}\|_{C_{\eta}^{0}}^{2}.
\end{equation*}%

%\begin{remark}
%We might be tempted to include the exponential decay $e^{-ky}$ in the $0$ -
%Fourier component. This has a heavy consequence for the spectrum of the
%Linear operators, as explained in section \ref{sect: spectra}, implying
%tragic consequences for the stability studies.
%\end{remark}

%%%%%%%%%%%%%%%%%%%%%%%%%%%%%%%%%%%%%%%%%%%%%%%%%%%%%%%%

\subsection{The Helmholtz decomposition in $\mathcal{X}_{\eta}$}

%%%%%%%%%%%%%%%%%%%%%%%%%%%%%%%%%%%%%%%%%%%%%%%%%%%%%%%%

Let $u\in L_{\eta}^{2}\oplus (L^{\infty })^{2}$. Then $u$ can be decomposed in  
$u=\widetilde{u}+u_{0}$
with $\widetilde{u}\in L_{\eta}^{2}$ 
%(hence with $\int_{0}^{2\pi/\alpha } \widetilde{u}(x,y)dx=0)$ , 
and $u_{0}\in \lbrack L^{\infty }(\mathbb{R}%
^{+})]^{2}$. We further decompose $\widetilde{u}$ in
\begin{eqnarray*}
\widetilde{u} = \widetilde{v}+\nabla \widetilde{\phi }
\end{eqnarray*}
with $\widetilde{v}\in \mathring{L}_{\eta}^{2}$ and
\begin{eqnarray*}
\Delta \widetilde{\phi } = \nabla \cdot \widetilde{u}, \quad 
\frac{\partial \widetilde{\phi }}{\partial y}|_{y=0}=\widetilde{u}^{y}|_{y=0}\in L^{2}(%
\mathbb{T}),
\end{eqnarray*}%
where $\nabla \cdot \widetilde{u}$ is understood in the distribution sense.
The component $u_{0}$ can be
decomposed as%
\begin{eqnarray*}
u_{0} &=&\binom{v_{0}^{x}}{0}+\binom{0}{D\phi _{0}},\text{ }\binom{v_{0}^{x}%
}{0}\in \mathring{L}^{\infty }, \\
\end{eqnarray*}%
with $D \phi_0 = u_0^y \in L^\infty(\mathbb{R}^+)$,
so that $\widetilde{v}+v_{0}\in \mathcal{X}_{\eta}$.  We note that $v_{0}$
and $\phi _{0}$ are independent of $x$. We now define the projection $\Pi$ on divergence free vector fields by 

\begin{lemma}
\label{lem:decomp Hk} Let $\eta < \alpha / 2$.
Let $u\in L_{\eta}^{2}\oplus (L^{\infty })^{2}$, which may be decomposed in  
\begin{equation*}
u=\widetilde{u}+u_{0},\text{ }\widetilde{u}\in L_{\eta}^{2},\text{ }u_{0}\in
(L^{\infty }(\mathbb{R}^{+}))^{2}.
\end{equation*}%
We can further decompose $\widetilde{u}$ in $\widetilde{u}=\widetilde{v}+\widetilde{\nabla \phi }$
where $\widetilde{v} \in \mathring{L}_{\eta}^{2}$.
We define the projection $\Pi u$ by 
\begin{equation*}
\Pi u=\Pi \widetilde{u}+(\Pi u)_{0}
\end{equation*}%
where $\Pi \widetilde{u} = \widetilde{v}$ and $(\Pi u)_0 = (v_0^x,0)$.
The projection $\Pi$ is bounded from $L_{\eta}^{2}$ to $\mathring{L}_{\eta}^{2}$
and from $(L^{\infty})^{2}$ to $\mathring{L}^{\infty }$
% is defined as follows. 
% The decomposition of $\widetilde{u}\in L_{\eta}^{2}$ as 
% leads to $\widetilde{v}=\Pi \widetilde{u}\in \mathring{L}_{\eta}^{2},$ with $%
% \Pi $ bounded from $L_{\eta}^{2}$ to $\mathring{L}_{\eta}^{2}.$ The mapping $%
% u\mapsto v_{0}=(\Pi u)_{0}=\binom{v_{0}^{x}}{0}$ is bounded from $(L^{\infty
% })^{2}$ to $\mathring{L}^{\infty }.$
\end{lemma}

The projector $\Pi$ is highly classical, however our function spaces are not the usual ones.
We thus detail the proof of this Lemma in Appendix \ref{App:decompHk}.

%%%%%%%%%%%%%%%%%%%%%%%%%%%%%%%%%%%%%%%%%%%%%%%%%%%%%%%%%%%%%%%%%%%%%%%%

\subsection{Linear operators\label{sect: lin operators}}

%%%%%%%%%%%%%%%%%%%%%%%%%%%%%%%%%%%%%%%%%%%%%%%%%%%%%%%%%%%%%%%%%%%%%%%%

We define the linear operators $\boldsymbol{L}_{(\nu )},\boldsymbol{L}^{(0)}_{(\nu )},%
\boldsymbol{L}^{(1)}$ by
\begin{align*}
\boldsymbol{L}_{(\nu )} &=\boldsymbol{L}_{(\nu )}^{(0)}+\boldsymbol{L}^{(1)}, \\
\boldsymbol{L}_{(\nu)}^{(0)}v &=\nu \Pi \Delta v, \quad
\boldsymbol{L}^{(1)}v =-\Pi \Bigl[ (U\cdot \nabla )v+(v\cdot \nabla )U \Bigr], 
\end{align*}
and further decompose $\boldsymbol{L}_{(\nu)}^{(0)}$ in 
$$
\boldsymbol{L}_{(\nu )}^{(0)} =\boldsymbol{L}^{(0)}+(\nu-\nu_0)\boldsymbol{L}'
$$
with
\begin{align*}
\boldsymbol{L}^{(0)}v &=\nu _{0}\Pi \Delta v,\text{ }\boldsymbol{L}^{\prime }v=\Pi \Delta v.
\end{align*}
Note that $\boldsymbol{L}'$ and
$\boldsymbol{L}^{(0)}_{(\nu)}$ are Stokes operators.
We easily obtain the following Lemma.

\begin{lemma}
\label{lem: L1} Assume that $U\in C^{1}(\mathbb{R}^{+})$ with 
\begin{equation*}
|U(y)|+|DU(y)|\leq M<\infty ,\text{ \ }y\in \mathbb{R}^{+},
\end{equation*}%
then the linear operator $\boldsymbol{L}^{(1)}$ cancels on $\mathring{W}%
^{1,\infty }$. For $\eta <\alpha /2$ and $v =  \widetilde{v}+v_{0}\in \mathcal{Z}%
_{\eta}$ with $\widetilde{v}\in \mathring{H}_{\eta}^{2},$ $v_{0}\in \mathring{W}%
^{2,\infty },$ then%
\begin{equation*}
\boldsymbol{L}^{(1)}v=\boldsymbol{L}^{(1)}\widetilde{v}\in \mathring{H}%
_{\eta}^{1}.
\end{equation*}%
Moreover for $u\in \widetilde{u}+u_{0}\in \mathcal{Y}_{\eta}$ with $\widetilde{u%
}\in \mathring{H}_{\eta}^{1},$ $u_{0}\in \mathring{W}^{1,\infty },$ then%
\begin{equation*}
\boldsymbol{L}^{(1)}u=\boldsymbol{L}^{(1)}\widetilde{u}\in \mathring{L}%
_{\eta}^{2}.
\end{equation*}%
There exists $C>0$ such that we have the estimates%
\begin{equation*}
\|\boldsymbol{L}^{(1)}v\|_{\mathring{H}_{\eta}^{1}}\leq C\|\widetilde{v}\|_{%
\mathring{H}_{\eta}^{2}},
\end{equation*}%
\begin{equation*}
\|\boldsymbol{L}^{(1)}u\|_{\mathring{L}_{\eta}^{2}}\leq C\|\widetilde{u}\|_{%
\mathring{H}_{\eta}^{1}}.
\end{equation*}
\end{lemma}

The main observation is that
\begin{equation*}
(U\cdot \nabla )v+(v\cdot \nabla )U=U_s\partial _{x}v+\binom{v^{y}DU_s}{0},
\end{equation*}%
cancels when $v$ is independent of $x$ and $v^{y}=0,$ which is the case in $%
\mathring{W}^{1,\infty }.$

%%%%%%%%%%%%%%%%%%%%%%%%%%%%%%%%%%%%%%%%%%%%%%%%%%%%%%%%%%%

\subsection{Study of $\boldsymbol{L}_{(\protect\nu )}^{(0)}$ and $\boldsymbol{%
L}_{(\protect\nu )}$\label{sect: spectra}}

%%%%%%%%%%%%%%%%%%%%%%%%%%%%%%%%%%%%%%%%%%%%%%%%%%%%%%%%%%%

We now turn to the study of the spectrum and of the resolvent of the operator $\boldsymbol{L}^{(0)}_{(\nu)}$.
This operator is the classical Stokes operator, however our spaces
are not the usual ones, thus we have to restart its study from the beginning.

\begin{lemma} \label{lemma3}
\label{lem:spectrumL0 expky}The spectrum of $\boldsymbol{L}_{(\nu
)}^{(0)}=\Pi \nu \Delta  $ acting in $\mathring{L}^{\infty }$ is only
formed by an essential spectrum which equals 
\begin{equation*}
Sp|_{\mathring{L}^{\infty }} \, \boldsymbol{L}_{(\nu )}^{(0)}=(-\infty ,0]=%
\mathbb{R}^{-},
\end{equation*}%
and we have the estimates
\begin{equation*}
\Bigl\|(\boldsymbol{L}_{(\nu )}^{(0)}-\lambda \mathbb{I)}^{-1} \Bigr\|_{\mathcal{L}(%
\mathring{L}^{\infty })}\leq \left\{ 
\begin{array}{ll}
\sqrt{2}|\lambda |^{-1} & \text{when }\Re \lambda \geq 0, \\ 
(|\lambda |\cos \theta /2)^{-1} & \text{when }\lambda =|\lambda |e^{i\theta
} \text{ with } \Re \lambda <0.%
\end{array}%
\right.
\end{equation*}%
For $\lambda \in (-\infty ,0],$ the range of the operator $\boldsymbol{L}%
_{(\nu )}^{(0)}-\lambda \mathbb{I}$ is not closed in $\mathring{L}%
^{\infty }$.

Now consider the linear operator $\boldsymbol{L}_{(\nu )}^{(0)}=\Pi \nu
\Delta$ acting in $\mathring{L}_{\eta}^{2}$ . For any $\delta >0$ such
that $0<\delta <\pi /6,$ and $\lambda \in \mathbb{C}$ such that 
\begin{eqnarray*}
0 &\leq &|\arg (\lambda +\nu \alpha ^{2})|\leq \frac{ 2\pi}{3} - \delta , \\
0 &< &  \varepsilon_0 \le |\lambda +\nu \alpha ^{2}|,
\end{eqnarray*}%
there exists $C>0$ such that, for $\eta>0$ small enough,%
\begin{equation}
\Bigl\|(\boldsymbol{L}_{(\nu)}^{(0)}-\lambda \mathbb{I)}^{-1} \Bigr\|_{\mathcal{L}(%
\mathring{L}_{\eta}^{2})}\leq \frac{C}{|\lambda +\nu \alpha ^{2}|}.
\label{estimresol1zero}
\end{equation}
\end{lemma}

The spectrum of $\boldsymbol{L}_{(\nu )}^{(0)}$
on ${\mathring{L}_{\eta}^{2}}$ is
thus contained in an angular sector of $\mathbb{C}$ defined by%
\begin{equation*}
\frac{2\pi}{3}\leq |\arg (\lambda +\nu \alpha ^{2})|\leq \pi .
\end{equation*}%
The proof of this Lemma is detailed in Appendix \ref{proof2}.

\begin{remark}
If we chose to include the exponential decay $e^{-\eta y}$ for the $0$-Fourier
mode, this would imply that 
\begin{equation*}
Sp|_{\mathring{L}_{\eta}^{\infty }} \, \boldsymbol{L}_{(\nu )}^{(0)}\subset
\Bigl\{\lambda \in \mathbb{C}; \, |\Im \lambda |\leq 4\nu \eta^{2}(\nu \eta^{2}-\Re \lambda
)\Bigr\}
\end{equation*}%
which is a parabolic region centered on, and containing, the negative real
axis, bounded on the right side by $\Re \lambda =\nu \eta^{2}.$
\end{remark}

Moreover, we have the following corollaries.

\begin{lemma}
\label{lem:equiv norms}The linear operator $\boldsymbol{L}_{(\nu)}^{(0)}=\Pi \nu \Delta$ 
acting in $\mathring{L}_{\eta}^{2}$ has a
bounded inverse in $\mathcal{L}(\mathring{L}_{\eta}^{2},\mathring{H}_{\eta}^{2}).$
In $\mathring{H}_{\eta}^{2}$, the norms $\|\cdot \|_{\mathring{H}_{\eta}^{2}}$ and $%
\|\boldsymbol{L}_{(\nu )}^{(0)}\cdot \|_{\mathring{L}_{\eta}^{2}}$ are
equivalent.
\end{lemma}

This Lemma is a direct consequence of Lemma \ref%
{lem:spectrumL0 expky} since
\begin{equation*}
\|\widetilde{v}\|_{\mathring{H}_{\eta}^{2}}\leq M\|\boldsymbol{L}^{(0)}_{(\nu)}%
\widetilde{v}\|_{\mathring{L}_{\eta}^{2}}\leq Mc\|\Delta \widetilde{v}\|_{%
\mathring{L}_{\eta}^{2}}\leq Mc\|\widetilde{v}\|_{\mathring{H}_{\eta}^{2}}.
\end{equation*}

\begin{lemma}
\label{lem: locationspect} There exists $M>0$ such that the spectrum of $%
\boldsymbol{L}_{(\nu )}$ is included in 
$$
\Bigl\{ |\lambda +\nu \alpha ^{2}| \le M \Bigr\} \cup
 \Bigl\{ \frac{2\pi}{3}  \leq |\arg (\lambda +\nu \alpha ^{2})|\leq \pi  \Bigr\}
\cup (-\infty ,0].
$$
Let $\delta > 0$ be arbitrarily small.
In the domain, $|\lambda +\nu \alpha ^{2}| > M$ and
 $0  \leq |\arg (\lambda +\nu \alpha ^{2})|\leq 2 \pi / 3$,
we have
\begin{equation}
\Bigl\|(\boldsymbol{L}_{(\nu)} -\lambda \mathbb{I)}^{-1} \Bigr\|_{\mathcal{L}(%
\mathring{L}_{\eta}^{2})}\leq \frac{C}{|\lambda |}.
\label{estimresol1}
\end{equation}
\end{lemma}

We detail the proof of this lemma in Appendix \ref{proof3}.

We now decompose $L^{(1)}$ in
$$
L^{(1)} = L^{(1,0)} + L^{(1,c)},
$$
where
$$
L^{(1,0)} v  = - \Pi \Bigl[ (U_+ \cdot \nabla) v \Bigr],
\qquad 
L^{(1,c)} v = - \Pi \Bigl[ \Bigl( (U - U_+) \cdot \nabla \Bigr) v  + (v \cdot \nabla) U \Bigr].
$$
\begin{lemma}
\label{relatcomp} Assume that $U\in C^{2}(\mathbb{R}^{+})$ satisfies%
\begin{equation*}
|U(y) - U_+|+|DU(y)|+|D^{2}U(y)|\leq ce^{-\gamma y},\text{ }y\in \mathbb{R}^{+},
\end{equation*}%
then the linear operator $\boldsymbol{L}^{(1,c)}$ is relatively compact with
respect to $\boldsymbol{L}_{(\nu )}^{(0)}$ acting in $\mathcal{X}_{\eta}.$
\end{lemma}

This Lemma is proved in Appendix \ref{proof4} (see \cite{Kato} for the
definition of relative compactness of an operator with respect to another
operator).

\begin{lemma} \label{newlemma}
The essential spectrum of $L_{(\nu)}$ is the essential spectrum of $L_{(\nu)}^{(0)} + L^{(1,0)}$. 
It is the half line $(-\infty,0]$ if restricted to 
$\mathring{L}^{\infty }.$ 
The rest of the essential spectrum, for the operator restricted to $\mathring{L}_{\eta}^{2},$ is included in the region $\Sigma_{U_+}$ defined by
$$
\Re \lambda < - \nu (\alpha^2 - \eta^2),
$$
and either
$$
\Re \lambda < - {\nu (\Im \lambda)^2 \over
U_+^2 + 4 \nu^2 \eta^2} + \nu \eta^2
$$
or
\begin{equation*}
 2\pi /3\leq |\arg (\lambda +\nu \alpha ^{2})|\leq \pi .
\end{equation*}%
The rest of the spectrum of $L_{(\nu)} = L_{(\nu)}^{(0)} + L^{(1)}$ is uniquely formed by
isolated eigenvalues with finite multiplicities at finite distances of $0$. See Figure \ref{spectrumb}.
\end{lemma}

\begin{proof}
The localization of the essential spectrum of $L_{(\nu)}^{(0)} + L^{(1,0)}$ is proved in Appendix \ref{newappendix}.
The result on the spectrum of $\boldsymbol{L}_{(\nu)}$ comes
from the fact that it is the addition
to $\boldsymbol{L}^{(0)}_{(\nu)} + L^{(1,0)}$ of the linear operator $\boldsymbol{L}^{(1,c)}$
which is a relatively compact perturbation, using the fact that 
$U(y) - U_+$ tends to 0 exponentially as $y\rightarrow \infty$. We then apply
theorem 5.35 in \cite{Kato}.
\end{proof}

A corollary of the previous Lemmas is the following estimate for $\lambda $ on the imaginary axis,
far enough from the origin. This estimate will be useful in the study of bifurcating periodic solutions.

\begin{lemma} \label{lemma11}
\label{lem: inv n omega - L}Assume $\omega >0$ and $\eta>0$ small enough, then
there exists $C>0$ and $N>0$ such that for $|n|>N$ 
\begin{equation*}
\Bigl\|(in\omega -\boldsymbol{L}_{(\nu )})^{-1} \Bigr\|_{\mathcal{L}(\mathcal{X}_{\eta})}\leq \frac{C}{|n|}.
\end{equation*}
\end{lemma}

\begin{figure}[th]
\begin{center}
\includegraphics[width=7cm]{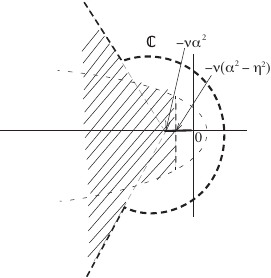}
\end{center}
\caption{a) Location of the spectrum of $\boldsymbol{L}_{(\nu)}$ inside the region
bounded by dashed line b) Location of the essential spectrum in the hatched
region  $\Sigma_{U_+}$ and on half left real line.}
\label{spectrumb}
\end{figure}

%%%%%%%%%%%%%%%%%%%%%%%%%%%%%%%%%%%%%%%%

\subsection{Semi-groups}

%%%%%%%%%%%%%%%%%%%%%%%%%%%%%%%%%%%%%%%%%

For the study of the linear and nonlinear stabilities of the basic flow $U,$
we need to understand the behavior of the semi-group $e^{\boldsymbol{L}%
_{(\nu )}t}$ for $t>0.$
We start with estimates on the Stokes flow.

\begin{lemma} Let $\eta > 0$ be small enough.
Assume $0<\delta <\pi /6.$ The linear operator $\boldsymbol{L}_{(\nu
)}^{(0)} $ is the infinitesimal generator of a bounded semi-group $e^{%
\boldsymbol{L}_{(\nu )}^{(0)}t}$ in $\mathcal{L}(\mathcal{X}_{\eta})$ which is 
holomorphic for $t\in \mathbb{C}$ in a sector of angle {$\pi /3-2\delta $}
centered on $\mathbb{R}^{+}.$ Moreover there exists $C>0$ such that for any $t \ge 0$,
\begin{eqnarray}
\|e^{\boldsymbol{L}_{(\nu )}^{(0)}t}\|_{\mathcal{L}(\mathring{L}_{\eta}^{2})}
&\leq &Ce^{-\nu \alpha ^{2}t},  \notag \\
\|e^{\boldsymbol{L}_{(\nu )}^{(0)}t}\|_{\mathcal{L}(\mathring{L}^{\infty })}
&\leq &C.  \label{est semigr L(0)}
\end{eqnarray}
\end{lemma}

\begin{proof}
This Lemma follows from the estimate (\ref{estimresol1}) and from classical results
on holomorphic semi-groups (see \cite{Kato}).
\end{proof}

Now the operator $\boldsymbol{L}_{(\nu )}$ is a perturbation of $\boldsymbol{%
L}_{(\nu )}^{(0)}$ with the same essential spectrum as $\boldsymbol{%
L}_{(\nu )}^{(0)}+\boldsymbol{L}^{(0,1)}$, and with possible eigenvalues in
a region bounded by the dashed curve in Figure \ref{spectrumb}. 
By assumptions (A1), (A2) and (A3), the spectrum of eigenvalues stays on the left half complex
plane for $\alpha$ fixed and $\nu < \nu_0(\alpha)$. 
Moreover, for $\nu = \nu_{0}$,  we have
two simple isolated eigenvalues $\pm i\omega _{0}$ with
no other eigenvalue of $\boldsymbol{L}_{(\nu _{0})}=\boldsymbol{L}^{(0)}+%
\boldsymbol{L}^{(1)}$ as well on the imaginary axis as on the right side of
the complex plane. \ It results from the bound of the spectrum found in
previous section, and from the fact that eigenvalues are isolated, that
there is a vertical line in the left complex plane bounding all other
eigenvalues, at a finite distance from the imaginary axis, hence staying on
the left side of the complex plane for $\nu $ close to $\nu _{0}.$ 

As soon
as we are able to justify the definition and obtain a good estimate of the
semi-group generated by $\boldsymbol{L}_{(\nu )},$ it will result that the
linear stability of the basic solution is determined by the sign of the real part of the
eigenvalues perturbing the above two eigenvalues, meaning in this case that 
\emph{spectral stability implies linear stability}.

\begin{remark}
If we introduced the decay in $e^{-\eta y}$ in the $0$- Fourier mode of the
function spaces, we could not obtain an estimate such as (\ref{est semigr
L(0)}), since we would have an exponential growing in $e^{\nu \eta^{2}t}$ for the bound. 
\end{remark}

For the study of the nonlinear stability of the basic flow $U$, we need to
estimate $e^{\boldsymbol{L}_{(\nu )}t}$ in $\mathcal{L(Y}_{\eta},\mathcal{Z}%
_{\eta})$ so that we can apply the semi-group to $B(u,u)\in \mathcal{Y}_{\eta}.$ This
is detailed  in the next Lemmas.

Since the part of the operator $\boldsymbol{L}_{(\nu )}$ acting in $%
\mathring{L}_{\eta}^{2}$ is uncoupled from the $\mathring{L}^{\infty }$ part,
we shall split the study in the two corresponding parts.

%%%%%%%%%%%%%%%%%%%%%%%%

\subsubsection{Study of $e^{\boldsymbol{L}_{(\protect\nu )}t}$ in $\mathring{L}%
^{\infty }$ and in $\mathring{L}_{\eta}^{\infty }$}

%%%%%%%%%%%%%%%%%%%%%%%%

\begin{lemma} \label{lem:estimsemigroup poids}
Assume that the eigenvalues $\lambda$ of $\boldsymbol{L}_{(\nu )}$ are such that $%
\Re \lambda <\gamma <0$.  Then, for any $v_{0}\in \mathring{W}_{\eta}^{2,\infty },$
 where%
\begin{equation*}
\|v_{0}\|_{\mathring{W}_{\eta}^{2,\infty }}=\|v_{0}\|_{\mathring{L}_{\eta}^{\infty }}
+\|Dv_{0}\|_{\mathring{L}_{\eta}^{\infty }}+\|D^2 v_{0}\|_{\mathring{L}_{\eta}^{\infty }}, \quad 
\|v_{0}e^{\eta\cdot }\|_{\mathring{L}^{\infty }}=\|v_{0}\|_{\mathring{L}_{\eta}^{\infty }},
\end{equation*}%
we have the estimate%
\begin{equation}
\Bigl\|e^{\boldsymbol{L}_{(\nu )}t}v_{0}\Bigr\|_{\mathring{W}^{2,\infty }}\leq \frac{M%
}{1+t}\|v_{0}\|_{\mathring{W}_{\eta}^{2,\infty }},\text{ }%
t\in \lbrack 0,\infty ).  \label{estim sg L(0) W2}
\end{equation}%
Moreover, for $f_{0}\in \mathring{W}_{\eta}^{1,\infty }$, we have%
\begin{equation}
\Bigl\|e^{\boldsymbol{L}_{(\nu )}t}f_{0}\Bigr\|_{\mathring{W}^{2,\infty }}\leq \frac{M%
}{\sqrt{t} ( 1 + \sqrt{t})}\|f_{0}\|_{\mathring{W}_{\eta}^{1,\infty }},\text{ 
}t>0.  \label{estim sg L(0) W1}
\end{equation}
\end{lemma}

Note that we have $\boldsymbol{L}_{(\nu )}^{(0)}+\boldsymbol{L}^{(1)}=\boldsymbol{L}%
_{(\nu )}^{(0)}$ on $\mathring{L}^{\infty }$ (since $\boldsymbol{L}^{(1)}$
cancels).
Note also that we can choose $\gamma > 0$ with $0 < \gamma < \nu \alpha^2 / 2$ such that
all the eigenvalues of the linearized operator $\boldsymbol{L}_{(\nu )}$ acting in $\mathring{L}%
_{\eta}^{2}$ are such that $\Re \lambda <-2\gamma$.

The proof, which relies on the study of Dunford's formula on the contour $\Gamma $
 described in Figure \ref{fig:gamma}, is detailed in Appendix \ref{app: lem
semigr poids}.

We observe on (\ref{estim sg L(0) W2}) that despite that the spectrum
contains the full negative real axis, we have a decay at infinity in $t$
as soon as the semi-group operates on functions decaying to 0 exponentially
as $y$ goes to $\infty $, which is better than (\ref{est semigr L(0)}) in $%
\mathcal{L}(\mathring{L}^{\infty }).$ However we loose the decay in $y$ for $%
e^{\boldsymbol{L}_{(\nu )}t}v_{0}.$ 

\begin{figure}[th]
\begin{center}
\includegraphics[width=5cm]{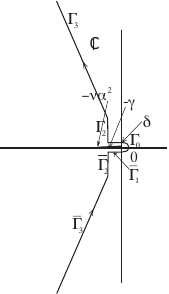}
\end{center}
\caption{Contour $\Gamma $ for the estimate of the semi-group used in Lemma \ref{lem:estimsemigroup poids}.}
\label{fig:gamma}
\end{figure}

%%%%%%%%%%%%%%%%%%%%%%%%%%%

\subsubsection{Study of the semi-group $e^{\boldsymbol{L}_{(\protect\nu )}t}$
in $\mathring{L}_{\eta}^{2}$}

%%%%%%%%%%%%%%%%%%%%%%%%%%%

The following Lemma, which can be found in \cite{Io72} chap VII, uses the
holomorphy of the semi-group $e^{\boldsymbol{L}_{(\nu )}t}:$

\begin{lemma}
\label{lem:basic Lem}If there exists $M>0$ such that the following
estimate holds%
\begin{equation}
\Bigl\|(\mathbb{I}-\varepsilon \boldsymbol{L}_{(\nu )}^{(0)})^{-1} \Bigr\|_{\mathcal{L}(%
\mathring{H}_{\eta}^{1},\mathring{H}_{\eta}^{2})}\leq M\varepsilon ^{-1/2},\text{ }%
\varepsilon \in (0,\delta _{0})\text{, }\delta _{0}>0,  \label{estimperturb}
\end{equation}%
then there exists $C>0$ such that%
\begin{equation}
\Bigl\|e^{\boldsymbol{L}_{(\nu )}^{(0)}t} \Bigr\|_{\mathcal{L}(\mathring{H}_{\eta}^{1},%
\mathring{H}_{\eta}^{2})}\leq \frac{C}{t^{1/2}},\text{ }t\in (0,1].
\label{estim semigroup hilb}
\end{equation}
\end{lemma}

In the Appendix \ref{proof6}, we will prove the following Lemma.

\begin{lemma}
\label{lem perturb hilb}There exists $M>0$ such that the estimate (\ref%
{estimperturb})
holds  in $\mathcal{L}(\mathring{H}_{\eta}^{1};\mathring{H}_{\eta}^{2})$.
\end{lemma}

Let us now consider the linear operator $\boldsymbol{L}_{(\nu )}=\boldsymbol{%
L}_{(\nu )}^{(0)}+\boldsymbol{L}^{(1)}$ acting in $\mathring{L}_{\eta}^{2}$. We
have%
\begin{equation*}
\mathbb{I}-\varepsilon \boldsymbol{L}_{(\nu )}=(\mathbb{I}-\varepsilon 
\boldsymbol{L}_{(\nu )}^{(0)}) \Bigl(\mathbb{I}-\varepsilon (\mathbb{I}%
-\varepsilon \boldsymbol{L}_{(\nu )}^{(0)})^{-1}\boldsymbol{L}^{(1)} \Bigr),
\end{equation*}%
and we note that 
\begin{eqnarray*}
\Bigl\|\varepsilon (\mathbb{I}-\varepsilon \boldsymbol{L}_{(\nu )}^{(0)})^{-1}%
\boldsymbol{L}^{(1)}\Bigr\|_{\mathcal{L}(\mathring{H}_{\eta}^{2})} \leq
\Bigl\|\varepsilon (\mathbb{I}-\varepsilon \boldsymbol{L}_{(\nu
)}^{(0)})^{-1}\Bigr\|_{\mathcal{L}(\mathring{H}_{\eta}^{1},\mathring{H}_{\eta}^{2})}\|%
\boldsymbol{L}^{(1)}\|_{\mathcal{L}(\mathring{H}_{\eta}^{2},\mathring{H}%
_{\eta}^{1})} 
\leq M\varepsilon ^{1/2},
\end{eqnarray*}%
hence, for $\varepsilon $ small enough, $M\varepsilon ^{1/2}<1/2$ and the
operator $\mathbb{I}-\varepsilon (\mathbb{I}-\varepsilon \boldsymbol{L}%
_{(\nu )}^{(0)})^{-1}\boldsymbol{L}^{(1)}$ has a bounded (by $2$) inverse
in $\mathcal{L}(\mathring{H}_{\eta}^{2}).$ Therefore,%
\begin{eqnarray*}
\Bigl\|(\mathbb{I}-\varepsilon \boldsymbol{L}_{(\nu )})^{-1}\Bigr\|_{\mathcal{L}(%
\mathring{H}_{\eta}^{1},\mathring{H}_{\eta}^{2})} \leq 2 \Bigl\|(\mathbb{I}%
-\varepsilon \boldsymbol{L}_{(\nu )}^{(0)})^{-1} \Bigr\|_{\mathcal{L}(\mathring{H}%
_{\eta}^{1},\mathring{H}_{\eta}^{2})} 
\leq \frac{2M}{\varepsilon ^{1/2}}.
\end{eqnarray*}%
 Applying Lemma \ref{lem:basic Lem}  to $\boldsymbol{L}_{(\nu )},$
 we have the following Lemma

\begin{lemma}
Assume that the eigenvalues $\lambda$ of $\boldsymbol{L}_{(\nu )}$ are such that $%
\Re \lambda <\gamma <0,$ then, for $\widetilde{u}\in \mathring{H}%
_{\eta}^{2},$ we have%
\begin{equation*}
\Bigl\|e^{\boldsymbol{L}_{(\nu )}t}\widetilde{u}\Bigr\|_{\mathring{H}_{\eta}^{2}}\leq
Me^{-\gamma t} \, \|\widetilde{u}\|_{\mathring{H}_{\eta}^{2}},\text{ }t\geq 0.
\end{equation*}%
Moreover, we have the estimate%
\begin{equation*}
\Bigl\| e^{\boldsymbol{L}_{(\nu )}t} \widetilde{u}\Bigr\|_{\mathcal{L}(\mathring{H}_{\eta}^{1},\mathring{%
H}_{\eta}^{2})}\leq M \Bigl( 1+\frac{1}{t^{1/2}} \Bigr)e^{-\gamma t}
\text{ for }t>0.
\end{equation*}
\end{lemma}

\begin{proof}
This results from the fact that the set of (isolated) eigenvalues of $%
\boldsymbol{L}_{(\nu)}$ is located as indicated in Lemma \ref{lem: locationspect},
Moreover, we have $\Re \lambda <\gamma$ for $\nu>\nu_0$, by
definition of $\nu _{0}(\alpha )$.  The estimate  (\ref{estim semigroup hilb}), which is valid on $(0,1)$,
can then be extended to $(0,\infty)$ using the $\mathring{H}_{\eta}^{2}$ estimate.
\end{proof}

%%%%%%%%%%%%%%%%%%%%%%%%%%%%%%%%%%%%%%%%%

\subsection{Study of the quadratic term}

%%%%%%%%%%%%%%%%%%%%%%%%%%%%%%%%%%%%%%%%

Let us define the projections $\widetilde{P}$ and $P_{0}$,
which separate the 0-Fourier mode from the oscillating part, by: for any $u\in 
\mathcal{X}_{\eta}$%
\begin{eqnarray*}
u &=&\widetilde{P}u+P_{0}u, \\
\widetilde{P}u &=&\widetilde{u}\in \mathring{L}_{\eta}^{2},\text{ }%
P_{0}u=u_{0}\in \mathring{L}^{\infty }.
\end{eqnarray*}%
The above projections naturally work  in $\mathcal{Z}_{\eta}$ or $\mathcal{Y}_{\eta}.$ 
The aim of this section is to show that the quadratic operator 
\begin{equation*}
\boldsymbol{B}(u,v)=-\frac{1}{2}\Pi  \Bigl[ (u\cdot \nabla )v+(v\cdot \nabla )u \Bigr]
\end{equation*}%
is well defined in $\mathcal{Y}_{\eta}$ for $u$ and $v$ in $\mathcal{Z}_{\eta}$.

\begin{lemma}
The quadratic operator $\boldsymbol{B}$ is bounded from $\mathcal{Z}_{\eta}$ to 
$\mathcal{Y}_{\eta}:$ there exists $C>0$ such that%
\begin{equation}
\|\boldsymbol{B}(u,v)\|_{\mathcal{Y}_{\eta}}\leq C\|u\|_{\mathcal{Z}_{\eta}}\|v\|_{%
\mathcal{Z}_{\eta}}.  \label{bound B}
\end{equation}
\end{lemma}

\begin{proof} Let us decompose $u,$ $v$ in $\mathcal{Z}_{\eta}$%
\begin{equation*}
u=\widetilde{u}+u_{0},\text{ }v=\widetilde{v}+v_{0}
\end{equation*}%
then%
\begin{eqnarray}
\widetilde{P}B(u,v) &=&\widetilde{P}B(\widetilde{u},\widetilde{v})+B(%
\widetilde{u},v_{0})+B(\widetilde{v},u_{0}),  \notag \\
P_{0}B(u,v) &=&P_{0}B(\widetilde{u},\widetilde{v}).
\label{quadratic identity}
\end{eqnarray}%
Indeed, 
\begin{equation*}
(u_{0}\cdot \nabla )v_{0}=u_{0}^{y}Dv_{0}=0\text{ }
\end{equation*}%
since $u_{0}^{y}=0$ for $u_{0}\in \mathring{L}^{\infty }.$ This implies that 
$P_{0}B(u_{0},v_{0})=0.$

Now, we observe that $(\widetilde{u}\cdot \nabla )v_{0}$ and $(u_{0}\cdot
\nabla )\widetilde{v}$ satisfy%
\begin{equation*}
\Bigl\|(\widetilde{u}\cdot \nabla )v_{0} \Bigr\|_{L_{\eta}^{2}}^{2}=\sum_{|n|\geq 1}\|%
\widetilde{u}_{n}^{y}Dv_{0}^{x}\|_{C_{\eta}^{0}}^{2}\leq \|v_{0}\|_{W^{1,\infty
}}^{2}\|\widetilde{u}\|_{L_{\eta}^{2}}^{2},
\end{equation*}%
\begin{eqnarray*}
\Bigl\|(\widetilde{u}\cdot \nabla )v_{0} \Bigr\|_{H_{\eta}^{1}}^{2} &=&\sum_{|n|\geq
1}n^{2}\|\widetilde{u}_{n}^{y}Dv_{0}^{x}\|_{C_{\eta}^{0}}^{2}+\|D(\widetilde{u}%
_{n}^{y}Dv_{0}^{x})\|_{C_{\eta}^{0}}^{2} \\
&\leq &\|v_{0}\|_{W^{1,\infty }}^{2}\|\widetilde{u}%
\|_{H_{\eta}^{1}}^{2}+\|v_{0}\|_{W^{2,\infty }}^{2}\|\widetilde{u}%
\|_{L_{\eta}^{2}}^{2},
\end{eqnarray*}%
and analogously for $(u_{0}\cdot \nabla )\widetilde{v}.$ Hence, as $\Pi $
is bounded from $H_{\eta}^{1}$ to $\mathring{H}_{\eta}^{1},$ we obtain easily%
\begin{eqnarray*}
\Bigl\|B(\widetilde{u},v_{0})+B(\widetilde{v},u_{0})\Bigr\|_{\mathring{H}_{\eta}^{1}}
&\leq &C \Bigl(\|\widetilde{u}\|_{H_{\eta}^{1}}\|v_{0}\|_{W^{2,\infty }}+\|\widetilde{%
v}\|_{H_{\eta}^{1}}\|u_{0}\|_{W^{2,\infty }} \Bigr) \\
&\leq &C\|u\|_{\mathcal{Z}_{\eta}}\|v\|_{\mathcal{Z}_{\eta}}.
\end{eqnarray*}%
Now, before considering $(\widetilde{u}\cdot \nabla )\widetilde{v}$, let us
first consider the product of a scalar function $f$ in $H_{\eta}^{1}$ with a
scalar $g$ function in $L_{\eta}^{2},$ both with 0-average$.$ We show that the
product $fg$ is in $L_{\eta}^{2}.$ Indeed, the Fourier coefficients satisfy
\begin{equation*}
(fg)_{n}=\sum_{p,q=n-p}f_{p}g_{q}
\end{equation*}%
and%
\begin{equation*}
\|(fg)_{n}\|_{C_{\eta}^{0}}\leq \sum_{p\neq 0,q=n-p\neq 0}\frac{1}{p}%
\|pf_{p}\|_{C_{\eta}^{0}}\|g_{q}\|_{C_{\eta}^{0}},
\end{equation*}%
so that%
\begin{equation*}
\|(fg)_{n}\|_{C_{\eta}^{0}}^{2}\leq \left( \sum_{p^{\prime }\neq 0}\frac{1}{%
p^{\prime 2}}\right)
\sum_{p,q=n-p}\|pf_{p}\|_{C_{\eta}^{0}}^{2}\|g_{q}\|_{C_{\eta}^{0}}^{2}.
\end{equation*}%
Hence%
\begin{equation*}
\|fg\|_{L_{\eta}^{2}}^{2}\leq
C\sum_{n,p,q=n-p}\|pf_{p}\|_{C_{\eta}^{0}}^{2}\|g_{q}\|_{C_{\eta}^{0}}^{2}\leq
C\|f\|_{H_{\eta}^{1}}^{2}\|g\|_{L_{\eta}^{2}}^{2}.
\end{equation*}%
Coming back to $(\widetilde{u}\cdot \nabla )\widetilde{v}$ , with $%
\widetilde{u}$ and $\widetilde{v}$ in $H_{\eta}^{2},$ we deduce immediately
that $\widetilde{P}(\widetilde{u}\cdot \nabla )\widetilde{v}\in H_{\eta}^{1},$
and the bound (\ref{bound B}) for $\widetilde{P}B(\widetilde{u},\widetilde{v}%
)$ holds immediately in $\mathring{H}_{\eta}^{1}.$

For the 0-Fourier mode, we have also%
\begin{equation*}
\|(fg)_{0}\|_{L_{\eta}^{\infty }}^{2}\leq \left( \sum_{p^{\prime }\neq 0}\frac{1%
}{p^{\prime 2}}\right)
\sum_{p}\|pf_{p}\|_{C_{\eta}^{0}}^{2}\|g_{-p}\|_{C_{\eta}^{0}}^{2}\leq
C\|f\|_{H_{\eta}^{1}}^{2}\|g\|_{L_{\eta}^{2}}^{2},
\end{equation*}%
so that $P_{0}(\widetilde{u}\cdot \nabla )\widetilde{v}$ satisfies easily%
\begin{equation*}
\|P_{0}(\widetilde{u}\cdot \nabla )\widetilde{v}\|_{W_{\eta}^{1,\infty }}\leq
C\|\widetilde{u}\|_{H_{\eta}^{2}}\|\widetilde{v}\|_{H_{\eta}^{2}},
\end{equation*}%
and estimate (\ref{bound B}) holds.
\end{proof}

%%%%%%%%%%%%%%%%%%%%%%%%%%%%%%%%%%%%%%%%%%%%%%%%%%%%%%%%%%%%%%%%%%%%%%%%%%

\section{Nonlinear stability of $U$}

%%%%%%%%%%%%%%%%%%%%%%%%%%%%%%%%%%%%%%%%%%%%%%%%%%%%%%%%%%%%%%%%%%%%%%%%%%

In this section, we consider the nonlinear evolution problem with an initial
data close to the basic flow $U,$  assuming that the spectrum of the
linearized operator $\boldsymbol{L}_{(\nu )}$ is situated on the left side
of the imaginary axis, except the essential spectrum which contains the full
negative real line. We prove the nonlinear stability of $U.$

Assume $\xi >0$ and let us define the Banach space $\mathcal{Z}_{\eta,\xi}$ by%
\begin{equation*}
\mathcal{Z}_{\eta,\xi }=\Bigl\{ v=\widetilde{v}+v_{0}; \, \widetilde{v}\in C^{0}(%
\mathbb{R}^{+},\mathring{H}_{\eta}^{2}), \, v_{0}\in C^{0}(\mathbb{R}^{+},\mathring{%
W}^{2,\infty }),
\end{equation*}%
\begin{equation*}
 |||v|||_{\eta,\xi  }=\underset{t\geq 0}{\sup } \Bigl( \|\widetilde{v}(t)e^{\xi t}\|_{\mathring{H}_{\eta}^{2}}
+ \|(1+t)v_{0}(t) \|_{\mathring{W}^{2,\infty}} \Bigr)<\infty \Bigr\}.
\end{equation*}%
Note the exponential weight in time for the non zero Fourier component and the $(1+t)$ weight
in time for the zero Fourier component.
We now detail the stability part of the Theorem \ref{maintheo} of the introduction.

\begin{theorem} \label{maintheo1}
\label{Lem:stabnonlin}
Let us assume $\eta > 0$ and assume that the set of
eigenvalues $\lambda $ of operator $\boldsymbol{L}_{(\nu )}$ satisfies%
\begin{equation*}
\Re \lambda <-\xi  <0,
\end{equation*}%
then there exists $\varepsilon >0$ such that for 
\begin{equation*}
\|\widetilde{P}v(0)\|_{\mathring{H}_{\eta}^{2}} + \|P_{0}v(0)\|_{\mathring{W}_{\eta}^{2}}\leq \varepsilon ,
\end{equation*}
there exists a unique solution $v\in \mathcal{Z}_{\eta,\xi }$ of the following
differential equation in $\mathcal{X}_{\eta}$%
\begin{equation}
\frac{dv}{dt}=\boldsymbol{L}_{(\nu )}v+B(v,v),  \label{equ diff basic}
\end{equation}%
with $v_{t=0}=v(0)=\widetilde{P}v(0)+P_{0}v(0).$ Moreover there exists $M>0$
such that%
\begin{eqnarray*}
|||\widetilde{P}v|||_{\eta,\xi  } &\leq &M\|\widetilde{P}v(0)\|_{\mathring{H}%
_{\eta}^{2}}, \\
|||P_{0}v|||_{\eta,\xi  } &\leq &M \Bigl( \|P_{0}v(0)\|_{\mathring{W}_{\eta}^{2}}+\|%
\widetilde{P}v(0)\|_{\mathring{H}_{\eta}^{2}}^{2} \Bigr).
\end{eqnarray*}
\end{theorem}

\begin{remark}
Note that this Theorem implies that 
\begin{eqnarray*}
\|\widetilde{P}v(t)\|_{\mathring{H}_{\eta}^{2}} &\leq &Me^{-\xi t}\|%
\widetilde{P}v(0)\|_{\mathring{H}_{\eta}^{2}},\text{ }t\geq 0, \\
\|P_{0}v(t)\|_{\mathring{W}^{2,\infty }} &\leq &\frac{M}{1+t}%
\Bigl( \|P_{0}v(0)\|_{\mathring{W}_{\eta}^{2}}+\|\widetilde{P}v(0)\|_{\mathring{H}%
_{\eta}^{2}}^{2} \Bigr),\text{ }t\geq 0.
\end{eqnarray*}%
The slow decrease of the $0$-mode as $t\rightarrow \infty $ is due to the
influence of the essential spectrum of $\boldsymbol{L}_{(\nu )}$ which contains
the full negative real line. Moreover, we may notice that the decay in $y$
at $\infty $ is lost for the $0$-mode for $t>0.$
\end{remark}

\begin{proof}
Using (\ref{quadratic identity}) and (\ref{bound B}) shows that there exists 
$C>0$ such that we have the estimates%
\begin{eqnarray*}
\Bigl\|\widetilde{P}B(\widetilde{v},\widetilde{v})+2B(\widetilde{v},v_{0})\Bigr\|_{%
\mathring{H}_{\eta}^{1}} &\leq &C \Bigl(\|\widetilde{v}\|_{\mathring{H}%
_{\eta}^{2}}^{2}+\|\widetilde{v}\|_{\mathring{H}_{\eta}^{2}}\|v_{0}\|_{\mathring{W}%
^{2,\infty }} \Bigr), \\
\|P_{0}B(\widetilde{v},\widetilde{v})\|_{\mathring{W}_{\eta}^{1,\infty }} &\leq
&C\|\widetilde{v}\|_{\mathring{H}_{\eta}^{2}}^{2}.
\end{eqnarray*}
Now (\ref{equ diff basic}) becomes%
\begin{eqnarray*}
\frac{d\widetilde{v}}{dt} &=&\boldsymbol{L}_{(\nu )}\widetilde{v}+\widetilde{%
P}B(\widetilde{v},\widetilde{v})+2B(\widetilde{v},v_{0}), \\
\frac{dv_{0}}{dt} &=&\boldsymbol{L}_{(\nu )}v_{0}+P_{0}B(\widetilde{v},%
\widetilde{v}),
\end{eqnarray*}%
which leads to the following integral formulation for $t\geq 0$%
\begin{eqnarray}
\widetilde{v}(t) &=&e^{\boldsymbol{L}_{(\nu )}t}\widetilde{P}%
v(0)+\int_{0}^{t}e^{\boldsymbol{L}_{(\nu )}(t-s)} \Bigl[\widetilde{P}B(\widetilde{v%
},\widetilde{v})+2B(\widetilde{v},v_{0}) \Bigr](s)ds,  \label{compvtilde} \\
v_{0}(t) &=&e^{\boldsymbol{L}_{(\nu )}t}P_{0}v(0)+\int_{0}^{t}e^{\boldsymbol{%
L}_{(\nu )}(t-s)}P_{0}B(\widetilde{v},\widetilde{v})(s)ds,  \label{compv0}
\end{eqnarray}%
where we look for $\widetilde{v}+v_{0}\in \mathcal{Z}_{\eta,\xi }.$ 
We can solve the system above by the implicit function theorem with respect
to $(\widetilde{v},v_{0})$ in the neighborhood of $0$ for $v(0)=\widetilde{P}%
v(0)+P_{0}v(0)$ provided 
\begin{eqnarray*}
&&\widetilde{P}v(0)\text{ sufficiently small in }\mathring{H}_{\eta}^{2}, \\
&&P_{0}v(0)\text{ sufficiently small in }\mathring{W}_{\eta}^{2}.
\end{eqnarray*}
Using the bound of the imaginary parts of eigenvalues which are isolated in
the left half complex plane, there exist $\xi_{m}>0$
such that all the eigenvalues $\lambda $ of $\boldsymbol{L}_{(\nu )}$ satisfy%
\begin{equation*}
\sup \Re \lambda =-\xi_{m} <-\xi  <0.
\end{equation*}%
We choose $\xi_1$ such that $- \xi_m <  - \xi_1 < -  \xi$.
Then we have the estimates%
\begin{eqnarray*}
\|e^{\boldsymbol{L}_{(\nu )}t}\|_{\mathcal{L}(\mathring{H}_{\eta}^{2})} &\leq
&Ce^{-\xi _{1}t},\text{ }t\geq 0, \\
\|e^{\boldsymbol{L}_{(\nu )}t}\|_{\mathcal{L}(\mathring{W}_{\eta}^{2,\infty },%
\mathring{W}^{2,\infty })} &\leq &\frac{C}{1+t},\text{ }t\geq 0, \\
\|e^{\boldsymbol{L}_{(\nu )}(t-s)}\|_{\mathcal{L}(\mathring{H}_{\eta}^{1},%
\mathring{H}_{\eta}^{2})} &\leq &C \Bigl[1+\frac{1}{\sqrt{t-s}} \Bigr] e^{-\xi_{1}(t-s)},\text{ }t>s, \\
\|e^{\boldsymbol{L}_{(\nu )}(t-s)}\|_{\mathcal{L}(\mathring{W}_{\eta}^{1,\infty
},\mathring{W}^{2,\infty })} &\leq &\frac{C}{\sqrt{t-s}(1 + \sqrt{t-s})},\text{ }t>s.
\end{eqnarray*}%
We note that%
\begin{eqnarray*}
\|e^{\xi t}e^{\boldsymbol{L}_{(\nu )}t}\widetilde{P}v(0)\|_{\mathring{H}%
_{\eta}^{2}} &\leq &C\|\widetilde{P}v(0)\|_{\mathring{H}_{\eta}^{2}}, \\
\|e^{\boldsymbol{L}_{(\nu )}t}P_{0}v(0)\|_{\mathring{W}^{2,\infty }} &\leq &%
\frac{C}{1+t}\|P_{0}v(0)\|_{\mathring{W}_{\eta}^{2}},
\end{eqnarray*}%
that
\begin{eqnarray*}
&&\Bigl\|e^{\xi t}e^{\boldsymbol{L}_{(\nu )}(t-s)} \Bigl[\widetilde{P}B(\widetilde{v}%
,\widetilde{v})+2B(\widetilde{v},v_{0})\Bigr](s)\Bigr\|_{\mathring{H}_{\eta}^{2}} \\
&\leq &MC \Bigl(1+\frac{1}{\sqrt{t-s}} \Bigr)e^{-(\xi _{1}-\xi )(t-s)}|||%
\widetilde{v}|||_{\eta,\xi }|||\widetilde{v}+v_{0}|||_{\eta,\xi },
\end{eqnarray*}%
and that 
\begin{equation*}
\Bigl\|e^{\boldsymbol{L}_{(\nu )}(t-s)}P_{0}B(\widetilde{v},\widetilde{v})(s) \Bigr\|_{%
\mathring{W}^{2,\infty }}\leq \frac{MC}{\sqrt{t-s} (1 + \sqrt{t-s})}e^{-2\xi s}|||%
\widetilde{v}|||_{\eta,\xi }^{2}.
\end{equation*}%
Hence the right hand side of (\ref{compvtilde}), and (\ref{compv0}) is
quadratic and $C^{1}$ in $(\widetilde{v},v_{0})\in \mathcal{Z}_{\eta,\xi }.$
It results by the implicit function theorem, that, provided $(\widetilde{P}%
v(0),P_{0}v(0))$ small enough in $\mathring{H}_{\eta}^{2}\oplus \mathring{W}%
_{\eta}^{2},$ there exists a unique solution $(\widetilde{v},v_{0})\in \mathcal{%
Z}_{\eta,\xi}$ in a neighborhood of $0.$ Moreover, directly from the system
(\ref{compvtilde}), (\ref{compv0}), we have the estimate%
\begin{align*}
\|e^{\xi  t}\widetilde{v}(t)\|_{\mathring{H}_{\eta}^{2}}
\leq C\|\widetilde{P}%
v(0)\|_{\mathring{H}_{\eta}^{2}}+MC\int_{0}^{t} &(1+\frac{1}{(t-s)^{1/2}}%
)
\\ 
& \times e^{-(\xi _{1}-\xi  )(t-s)}|||\widetilde{v}|||_{\eta,\xi  }|||\widetilde{%
v}+v_{0}|||_{\eta,\xi } ds,
\end{align*}%
hence%
\begin{equation*}
|||\widetilde{v}|||_{\eta,\xi }\leq C\|\widetilde{P}v(0)\|_{\mathring{H}%
_{\eta}^{2}}+K|||\widetilde{v}|||_{\eta,\xi }|||\widetilde{v}%
+v_{0}|||_{\eta,\xi  },
\end{equation*}%
using%
\begin{equation*}
\int_{0}^{t} \Bigl[ 1+\frac{1}{(t-s)^{1/2}} \Bigr] e^{-\delta (t-s)}ds\leq C_{1}(\delta )%
\text{ for }\delta >0.
\end{equation*}%
In the same way, we obtain%
\begin{equation*}
|||v_{0}|||_{\eta,\xi }\leq C\|P_{0}v(0)\|_{\mathring{W}_{\eta}^{2,\infty
}}+K|||\widetilde{v}|||_{\eta,\xi }^{2},
\end{equation*}%
using%
\begin{equation*}
\int_{0}^{t}\frac{1}{\sqrt{t-s}) (1+\sqrt{t-s})}e^{-\delta s}ds\leq \frac{C_{2}(\delta )}{1+t},\text{ for }\delta >0.
\end{equation*}%
Hence, 
\begin{equation*}
|||\widetilde{v}|||_{\eta,\xi }\leq C\|\widetilde{P}v(0)\|_{\mathring{H}%
_{\eta}^{2}}+CK\|P_{0}v(0)\|_{\mathring{W}_{\eta}^{2,\infty }}|||\widetilde{v}%
|||_{\eta,\xi }+K|||\widetilde{v}|||_{\eta,\xi }^{2}+K^{2}|||\widetilde{v}%
|||_{\eta,\xi }^{3},
\end{equation*}%
so that for $v(0)$ such that%
\begin{equation*}
CK\|P_{0}v(0)\|_{\mathring{W}_{\eta}^{2,\infty }}+2CK\|\widetilde{P}v(0)\|_{%
\mathring{H}_{\eta}^{2}}+4C^{2}K^{2}\|\widetilde{P}v(0)\|_{\mathring{H}%
_{\eta}^{2}}^{2}< \frac{1}{2},
\end{equation*}%
we obtain%
\begin{eqnarray*}
|||\widetilde{v}|||_{\eta,\xi } &\leq &2C\|\widetilde{P}v(0)\|_{\mathring{H}%
_{\eta}^{2}}, \\
|||v_{0}|||_{\eta,\xi } &\leq &C\|P_{0}v(0)\|_{\mathring{W}_{\eta}^{2,\infty
}}+4C^{2}K\|\widetilde{P}v(0)\|_{\mathring{H}_{\eta}^{2}}^{2}
\end{eqnarray*}%
which ends the proof of the nonlinear asymptotic stability.
\end{proof}

%%%%%%%%%%%%%%%%%%%%%%%%%%%%%%%%%%%%%%%%%%%%%%%%%%%

\section{Study of the bifurcation}

%%%%%%%%%%%%%%%%%%%%%%%%%%%%%%%%%%%%%%%%%%%%%%%%%%%

%%%%%%%%%%%%%%%%%%%%%%%%%%%%%%%%%%%

\subsection{Setup}

%%%%%%%%%%%%%%%%%%%%%%%%%%%%%%%%%%%

Let us define
$$
s = \omega_0 t
$$
such that the time-periodic solution which we are looking for is now $2\pi$ periodic in $s$. 
The wave number $\omega_0$ is an unknown of the bifurcation, which must be determined.

Let us define the bifurcation parameter $\mu$ as
$$
\mu = \nu - \nu_0
$$
By assumption, ${\bf L} := {\bf L}_{(\nu_0)}$ has two purely imaginary simple eigenvalues $\pm i \omega_0$,
associated to eigenvectors of the form
\begin{equation} \label{formzeta}
\zeta = e^{i \alpha x} \hat v(y), 
\quad 
\overline{\zeta} = e^{- i \alpha x} \overline{ \hat v(y)}
\end{equation}
and no other eigenvalues of non negative real part.

In what follows, we need to invert the linear system
$$
\omega_0 {dv \over ds} - {\bf L} \, v = f
$$
in a space of vector functions which are $2 \pi$ periodic in $s$ and 
$2 \pi / \alpha$ periodic in $x$.
We expand $f$ and $v$ in Fourier series in time and space, $k$ being the Fourier wave number in time and $n$ the Fourier wave number in space. This leads to the sequence of equations
$$
i k \omega_0 v_{k,n} - {\bf L}_n v_{k,n} = f_{k,n}
$$
where $(k,n) \in \mathbb{Z}^2$, and where the linear operator ${\bf L}_n$ is defined
in a space of vector functions of $y$ by
\begin{align*}
{\bf L}_n v_n &= \nu_0 (D^2 - n^2 \alpha^2) v_n 
+ \Bigl( \begin{array}{c} i \alpha n \cr D \cr \end{array} \Bigr) q_n
- i n \alpha U v_n 
- \Bigl( \begin{array}{c} v_n^y U' \cr 0 \cr \end{array} \Bigr),
\\
0 &= i n \alpha v_n^x + D v_n^y,
\end{align*}
where $q_n$ is the related component of the pressure.
We know by assumption that 
\begin{itemize}
\item ${\bf L}_n - i k \omega_0$ is invertible for  $n \ne 0$ and  $k \ne \pm 1$,

\item for  $n = 1$, ${\bf L}_n - i \omega_0$ has a 1-dim kernel spanned
by $\zeta$, 

\item for $n = -1$, ${\bf L}_n + i \omega_0$ has a 1-dim kernel spanned by 
$\bar \zeta$,

\item for  $n \ne \pm 1$, ${\bf L}_n \mp i \omega_0$ is invertible,

\item for $k \ne 0$, ${\bf L}_0 - i k \omega_0$ is invertible.
\end{itemize}
It remains to study the invertibility of ${\bf L}_0$.

%%%%%%%%%%%%%%%%%%%%%%

\subsection{Study of the inverse of ${\bf L}_0$}

%%%%%%%%%%%%%%%%%%%%%%

\begin{lemma}\label{lemma 20}
If $f_{0,0} \in L^{\infty}_\eta$, then the equation
\begin{equation} \label{L0}
{\bf L}_0 v_{0,0} = \Bigl( \begin{array}{c} f_{0,0} \cr 0 \cr \end{array} \Bigr)
\end{equation}
has a unique solution such that $v_{0,0}^x$ is bounded.
Moreover, 
$$
\| v_{0,0} \|_{W^{2,\infty}}  \le C_\eta \| f_{0,0} \|_{L^{\infty}_\eta}
$$
for some positive constant $C_\eta$.
\end{lemma}

\begin{proof}
The equation (\ref{L0}) leads to
\begin{align*}
\nu_0 D^2 v_{0,0} = f_{0,0}.
\end{align*}
Taking care of the boundary condition $v_{0,0}^x = 0$ at $y = 0$ and its boundedness at infinity, we find
\begin{align*}
D v_{0,0}^x &= {1 \over \nu_0} \Bigl[ - \int_y^\infty f_{0,0}(s) \, ds + A \Bigr],
\\
v_{0,0}^x &= {1 \over \nu_0} \int_0^y d\tau \Bigl[ - \int_\tau^\infty f_{0,0}(s) \, ds
+ A \Bigr]
\end{align*}
for some constant $A$ to be determined so that $v_{0,0}^x$ is bounded
and $D v_{0,0}^x$ tends exponentially to $0$ at infinity. This leads
to $A = 0$ and to
$$
v_{0,0}^x(y) = - {1 \over \nu_0} \int_0^y d\tau \int_\tau^\infty f_{0,0}(s) \, ds ,
$$
and, by Fubini's theorem,
$$
v_{0,0}^x(y) = - {1 \over \nu_0} \Bigl[ \int_0^y s f_{0,0}(s) \, ds + y \int_y^\infty f_{0,0}(s) \, ds \Bigr],
$$
which ends the proof of the Lemma.
\end{proof}

We note that 
$$
\lim_{y \to + \infty} v_{0,0}^x(y) = - {1 \over \nu_0} \int_0^{+ \infty} s f_{0,0}(s) \, ds,
$$
which in general does not vanish. Thus, in general, $v_{0,0}^x$ does not go to $0$ at infinity, 
but $D v_{0,0}^x$ and $D^2 v_{0,0}^x$ decay exponentially fast at infinity. 
However $v_{0,0}$ only appears in the operator $(v \cdot \nabla) v$,
and thus only in terms of the form $(v_{0,0} \cdot \nabla) v$  and $(v \cdot \nabla) v_{0,0}$ where $v$ will be exponentially decaying.

\subsection{The eigenvectors $\zeta$ and $\zeta^\star$}

%%%%%%%%%%%%%%%%%%%%%%%%%%%%%%%%%%%%%%%%

We look for 
$\zeta = e^{i \alpha x} \hat v(y) \in \mathcal{Z}_\eta$
where $\alpha \ne 0$ and 
$\hat v = (\hat v^x, \hat v^y) \in C^2_\eta$
satisfy
\begin{align*}
i \omega_0 \hat v &= \nu_0 (D^2 - \alpha^2) \hat v
+ \Bigl( \begin{array}{c} i \alpha \cr D \cr \end{array} \Bigr) \hat q
- i \alpha U \hat v
- \Bigl( \begin{array}{c} \hat v^y U' \cr 0 \cr \end{array} \Bigr),
\\
0 &= i \alpha \hat v^x + D \hat v^y 
\end{align*}
with $\hat v^x = \hat v^y = 0$ for $y = 0$.
This leads to $\hat v^y \in C^4_\eta$ and
\begin{equation} \label{2}
\Bigl[ i \omega_0 - \nu_0 (D^2 - \alpha^2) - i \alpha U \Bigr]
(D^2 - \alpha^2) \hat v^y + i \alpha U'' \hat v^y = 0
\end{equation}
with $\hat v^y = D \hat v^y = 0$ for $y = 0$, which is the Orr-Sommerfeld equation.

By assumption, (\ref{2}) has a non trivial smooth solution $\hat v^y(y)$, satisfying the boundary conditions, 
unique up to a multiplicative constant. Moreover, there exist $c_1 > 0$ and $\eta > 0$ such that
$$
| \hat v^y(y) | + | D \hat v^y(y)| \le c_1 e^{- \eta y}
$$
for $0 \le y < + \infty$. 
It results that there exists $c > 0$ and $\eta > 0$ such that
$$
| \hat v(y) | + | D \hat v(y) | \le c e^{- \eta y}
$$
for $0 \le y < + \infty$ (see \cite{Guo}).

\subsection{Pseudo-inverse of $\omega_0 d/ds - {\bf L}$}

%%%%%%%%%%%%%%%%%%%%%%%%%%%%%%%%%%%

The aim of this section is to construct a pseudo inverse to the linear operator
$$
\mathcal{L} = \omega_0 {d \over ds} - {\bf L}
$$ 
in a space of $2 \pi$ periodic vector functions (periodic in time and space).

Let us define $\zeta^*$, the eigenvector of ${\bf L}^\star$
for the eigenvalue $- i \omega_0$, such that 
$\langle \zeta, \zeta^\star \rangle = 1$ (see \cite{Kato}).
Let $\mathbb{T}_1 = \mathbb{R} / 2 \pi \mathbb{Z}$.
Let $\mathcal{H}$ be a subspace of $L^2(\mathbb{T}_1, \mathcal{X}_\eta)$ such that, for every $v \in \mathcal{H}$,
$$
\langle v, e^{i s} \zeta^\star \rangle_{\mathcal{H}} =
\langle v, e^{- i s} \overline{\zeta}^\star \rangle_{\mathcal{H}}
= 0
$$
and such that $v_{00} = 0$. 
Let us define
$$
H^\sharp = H^1(\mathbb{T}_1,\mathcal{X}_\eta)
\cap L^2(\mathbb{T}_1, \mathcal{Z}_\eta),
$$
together with its norm
$$
\| v \|_{H^\sharp}^2 = \| \partial_s v \|_{L^2(\mathbb{T}_1,\mathcal{X}_\eta)}^2 + \| v \|_{L^2(\mathbb{T}_1, \mathcal{Z}_\eta)}^2.
$$

\begin{lemma}
For $f \in \mathcal{H}$, there exists a unique $v \in H^\sharp \cap \mathcal{H}$ such that
$$
\Bigl( \omega_0 {d \over ds } - {\bf L} \Bigr) v = f.
$$
This defines a pseudo-inverse $\mathcal{L}^{-1}$ to ${\cal L}$ which satisfies
$$
\| v \|_{H^\sharp} \le \| f \|_{\mathcal{H}} .
$$
\end{lemma}

\begin{proof}
As $f \in \mathcal{H}$, $f$ is of the form
$$
f(x,y,s) = \sum_{(n,k) \in \mathbb{Z}^2} f_{k,n}(y) e^{i (k s + n \alpha x)} 
$$
with
$$
\Bigl\langle f_{1,1} e^{i \alpha x} , \zeta^\star \Bigr\rangle 
= \Bigl\langle f_{-1,-1} e^{-i \alpha x} , \overline{\zeta^\star} \Bigr\rangle
= 0,
\quad f_{0,0} = 0.
$$
Similarly, we look for $v \in \mathcal{H}$, namely for a function $v$ of the form
$$
v(x,y,s) = \sum_{(n,k) \in \mathbb{Z}^2} v_{k,n}(y) e^{i (k s + n \alpha x)} 
$$
with
$$
\Big\langle v_{1,1} e^{i \alpha x} , \zeta^\star \Big\rangle 
= \Big\langle v_{-1,-1} e^{-i \alpha x} , \overline{\zeta^\star} \Big\rangle
= 0,
\quad v_{00} = 0.
$$
This leads to
\begin{equation} \label{bibi}
\Bigl( i k \omega_0 - {\bf L}_n \Bigr) v_{k,n} = f_{k,n}
\end{equation}
for $(k,n) \ne \pm (1,1)$ and $(k,n) \ne (0,0)$.

Using Lemma \ref{lemma11}, we know that there exist $k_0$ and $c$
such that if $|k| > k_0$ is large enough, the equation (\ref{bibi})
has a unique solution such that
%\begin{equation} \label{estimate1}
%\| v \|_{\mathring{L}^2_\eta} \le {c \over |k|} \| f \|_{\mathring{L}^2_\eta}
%\end{equation}
%and
%\begin{equation} \label{estimate2}
%\| v \|_{\mathring{H}^2_\eta} \le c \| f \|_{\mathring{L}^2_\eta}.
%\end{equation}
%Using that there exists $c > 0$ such that for $|k|$ large enough, say for $|k| > k_0$, we have a unique $v_{k,n}$ 
%satisfying
\begin{equation} \label{3}
\| D^2 v_{k,n} \|_{C^0_\eta}
+ |n| \, \| D v_{k,n} \|_{C^0_\eta}
+ (1 + n^2) \| v_{k,n} \|_{C^0_\eta}
\le \| f_{k,n} \|_{C^0_\eta},
\end{equation}
and
$$
\| v_{k,n} \|_{C^0_\eta} \le {c \over |k|}
\| f_{k,n} \|_{C^0_\eta}
$$
provided $(k,n) \ne \pm (1,1),  (0,0)$.

For $n \ne \pm 1$, $n \ne 0$ and $|k| \le k_0$, ${\bf L}_n$ is invertible
by assumption in section $2.1$, thus (\ref{3}) is still valid
(up to the change of the factor $c > 0$).
For $n = \pm 1$, (\ref{3}) is also valid since $f_{k,n}$ is orthogonal to $\zeta^\star$ and
$\overline{\zeta^\star}$, and thus in the range of $i k \omega_0 - {\bf L}_n$.
For $n = 0$, as $f_{0,0} = 0$, $v_{0,0} = 0$, and we just have to consider the case $k \ne 0$, where $i k \omega_0 - {\bf L}_0$
is invertible by Lemma \ref{lemma11}, thus (\ref{3}) is still valid.

Combining all these estimates, we get that the pseudo inverse of $\mathcal{L}$ is bounded from ${\cal H}$ to $H^\sharp$.
\end{proof}

% {\color{red}
% For the $(0,0)$ mode, we need that $f_{00}$ lies in the range of ${\bf L}_0$. 
% We shall use that for $f_{00}$ in $L^2_\tau$ such that there exists $\tau > 0$
% and
% $$
% \int | f(y)|^2 e^{2 \tau y} \, dy = \| f \|_{\tau}^2 < + \infty,
% $$
% then, as shown above, $v_{00}$ is uniquely defined in $H_0^2(\mathbb{R}^+)$
% where
% $$
% \| v_{0,0} \|_{H_\tau^2(\mathbb{R}^+)} = \| D^2 v_{0,0} \|_\tau^2 + \| v_{0,0} \|_\tau^2 
% \le c \| f_{0,0} \|_\tau^2 .
% $$
% Notice that the exponential decay at infinity in $y$ of the eigenvectors $\zeta$
% and $\zeta^\star$ will allow to insure that the right hand side $f_{00}$ built below
% satisfies a good decay at infinity.

% }

%%%%%%%%%%%%%%%%%%%%%%%%%%%%%%%%%%%%%%%%%%%%%%%%%%%%%%%%%

\subsection{Bifurcation of a time periodic solution}

%%%%%%%%%%%%%%%%%%%%%%%%%%%%%%%%%%%%%%%%%%%%%%%%%%%%%%%%%%

We are looking for a solution $v$, which is $2 \pi$ periodic in $s$, 
$2 \pi / \alpha$ periodic in $x$, of the non linear equation
\begin{equation} \label{4}
\Bigl[ \omega {d \over ds} - {\bf L} \Bigr] v = \mu {\bf L}' v
+ B(v,v)
\end{equation}
where
$$
{\bf L}' v = \Pi \Delta v,
$$
$$
B(v,v) = - \Pi (v \cdot \nabla) v.
$$
% and where $\Pi$ is the projection on divergence free vector fields,
% defined by $v = \Pi u$ where
% $u = v + \nabla \phi$, $u$ is $2 \pi / \alpha$ periodic in $x$,
% $u$, $v$ and $\nabla \phi$ are $L^2([(\mathbb{R} / (2 \pi / \alpha) \mathbb{Z})
% \times \mathbb{R}^2])^2$  functions,
% $\nabla \cdot v = 0$, $v^y_{| y = 0} = 0$,
% $\Delta \phi = \nabla \cdot u$ is in $H^{-1}$ and $\partial_y \phi_{| y = 0}$
% is in $H^{-1/2}(\mathbb{R} / (2 \pi / \alpha) \mathbb{Z})$.
Let
%$$
%\mathcal{H}^\sharp = H^1(\mathbb{T}_1, \mathcal{X}_\eta)
%\cap \mathcal{H}
%$$
%and
$$
H^{\sharp\sharp} = H^2(\mathbb{T}_1,\mathcal{X}_\eta)
\cap H^1(\mathbb{T}_1, \mathcal{Z}_\eta).
$$
We decompose $v \in H^{\sharp\sharp}$ as
$$
v = \underline{v} + w,
$$
$$
\underline{v} = A e^{i s } \zeta + \overline{A} e^{- i s } \bar \zeta + v_{0,0},
$$
$$
\langle w_{1,1} e^{i \alpha x}, \zeta^{\star} \rangle
= \langle w_{-1,-1} e^{- i \alpha x}, \overline{\zeta^\star} \rangle = 0,
$$
$$
w_{0,0} = 0,
$$
and define the projection $\frak{P}_0$ by
$$
w = \frak{P}_0 v.
$$
The projection $\frak{P}_0$ is bounded in $H^1(\mathbb{T}_1,{\cal X}_\eta)$
and in $H^{\sharp\sharp}$ and commutes with the linear operator $( \omega \partial_s - {\bf L})$. Later 
we will use the fact that our system is invariant under two different $SO(2)$ symmetries, an hence commutes
 with  the operator $T_a$ representing the translation in time 
$s \to s + a$, and the operator  $\tau_\beta$ representing the translation in space $ x \to x + \beta$.

Let 
$$
\xi = x + {s \over \alpha} .
$$

We now prove the following  bifurcation theorem, which precises Theorem \ref{maintheo} and gives an expansion 
of the bifurcation parameter $\mu$, of the time frequency $\omega$
and of  the time and space periodic solution near the bifurcation point in terms of
the amplitude $\varepsilon$.

\begin{theorem} \label{maintheo2}
For $\mu$ in a right or left neighborhood of $0$, there exists a bifurcated time-periodic solution of (\ref{4}) which is a traveling wave function $\widehat{V}_{\varepsilon }(\xi,y)$,
of the form
\begin{equation*}
\widehat{V}_{\varepsilon }=\varepsilon V_{1}+\varepsilon ^{2}V_{2}+\mathcal{O%
}(\varepsilon ^{3})\in \mathcal{Z}_{\eta},
\end{equation*}%
with%
\begin{eqnarray*}
V_{1} &=&(e^{i\alpha \xi }\widehat{\zeta }+e^{-i\alpha \xi }\overline{%
\widehat{\zeta }})\text{ }\in \mathring{H}_{\eta}^{2}, \\
V_{2} &=&\widehat{w}_{2,2}e^{2i\alpha \xi }+\overline{\widehat{w}_{2,2}}%
e^{-2i\alpha \xi }+v_{2,0},
\end{eqnarray*}%
\begin{eqnarray*}
\mu (\varepsilon ) &=&\varepsilon ^{2}\mu _{2}+\mathcal{O}(\varepsilon ^{4}),
\\
\omega (\varepsilon ) &=&\omega _{0}+\varepsilon ^{2}\omega _{2}+\mathcal{O}%
(\varepsilon ^{4}),
\end{eqnarray*}%
Moreover, we have 
\begin{eqnarray*}
\mu _{2} \Bigl\langle (D^{2}-\alpha ^{2})\widehat{\zeta },\widehat{\zeta }^{\ast
} \Bigr\rangle -i\omega _{2} 
&=&c-b,\text{ }\Re \Bigl\langle (D^{2}-\alpha ^{2})%
\widehat{\zeta },\widehat{\zeta }^{\ast } \Bigr\rangle >0, \\
\widehat{w}_{2,2}e^{2i\alpha \xi } &=&(2i\omega _{0}-\boldsymbol{L}^{(0)}%
)^{-1}B(e^{i\alpha \xi }\widehat{\zeta },e^{i\alpha \xi }\widehat{\zeta }%
)\in \mathring{H}_{\eta}^{2}, \\
v_{2,0} &=&-2L_{0}^{-1}B(e^{i\alpha \xi }\widehat{\zeta },e^{-i\alpha \xi }%
\overline{\widehat{\zeta }})\in \mathring{W}^{2,\infty }, \\
c &=&-\Bigl\langle 2B(e^{i\alpha \xi }\widehat{\zeta },v_{2,0}),e^{i\alpha \xi }%
\widehat{\zeta }^{\ast } \Bigr\rangle , \\
b &=& \Bigl\langle 2B(e^{-i\alpha \xi }\overline{\widehat{\zeta }},e^{2i\alpha \xi
}\widehat{w}_{2,2}),e^{i\alpha \xi }\widehat{\zeta }^{\ast } \Bigr\rangle ,
\end{eqnarray*}%
\end{theorem}

We note that the bifurcation is supercritical if $\mu_2>0$, i.e. $\Re (c-b)>0$, where
$$
c - b = \Bigl\langle 4 B \Bigl( \zeta, {\bf L}_{0}^{-1} B(\zeta,\overline{\zeta})
- 2 B(\bar \zeta, (2 i \omega_0 - L)^{-1} B(\zeta,\zeta) \Bigr), \zeta^\star \Bigr\rangle .
$$

\begin{proof}
We use an adapted Lyapunov-Schmidt method (variant of the implicit function theorem).
Using the decomposition $v = \underline{v} + w$, the system becomes 
$$
\Bigl( \omega_0 {d \over ds} - {\bf L} \Bigr) w - {\bf L}_{0} v_{0,0}
= \mu {\bf L}' v - (\omega - \omega_0) {dv \over ds} + B(v,v),
$$
where we look for $v$ such that $w \in \frak{P}_0 H^{\sharp\sharp}$, $A \in \mathbb{C}$,
$v_{0,0} \in W^{2,\infty}$
when $(\mu, \omega - \omega_0)$ is close to $0$
in $\mathbb{R}^2$. After decomposition, this becomes
\begin{align} \label{5}
\Bigl( \omega_0 {d \over ds} - {\bf L} \Bigr) w
&= \mu \frak{P}_0 {\bf L}'(\underline{v} + w) 
- (\omega - \omega_0) {d w \over ds} 
+ \frak{P}_0 B(\underline{v} + w, \underline{v} + w),
\end{align}
\begin{equation} \label{6}
\Big\langle \mu {\bf L}'(\underline{v} + w) - (\omega - \omega_0)
{d \underline{v} \over ds} 
+ B(\underline{v} + w, \underline{v} + w) , \zeta^\star e^{i s} \Big\rangle
= 0,
\end{equation}
\begin{equation} \label{7}
{\bf L}_{0} v_{0,0} + \mu {\bf L}' v_{0,0} 
+ \Big[ B(\underline{v} + w, \underline{v} + w ) \Bigr]_{0,0}
= 0,
\end{equation}
where we look for solutions $(w,A,v_{0,0})$ in a neighborhood of $0$
in
$$
\frak{P}_0 H^{\sharp\sharp} \times \mathbb{C} \times W^{2,\infty} .
$$
We first solve (\ref{5}) with respect to $w \in \frak{P}_0 H^{\sharp\sharp}$ in function of 
$\underline{v}$, $\omega - \omega_0$, $v_{0,0}$ and $\mu$ in a neighborhood of $0$.
Moreover, the choice of our spaces (see (\ref{3})) implies that 
$$
{dw \over ds}, \frak{P}_0 B(v,v) \in
H^1(  \mathbb{T}_1 ,{\cal X}_\eta),
$$
which allows to apply the implicit function theorem with respect to $w$. 
The principal part which is independent of $w$ in the right hand side of (\ref{5}), is
$$
\mu {\bf L}' \underline{v} + \frak{P}_0 B(\underline{v},\underline{v})
$$
so we find
$$
w = W(A,\overline{A}, \omega - \omega_0, \mu, v_{0,0}) \in \frak{P}_0 H^{\sharp\sharp},
$$
with
\begin{align*}
W(A,\overline{A}, \omega - \omega_0, \mu, v_{0,0})  &= \Bigl( \omega_0
{d \over ds} - {\bf L} \Bigr)^{-1} 
\Bigl[\mu \frak{P}_0{\bf L}' \underline{v} + \frak{P}_0 B(\underline{v},\underline{v}) \Bigr]
\\
& \quad + O \Bigl(
\| \underline{v} \|^3 + \Bigl[ |\mu| + |\omega - \omega_0 | \Bigr]
\| \underline{v} \|^2 + |\mu| \, | \omega - \omega_0 | \, \| \underline{v} \| \Bigr).
\end{align*}
We note that for $A = 0$, {\it id est} for $\underline{v} = v_{0,0}$, we have
$$
W(0,0,\omega - \omega_0,\mu,v_{0,0}) = 0.
$$
Using now $T_a$, we have
$$
T_a( A\zeta e^{is} ) = A e^{i a} \zeta e^{is},
\qquad
T_a \underline{v}(s) = \underline{v}(s+a),
$$
hence, the commuting property leads to
$$
T_a W (A, \overline{A}, \omega - \omega_0, \mu, v_{0,0}) 
= W(e^{i a} A , e^{- i a } \overline{A}, \omega - \omega_0, \mu, v_{0,0}).
$$
 Moreover, we note that,  using $\tau_{-a / \alpha}$
\begin{equation} \label{propagation}
T_a \tau_{-a / \alpha} = Id, \text{ for any } a\in \mathbb{R}
\end{equation}
on the $\underline{v}$ part.

Replacing $w$ by $W$ in equations (\ref{6}) and (\ref{7}), we obtain an infinite
dimensional system
\begin{align*}
F(A, \overline{A}, \omega - \omega_0, \mu, v_{0,0}) &= 0,
\\
{\bf L}_{0} v_{0,0} + \mu {\bf L}' v_{0,0} 
+ [ B(\underline{v} + W, \underline{v} + W) ]_{0,0} &= 0.
\end{align*}
We note that, applying $T_a$ for any $a \in \mathbb{R}$,
$$
F(e^{i a } A , e^{- i a} \overline{A}, \omega - \omega_0, \mu, v_{0,0} )
= e^{ia} F(A, \overline{A}, \omega - \omega_0, \mu, v_{0,0} ),
$$
so that $A$ is in factor in $F$, which only depends on $|A|^2$. Moreover
$$
\Big\langle \mu {\bf L}' (\underline{v} + W), \zeta^\star e^{is} \Big\rangle
= A \mu \langle \Delta \zeta, \zeta^\star \rangle+ O( | \mu | \, \| W \|),
$$
$$
\Big\langle (\omega - \omega_0) {d\underline{v} \over ds}, \zeta^\star e^{is} \Big\rangle
= i A (\omega - \omega_0),
$$
$$
\Big\langle B(\underline v, \underline v), \zeta^\star e^{is} \Big\rangle
= 2 A \Big\langle B(\zeta, v_{0,0} ), \zeta^\star \Big\rangle,
$$
hence, it is clear that $F$ is of the form
\begin{align} \label{8}
0 &= A \mu \langle \Delta \zeta, \zeta^\star \rangle
- i A (\omega - \omega_0) + b A |A|^2 + 2 A \langle B(\zeta,v_{0,0}),\zeta^\star \rangle
\\
& \quad + 
O\Bigl( |A| \, |\mu|^2 + |\mu| \, |A|^2
+ |\mu| \, |A| \, |v_{0,0} |
+ |A| \, |v_{0,0}|^2 + |A|^4 \Bigr)
\end{align}
where
\begin{equation} \label{9}
b = 2 \Big\langle B \Big( \overline{\zeta}, (2 i \omega_0 - L)^{-1} B(\zeta,\zeta) \Big), \zeta^\star  \Big\rangle.
\end{equation}
Finally, looking for $A \ne 0$ leads to
\begin{align} \label{10}
0 &= \mu \langle \Delta \zeta, \zeta^\star \rangle- i (\omega - \omega_0)
+ b |A|^2 + 2 \langle B(\zeta,v_{0,0}), \zeta^\star \rangle
\\ & \quad
+ O \Bigl(  |\mu|^2 + (|\mu| + |\omega - \omega_0| ) |A|^2 + |A|^4 \Bigr).
\end{align}
Let us notice that
$$
\langle {\bf L}' \zeta, \zeta^\star \rangle 
= \langle \Delta \zeta, \zeta^\star \rangle
= \partial_\nu \lambda(\alpha,\nu_0(\alpha)).
$$
The $(0,0)$ mode equation  (\ref{7}) has a very specific form. Indeed, let $P_0$ be the projection giving the 0-Fourier mode.  We observe that 
\begin{equation*}
B(P_0 v,P_0 v)=0,\text{   for any  }v\in W^{2,\infty},
\end{equation*}
hence the term $B(\underline{v} + W, \underline{v} + W )$ decays like $e^{-\eta y}$. It results that  $[B(\underline{v} + W, \underline{v} + W )]_{0,0} \in L_{\eta}^{\infty}$. Moreover we observe that
\begin{equation*}
{\bf L}_{0,0}+\mu {\bf L}'=(1+\frac{\mu}{\nu_0}){\bf L}_{0,0},
\end{equation*}
hence, using Lemma \ref{lemma 20}, (\ref{7}) may be written in $W^{2,\infty}$ as
\begin{equation*}
(1+\frac{\mu}{\nu_0})v_{0,0}+{\bf L}_{0,0}^{-1}[B(\underline{v} + W, \underline{v} + W )]_{0,0}=0,
\end{equation*}
which can be solved in $v_{0,0} \in W^{2,\infty}$ by the implicit function theorem, noticing that
 $A$ only occurs through $|A|^2$.
The principal part which is independent of $v_{0,0}$ in (\ref{7})
is
$$
\Bigl[ B \Bigl( A e^{i s} \zeta 
+ \bar A e^{- i s} \bar \zeta,
A e^{i s} \zeta + \bar A e^{-i s} \bar \zeta \Bigr) \Bigr]_{0,0}
= 2 | A|^2 B(\zeta, \bar \zeta) \in C^0_\eta.
$$
This leads to
$$
v_{0,0} = - 2 |A|^2 {\bf L}_{0,0}^{-1} B(\zeta, \overline{\zeta})
+ O \Bigl( |\mu | \, |A|^2 + |A|^4 \Bigr).
$$
Replacing $v_{0,0}$ by its expression
in (\ref{10}) leads to a two dimensional real system, which can be solved by the implicit
function theorem, with respect to $\mu$ and $\omega - \omega_0$.
The equivariance of the system under the groups $T_a$ and
$\tau_\beta$ 
%and (\ref{propagation}) of the $\underline{v}$ part
implies that (\ref{propagation}) also holds on $W$.
Thus we obtain a traveling wave, with a function depending on $(s,x)$ through $x+s/\alpha$. 
Defining
$$
c = 4 \Big\langle B(\zeta, L_{0,0}^{-1} B(\zeta, \overline{\zeta})),
\zeta^\star \Big\rangle,
$$
we finally obtain the bifurcating solution, parametrized by $A$
(defined up to a phase shift)
\begin{align*}
\mu &= {\Re(c - b) \over \Re \langle \Delta \zeta, \zeta^\star \rangle}
|A|^2 + O(|A|^4)
\\ 
\omega - \omega_0 & = \Bigl( \Im (b-c) 
+ {\Re (c - b) \Im \langle \Delta \zeta, \zeta^\star \rangle \over
\Re \langle \Delta \zeta, \zeta^\star \rangle} \Bigr) |A|^2 + O(|A|^4),
\\
v_{0,0} &= - 2 |A|^2 L_{0,0}^{-1} B(\zeta,\overline{\zeta}) + O(|A|^4)
\in W^{2,\infty},
\\
w &= A^2 e^{2 i s} (2 i \omega_0 - L)^{-1} B(\zeta,\zeta)
+ c.c. + O(|A|^3) \in \frak{P}_0 H^{\sharp\sharp},
\\
v &= A e^{is} \zeta + \overline{A} e^{-i s} \overline{\zeta} + v_{0,0} + w.
\end{align*}
The theorem is proved as soon as we replace $|A|$ by $\varepsilon$.
\end{proof}

We note that the phase of $A$ is arbitrary, which corresponds to an arbitrary shift parallel to the $x$ axis , 
or to a shift in time. Moreover, by assumption,
$$
\Re \langle \Delta \zeta, \zeta^\star \rangle > 0.
$$
Then, the supercriticality or subcriticality of the bifurcation depends on
the sign of $\Re (c - b)$ which needs to be computed.

\section{Nonlinear stability of the bifurcating traveling wave}

%%%%%%%%%%%%%%%%%%%%%%%%%%%%%%%%%%%%%%%%%%%%%%%%%%%%%%%%%%%%%%%%

In all what follows the bracket $\langle \cdot ,\cdot \rangle $ is
understood as the duality product between $\mathcal{X}_{\eta}=\mathring{L}%
_{\eta}^{2}\oplus \mathring{L}^{\infty }$ and $\mathcal{X}_{\eta}^{\ast }=(%
\mathring{L}_{\eta}^{2})^{\ast }\oplus \mathring{L}^{1}$.

%%%%%%%%%%%%%%%%%%%%%%%%%%%%%%%%%%%%%%%%%%%%

\subsection{Study of the linearized operator}

%%%%%%%%%%%%%%%%%%%%%%%%%%%%%%%%%%%%%%%%%%%%

We proved that the bifurcating periodic solution is in fact a travelling
wave, function of $(\alpha x+\omega t,y).$ It is then natural to change
coordinates and to replace $x$ by 
\begin{equation*}
\xi =x+\frac{\omega t}{\alpha }.
\end{equation*}%
With these coordinates, we define%
\begin{equation*}
\widehat{v}(\xi ,y,t)=v(x,y,t),
\end{equation*}%
so that the Navier-Stokes system becomes%
\begin{equation}
\frac{d\widehat{v}}{dt}=\boldsymbol{L}_{\nu ,\omega }\widehat{v}+B(\widehat{v%
},\widehat{v}),\quad\nabla \cdot \widehat{v}=0,  \label{new NS equ}
\end{equation}%
with%
\begin{eqnarray*}
\boldsymbol{L}_{\nu ,\omega }\widehat{v} =\boldsymbol{L}_{(\nu )}\widehat{v%
}-\frac{\omega }{\alpha }\frac{\partial \widehat{v}}{\partial \xi } 
=\boldsymbol{L}^{(0)}\widehat{v}+\mu \boldsymbol{L}^{\prime }\widehat{v}-\frac{%
\omega }{\alpha }\frac{\partial \widehat{v}}{\partial \xi },
\end{eqnarray*}%
where%
\begin{equation*}
\boldsymbol{L}^{(0)}=\boldsymbol{L}_{(\nu_0)}^{(0)}+\boldsymbol{L}^{(1)}.
\end{equation*}%
The linear operator $\boldsymbol{L}_{\nu ,\omega }$ is acting in $\mathcal{X}%
_{\eta}$, with domain $\mathcal{Z}_{\eta}.$ In the following we need to localize
the essential spectrum of the linear operator 
$$\boldsymbol{L}_{(\nu )}^{(0)}+\boldsymbol{L}^{(0,1)}-%
\frac{\omega}{\alpha }\frac{\partial }{\partial \xi }=\nu \Pi \Delta -\frac{%
\omega+\alpha U_+ }{\alpha }\frac{\partial }{\partial \xi }$$ in $\boldsymbol{\mathring{L%
}}_{\eta}^{2}.$ This is described in the following Lemma.

\begin{lemma}
\label{lem: spect L(0)omega}There exists $\eta>0$ small enough, such that for
any $\lambda $ in the spectrum of 
$$
\nu \Pi \Delta -\frac{\omega+\alpha U_+ }{\alpha }%
\frac{\partial }{\partial \xi }
$$ 
acting in $\boldsymbol{\mathring{L}}_{\eta}^{2}$,
we have $\lambda \in \Sigma_{U_{+},\omega}$ where (see Figure \ref{spectrumc}) 
\begin{eqnarray*}
\Sigma_{U_{+},\omega} = & & \Bigl\{ \Re \lambda \leq -\nu (\alpha ^{2}-\eta^{2}) \Bigr\} 
\\
&\cap& \Bigl\{ \, \Bigl\{ 2\pi /3 \leq \arg |(\lambda _{0}+\nu \alpha ^{2})|\leq \pi \Bigr\} 
\cup \Bigl\{
\Re \lambda \leq -\frac{\nu \alpha ^{2}}{(\omega+\alpha U_+) ^{2}+4\nu ^{2}\alpha^{2}\eta^{2}}(\Im \lambda )^{2}+\eta^{2}\nu  \Bigr\} \, \Bigr\}.
\end{eqnarray*}%
For $\lambda $ outside of this region $\Sigma _{U+,\omega}$, where 
\begin{equation*}
0\leq |\arg (\lambda _{0}+\nu \alpha ^{2})|\leq \frac{2\pi}{3}-\delta
\end{equation*}%
with $0<\delta <\pi /6$, and 
where we add some small $\delta_1$ to $U_+$, there exists $C>0$ such that%
\begin{equation*}
\Bigl\| \Bigl( \nu \Pi \Delta -\frac{\omega+\alpha U_+ }{\alpha }\frac{\partial }{\partial \xi }%
-\lambda \mathbb{I} \Bigr)^{-1} \Bigr\|_{\mathcal{L}(\boldsymbol{\mathring{L}}%
_{\eta}^{2})}\leq \frac{C}{|\lambda +\nu \alpha ^{2}|}.
\end{equation*}
\end{lemma}
The proof of this Lemma is identical to the proof of Lemma \ref{newlemma} for the essential spectrum, up 
to the change of $U_+$ in $U_+ + \omega/\alpha$. This enlarges the parabolic region.
\begin{figure}[th]
\begin{center}
\includegraphics[width=5cm]{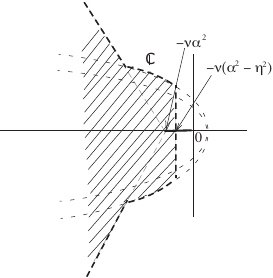}
\end{center}
\caption{Hatched region $\Sigma_{U_+,\omega}$ and negative real line
containing the spectrum of $\protect\nu\Pi\Delta-(U_{+}+\frac{\protect\omega }{%
\protect\alpha })\frac{\partial }{\partial \protect\xi }$ and the essential
spectrum of $\boldsymbol{L}_{\protect\nu,\protect\omega}$.}
\label{spectrumc}
\end{figure}

We know that there is a bifurcating {steady} solution of (\ref%
{new NS equ}) which is%
\begin{equation*}
\widehat{V}_{\varepsilon }(\xi ,y)\overset{def}{=}V_{\varepsilon }(x,y,t),%
\text{ }\omega =\omega (\varepsilon ),\mu =\mu (\varepsilon ),
\end{equation*}%
as described in Theorem \ref{maintheo2}.

Now we introduce $\phi _{0}^{\ast },\phi _{1}^{\ast }\in (\mathring{L}%
_{\eta}^{2})^{\ast }$ defined by%
\begin{eqnarray*}
\phi _{0}^{\ast } &=&\frac{i}{2}(e^{i\alpha \xi }\widehat{\zeta ^{\ast }}%
-e^{-i\alpha \xi }\overline{\widehat{\zeta ^{\ast }}}), \\
\phi _{1}^{\ast } &=&\frac{1}{2}(e^{i\alpha \xi }\widehat{\zeta ^{\ast }}%
+e^{-i\alpha \xi }\overline{\widehat{\zeta ^{\ast }}}),
\end{eqnarray*}%
and we prove the following Lemma:

\begin{lemma}
\label{lem:structure Lomega}For $\eta>0$ small enough, 
there exists $M$ such
that the spectrum of $\boldsymbol{L}_{\nu ,\omega }$ is for one part composed of an
essential spectrum included in the region $\Sigma _{\omega }$
and in the half negative real line. The rest of
the spectrum is a set of isolated eigenvalues $\lambda $ with finite
multiplicities, such that%
\begin{equation*}
\lambda =\lambda _{n}^{(0)}-in\omega ,\text{ }
\end{equation*}%
where $\lambda _{n}^{(0)}$ is an eigenvalue of $\boldsymbol{L}_{(\nu )}$
associated with the eigenvector $\zeta _{n}=e^{in\alpha }\widehat{v}_{n}\in 
\mathring{H}_{\eta}^{2}.$

Defining the linear operator%
\begin{equation*}
\mathcal{A}_\varepsilon = \boldsymbol{L}_{\nu ,\omega }+2B(\widehat{V}_{\varepsilon },\cdot ),
\end{equation*}%
then, the essential spectrum of $\mathcal{A}_\varepsilon$ is made of the half negative real line and of  
a perturbation of order $\varepsilon^2$ of the region $\Sigma_\varepsilon$ of figure \ref{spectrumc}.

For $\varepsilon =0,$ the operator%
\begin{equation*}
\mathcal{A}_0 = \boldsymbol{L}_{\nu ,\omega }|_{\varepsilon =0}=\boldsymbol{L}_{(\nu_0)}-\frac{\omega
_{0}}{\alpha }\frac{\partial }{\partial \xi }
\end{equation*}%
has a double $0$ eigenvalue. For $\varepsilon \neq 0,$ $0$ is a simple
eigenvalue of $\mathcal{A}_\varepsilon$ with corresponding eigenvector%
\begin{eqnarray*}
v_{\varepsilon } &=&\frac{1}{\alpha \varepsilon }\frac{\partial \widehat{V}%
_{\varepsilon }}{\partial \xi }=i(e^{i\alpha \xi }\widehat{\zeta }%
-e^{-i\alpha \xi }\overline{\widehat{\zeta }})+\mathcal{O}(\varepsilon )%
\text{ }\in \mathring{H}_{\eta}^{2} \\
\langle v_{\varepsilon },\phi _{0}^{\ast }\rangle &=&1,\text{ }\langle
v_{\varepsilon },\phi _{1}^{\ast }\rangle =0,
\end{eqnarray*}%
and there is another simple eigenvalue $\sigma _{\varepsilon
}=2 \Re (b-c)\varepsilon ^{2}+\mathcal{O}(\varepsilon ^{3})$ in the neighborhood
of $0.$ All other eigenvalues $\lambda $ of $\mathcal{A}_\varepsilon$ are such that there exists $\gamma $
independent of $\varepsilon ,$ with%
\begin{equation*}
\Re \lambda <-\gamma <0
\end{equation*}%
for $\varepsilon $ close to $0$. Moreover the eigenvector $u_{\varepsilon }$
corresponding to $\sigma _{\varepsilon }$ has an order $1$ part in $\mathring{H}%
_{\eta}^{2},$ and an order $\varepsilon $ part in $\mathring{%
W}^{2,\infty },$ and $u_{\varepsilon }$ satisfies%
\begin{equation*}
\langle u_{\varepsilon },\phi _{0}^{\ast }\rangle =0,\text{ }\langle
u_{\varepsilon },\phi _{1}^{\ast }\rangle =1.
\end{equation*}
\end{lemma}

\begin{proof}
The perturbation $-\frac{\omega }{\alpha }\frac{\partial }{\partial \xi }$
is in $\mathcal{L}(\mathcal{Z}_{\eta},\mathring{H}_{\eta}^{1})$ and is
relatively bounded with respect to $\boldsymbol{L}^{(0)}.$ The proof of
for Lemma \ref{lem: locationspect} may be applied to the perturbated operator $%
\boldsymbol{L}^{(1)}-\frac{\omega }{\alpha }\frac{\partial }{\partial \xi }$
and implies that the spectrum of $\boldsymbol{L}_{\nu ,\omega }$ is included in
the region $\Sigma $ defined by%
\begin{equation*}
|\lambda +\nu \alpha ^{2}|<M,
\end{equation*}%
\begin{equation*}
\text{or } 2\pi /3 \leq |\arg (\lambda +\nu \alpha ^{2})|\leq \pi ,\text{ or }%
\lambda \in (-\infty ,0],
\end{equation*}%
which is a right bounded region, centered on the negative axis. This result
is valid for the whole spectrum, including the essential spectrum and the
isolated eigenvalues.

Since the perturbation operator $\boldsymbol{L}^{(0,1)}$ is relatively compact
with respect to $\boldsymbol{L}^{(0)}$ and to $\nu \Pi
\Delta -\frac{\omega+\alpha U_{+} }{\alpha }\frac{\partial }{\partial \xi },$ we can
assert that the essential spectrum of $\boldsymbol{L}_{\nu ,\omega }$ is
formed by $(-\infty ,0]$, corresponding to its action on  $\mathring{W}%
^{2,\infty },$ and a part in $\Sigma _{U_+,\omega }$ corresponding to its
action on $\mathring{H}_{\eta}^{2}$ (see Figure \ref{spectrumc}). The rest of
the spectrum is composed of isolated eigenvalues with finite multiplicities,
deduced from those of $\boldsymbol{L}_{(\nu )}$ in the following simple way.
First we notice that there is no change in the subspace $%
\mathring{L}^{\infty }$ (0-Fourier mode),  where the eigenvalues are the same for both operators.
All other eigenvalues correspond to eigenvectors in $\mathring{H}_{\eta}^{2}$
of the form 
\begin{equation*}
e^{ni\alpha \xi }\widehat{\zeta }(y),\text{ }n\neq 0,
\end{equation*}%
so that an eigenvalue $\lambda $ of $\boldsymbol{L}_{(\nu )}$ corresponds to
an eigenvalue $\lambda -in\omega $ of $\boldsymbol{L}_{\nu ,\omega },$ which
does not change the real part of the eigenvalue.

Now about the linear operator ${\cal A}_\varepsilon$, we observe that the perturbation $2 B(\hat{V}_\varepsilon,\cdot)$ is not relatively compact because of the $0$ Fourier component of $\hat{V}_\varepsilon$, which
does not decay as $y$ goes to infinity. However, the part 
$2 B(P_0 \hat{V}_\varepsilon,\cdot)$, which is of order $\varepsilon^2$,
only acts on non zero Fourier components. Hence it perturbs the spectrum
(including the essential spectrum) at order $\varepsilon^2$ (see \cite{Kato}).
The rest of the perturbation $2 B(\tilde{P} \hat{V}_\varepsilon,\cdot)$
is relatively compact, hence does not perturb the essential spectrum.

Now, we observe that $0$ is a double eigenvalue of $\boldsymbol{L}_{\nu
,\omega }|_{\varepsilon =0}=\boldsymbol{L}_{(\nu_0)}-\frac{\omega _{0}}{\alpha }\frac{%
\partial }{\partial \xi }$ with corresponding $2$-dimensional (on $\mathbb{R}$)
eigenspace%
\begin{equation}
ae^{i\alpha \xi }\widehat{\zeta }+\overline{a}e^{-i\alpha \xi }\overline{%
\widehat{\zeta }}.  \label{eigenspace for 0}
\end{equation}%
Then differentiating 
$$
{\bf L}_{\nu,\omega} \widehat{V}_\varepsilon + B( \widehat{V}_\varepsilon, \widehat{V}_\varepsilon) = 0
$$
with respect to $\xi $ leads to%
\begin{equation*}
\boldsymbol{L}_{\nu ,\omega }v_{\varepsilon }+2B(\widehat{V}_{\varepsilon
},v_{\varepsilon }) = 0,
\end{equation*}%
which shows that $v_{\varepsilon }\in \mathring{H}_{\eta}^{2}$ (defined in Lemma %
\ref{lem:structure Lomega}) belongs to the kernel of the linearized operator ${\cal A}_\varepsilon$,
%\begin{equation*}
%\boldsymbol{L}_{\nu ,\omega }+2B(\widehat{V}_{\varepsilon },\cdot ),
%\end{equation*}%
in particular for $\varepsilon =0,$ which corresponds to $a=i$ in (\ref%
{eigenspace for 0}). Let us look for eigenvalues close to $0$ for $%
\varepsilon $ close to $0.$ Here the problem is not standard since $0$ is
double, not isolated and lies in the essential spectrum of $\boldsymbol{L}%
_{\nu _{0},\omega _{0}}.$ However we can justify the following computations, which are
identical to the computations used in the standard case. As soon as we can
obtain the principal part of a potential eigenvector and of the potential
eigenvalue, the rest of the expansion in powers of $\varepsilon $ relies on
the implicit function theorem, as in the computation of the bifurcating
periodic solution. Now we check that%
\begin{equation*}
\langle v_{\varepsilon },\phi _{0}^{\ast }\rangle =1,\text{ }\langle
v_{\varepsilon },\phi _{1}^{\ast }\rangle =0,
\end{equation*}%
due to the fact that higher order terms in $v_{\varepsilon }$ occur with
harmonics $e^{ni\xi },$ with $|n|>1.$ We note that, for any $u\in \mathring{%
H}_{\eta}^{2}$, we have in $\mathring{L}_{\eta}^{2},$ 
\begin{equation}
\Bigl\langle (\boldsymbol{L}_{(\nu_0)}-\frac{\omega _{0}}{\alpha }\frac{\partial }{\partial
\xi })u,\phi _{j}^{\ast } \Bigr\rangle =0,\text{ }j=0,1.  \label{prop phi*}
\end{equation}%
Indeed, by definition of the duality product $\langle \mathring{L}_{\eta}^{2},(%
\mathring{L}_{\eta}^{2})^{\ast }\rangle$, we have 
$$
\Bigl\langle \frac{\partial }{\partial \xi }u,\phi _{j}^{\ast }\Bigr\rangle 
+\Bigl\langle u,\frac{\partial }{\partial \xi }\phi _{j}^{\ast }\Bigr\rangle =0,
$$ 
hence
\begin{equation*}
\Bigl\langle (\boldsymbol{L}_{(\nu_0)}-\frac{\omega _{0}}{\alpha}\frac{\partial }{\partial
\xi })u,\phi _{j}^{\ast } \Bigr\rangle =
\Bigl\langle u,(\boldsymbol{L}_{(\nu_0)}^{\ast }+\frac{\omega _{0}}{\alpha }\frac{\partial }{\partial \xi })\phi _{j}^{\ast} \Bigr\rangle =0
\end{equation*}%
since%
\begin{equation*}
\Bigl( \boldsymbol{L}_{(\nu_0)}^{\ast }+\frac{\omega _{0}}{\alpha }\frac{\partial }{\partial
\xi } \Bigr)\widehat{\zeta ^{\ast }}=0.
\end{equation*}%
We may interpret (\ref{prop phi*}) in saying that $\phi _{j}^{\ast }$ for $j=1,2$
are in the kernel of $(\boldsymbol{L}_{(\nu_0)}^{\ast }+\frac{\omega _{0}}{\alpha }%
\frac{\partial }{\partial \xi })$.  It is known (see \cite%
{Kato} p.185) that non isolated eigenvalues might not exist for the adjoint
operator. Here we are saved by the fact that this occurs in $\mathring{H}%
_{\eta}^{2}$ where $0$ is not in the essential spectrum of the reduced operator.

Now, since we perturb a double eigenvalue, we need to find another
eigenvalue (necessarily real) close to $0.$ For this search, we make the Ansatz%
\begin{eqnarray*}
\sigma &=&\varepsilon \sigma _{1}+\varepsilon ^{2}\sigma _{2}+... \\
u &=&u_{0}+\varepsilon u_{1}+\varepsilon ^{2}u_{2}+...\in \mathcal{Z}_{\eta}
\end{eqnarray*}%
and identify powers of $\varepsilon $ in the identity%
\begin{equation*}
\Bigl[ \boldsymbol{L}_{(\nu_0)}+\mu \boldsymbol{L}^{\prime }-\frac{\omega }{\alpha }%
\frac{\partial }{\partial \xi }+2B(\widehat{V}_{\varepsilon },\cdot
)\Bigr]u=\sigma u.
\end{equation*}%
Order $\varepsilon ^{0}$ terms lead to%
\begin{equation*}
\boldsymbol{L}_{(\nu_0)}u_{0}-\frac{\omega _{0}}{\alpha }\frac{\partial u_{0}}{%
\partial \xi }=0,
\end{equation*}%
which gives%
\begin{equation*}
u_{0}=ae^{i\alpha \xi }\widehat{\zeta }+\overline{a}e^{-i\alpha \xi }%
\overline{\widehat{\zeta }},\text{ }a\in \mathbb{C},
\end{equation*}%
where we notice that $a=i$ corresponds to the already known solution $%
v_{\varepsilon }.$ At order $\varepsilon ,$ we obtain%
\begin{equation*}
\sigma _{1}u_{0}=\boldsymbol{L}_{(\nu_0)}u_{1}-\frac{\omega _{0}}{\alpha }\frac{%
\partial u_{1}}{\partial \xi }+2B(V_{1},u_{0}).
\end{equation*}%
Taking the duality product with $e^{i\alpha \xi }\widehat{\zeta ^{\ast }}$
leads to%
\begin{equation*}
a\sigma _{1}= \Bigl\langle 2B(V_{1},u_{0}),e^{i\alpha \xi }\widehat{\zeta ^{\ast }}%
\Bigr\rangle
\end{equation*}%
which vanishes because of the periodicity in $\xi $ and of factors like $%
e^{ni\alpha \xi },$ with $n$ odd in the duality product. Hence $\sigma
_{1}=0 $ and%
\begin{eqnarray*}
u_{1} &=&-\widetilde{ \Bigl( \boldsymbol{L}_{(\nu_0)}-\frac{\omega _{0}}{\alpha }\frac{%
\partial }{\partial \xi } \Bigr)}^{-1}2B(V_{1},u_{0}) \\
&=&2a\widehat{w}_{2,2}e^{2i\alpha \xi }+2\overline{a}\overline{\widehat{w}%
_{2,2}}e^{-2i\alpha \xi }+(a+\overline{a})v_{2,0}.
\end{eqnarray*}%
At order $\varepsilon ^{2}$ we obtain%
\begin{equation*}
\sigma _{2}u_{0}= \Bigl( \boldsymbol{L}_{(\nu_0)}-\frac{\omega _{0}}{\alpha }\frac{\partial }{%
\partial \xi } \Bigr)u_{2}+\mu _{2}\boldsymbol{L}^{\prime }u_{0}-\frac{\omega _{2}%
}{\alpha }\frac{\partial u_{0}}{\partial \xi }%
+2B(V_{1},u_{1})+2B(V_{2},u_{0}).
\end{equation*}%
Taking the duality product with $e^{i\alpha \xi }\widehat{\zeta ^{\ast }}$
leads to%
\begin{equation*}
a\sigma _{2}=a\mu _{2} \Bigl\langle (D^{2}-\alpha ^{2})\widehat{\zeta },\widehat{%
\zeta ^{\ast }} \Bigr\rangle -ia\omega _{2}-(a+\overline{a})c+2ab+\overline{a}b-ac.
\end{equation*}%
The identity%
\begin{equation*}
\mu _{2} \Bigl\langle (D^{2}-\alpha ^{2})\widehat{\zeta },\widehat{\zeta }^{\ast
}\Bigr\rangle -i\omega _{2}=c-b
\end{equation*}%
reduces the above identity to%
\begin{eqnarray*}
a\sigma _{2} =(a+\overline{a})(b-c).
\end{eqnarray*}%
We know that $(a+\overline{a})=0$ is a solution at any order, which gives the
eigenvalue $0$ and the eigenvector $v_{\varepsilon }$. We deduce the other
solution (which is defined up to a real factor) by choosing %
\begin{equation*}
\arg a=\theta,\text{  }a=\frac{e^{i\theta}}{\cos\theta}
\end{equation*}%
and%
\begin{equation*}
\sigma _{2}=2\Re(b-c).
\end{equation*}%
Now
\begin{equation*}
\langle u_{0},\phi _{0}^{\ast }\rangle =0,\text{ }\langle u_{0},\phi
_{1}^{\ast }\rangle =1.
\end{equation*}%
 We can go on the computation of $\sigma_n$ using the Fredholm alternative as in the simple case, so that finally
\begin{equation*}
\langle u_{\varepsilon },\phi _{0}^{\ast }\rangle =0,\text{ }\langle
u_{\varepsilon },\phi _{1}^{\ast }\rangle =1,
\end{equation*}%
because of periodicity in $\xi $ and factors as $e^{in\xi }$ with $|n|\neq 1$
for $u_{\varepsilon }-u_{0}.$
This ends the proof of Lemma \ref%
{lem:structure Lomega}. 
\end{proof}

%%%%%%%%%%%%%%%%%%%%%%%%%%%%%%%%%%%%%%%%%%%%%%%%%%%

\subsection{Elimination of the Goldstone mode. Operator $\boldsymbol{A}_{\protect\varepsilon }.$}

%%%%%%%%%%%%%%%%%%%%%%%%%%%%%%%%%%%%%%%%%%%%%%%%%%%

The Goldstone mode is the eigenvector $\frac{\partial V_{\varepsilon }}{%
\partial \xi }=\alpha \varepsilon v_{\varepsilon }.$ Using the invariance of
the system (\ref{new NS equ}) under translations in $\xi ,$ we can eliminate
the $0$ eigenvalue and obtain a system in the codimension 1 subspace%
\begin{equation*}
\langle u,\phi _{0}^{\ast }\rangle =0.
\end{equation*}%
Indeed, let us set%
\begin{equation}
\widehat{v}=\tau _{b}(\widehat{V}_{\varepsilon }+\varepsilon u),\text{ }%
\langle u,\phi _{0}^{\ast }\rangle =0,  \label{def shift mode}
\end{equation}%
where $\tau _{b}$ represents the shift $\xi \rightarrow \xi +b,$ 
keeping in mind that
$\langle v_{\varepsilon },\phi _{0}^{\ast }\rangle =1$.
Then 
\begin{equation*}
\frac{d\widehat{v}}{dt}=\tau _{b}\frac{\varepsilon du}{dt}+\varepsilon \frac{%
db}{dt}\tau _{b} \Bigl( \alpha v_{\varepsilon }+\frac{\partial u}{\partial \xi } \Bigr)
\end{equation*}%
and (\ref{new NS equ}) becomes, after factoring out $\varepsilon \tau _{b}$%
\begin{equation}
\frac{du}{dt}+\frac{db}{dt} \Bigl( \alpha v_{\varepsilon }+\frac{\partial u}{%
\partial \xi } \Bigr)=\boldsymbol{L}_{\nu ,\omega }u+2B(\widehat{V}_{\varepsilon
},u)+\varepsilon B(u,u).  \label{decomp NS}
\end{equation}%
Taking the duality product with $\phi _{0}^{\ast }$ leads to%
\begin{equation}
\frac{db}{dt} \Bigl( \alpha + \Bigl\langle \frac{\partial u}{\partial \xi },\phi
_{0}^{\ast } \Bigr\rangle \Bigr)=\langle \boldsymbol{L}_{\nu ,\omega }u,\phi _{0}^{\ast
}\rangle + \Bigl\langle 2B(\widehat{V}_{\varepsilon },u)+\varepsilon B(u,u),\phi
_{0}^{\ast } \Bigr\rangle ,  \label{equ ba}
\end{equation}%
which may be written as%
\begin{equation}
\frac{db}{dt}=g_{\varepsilon }(u).  \label{equ b}
\end{equation}%
Taking into account 
\begin{equation*}
\Bigl\langle \frac{\partial u}{\partial \xi },\phi _{0}^{\ast } \Bigr\rangle 
= -\Bigl\langle
u,\frac{\partial }{\partial \xi }\phi _{0}^{\ast } \Bigr\rangle =\alpha \langle
u,\phi _{1}^{\ast } \rangle ,
\end{equation*}%
\begin{equation*}
\langle \boldsymbol{L}_{\nu ,\omega }u,\phi _{0}^{\ast }\rangle =\langle \mu 
\boldsymbol{L}^{\prime }u,\phi _{0}^{\ast }\rangle +\frac{\omega -\omega _{0}%
}{\alpha }\langle u,\phi _{1}^{\ast }\rangle =\mathcal{O}(\varepsilon ^{2}\|%
\widetilde{u}\|_{\mathring{L}_{\eta}^{2}}),
\end{equation*}%
we observe that, for $\|u\|_{\mathcal{Z}_{\eta}}$ small enough, 
\begin{equation*}
g_{\varepsilon }(u)=\frac{\varepsilon }{\alpha } \Bigl\langle 2B(V_{1},u),\phi
_{0}^{\ast } \Bigr\rangle +\mathcal{O} \Bigl(\varepsilon ^{2}\|u\|_{\mathcal{Z}%
_{\eta}}+\varepsilon \|u\|_{\mathcal{Z}_{\eta}}^{2} \Bigr).
\end{equation*}%
Now we define the projection $Q_{\varepsilon }$ for any $v\in \mathcal{X}_{\eta}$ by
\begin{equation*}
Q_{\varepsilon }v=v-\langle v,\phi _{0}^{\ast }\rangle v_{\varepsilon },
\end{equation*}%
and observe that, since $v_{\varepsilon }\in \mathring{H}_{\eta}^{2},$ 
\begin{equation*}
P_{0}(Q_{\varepsilon }v)=P_{0}v,\text{ }Q_{\varepsilon }u_{\varepsilon
}=u_{\varepsilon }.
\end{equation*}%
The rest of (\ref{decomp NS}), which is now independent of $b,$ becomes
\begin{equation}
\frac{du}{dt}=\boldsymbol{A}_{\varepsilon }u+\varepsilon Q_{\varepsilon
}B(u,u)-g_{\varepsilon }(u)Q_{\varepsilon }\frac{\partial u}{\partial \xi },
\label{proj NS}
\end{equation}%
where the linear operator $\boldsymbol{A}_{\varepsilon }$ is defined by 
\begin{equation*}
\boldsymbol{A}_{\varepsilon }=Q_{\varepsilon }(\boldsymbol{L}_{\nu ,\omega
}+2B(\widehat{V}_{\varepsilon },\cdot ))=Q_{\varepsilon }\mathcal{A}_\varepsilon.
\end{equation*}%
We know that \ the eigenvector $u_{\varepsilon }$ satisfies%
\begin{equation*}
\langle u_{\varepsilon },\phi _{0}^{\ast }\rangle =0,
\end{equation*}%
and 
\begin{equation*}
\Bigl( \boldsymbol{L}_{\nu ,\omega }+2B(\widehat{V}_{\varepsilon },\cdot
) \Bigr) u_{\varepsilon }=\sigma _{\varepsilon }u_{\varepsilon },
\end{equation*}%
so that we have

\begin{equation*}
\boldsymbol{A}_{\varepsilon }u_{\varepsilon }=Q_{\varepsilon } \Bigl( \boldsymbol{L}%
_{\nu ,\omega }+2B(\widehat{V}_{\varepsilon },\cdot ) \Bigr) u_{\varepsilon
}=\sigma _{\varepsilon }u_{\varepsilon }.
\end{equation*}

\begin{remark}
We note that $u_{\varepsilon }$ is not in $\mathring{H}^{2}$ since it
contains a $0-$Fourier mode. In fact we have
\begin{equation*}
u_{\varepsilon }=\widetilde{P}u_{\varepsilon }+P_{0}u_{\varepsilon },\text{ }%
\langle u_{\varepsilon },\phi _{1}^{\ast }\rangle =1,\text{ }%
P_{0}u_{\varepsilon }=2v_{2,0}\varepsilon +\mathcal{O}(\varepsilon ^{2}).
\end{equation*}
\end{remark}

The rest of the spectrum of $\boldsymbol{A}_{\varepsilon }$ is the same as
the spectrum of $\mathcal{A}_\varepsilon$ except the eigenvalues $\sigma _{\varepsilon }$ and $0.$ The
estimates obtained for the spectrum are similar to those for $\boldsymbol{L}%
_{\nu ,\omega }$.  Due to the perturbation of
order $\varepsilon ,$ we can assert that all eigenvalues other than $\sigma
_{\varepsilon }$, outside of the region where the essential spectrum is
located, are in the region indicated on Figure \ref{spectrumc}, have finite
multiplicities, are isolated and located on the left of the line $\Re
\lambda <-k<0.$ Observe that now the subspace with $0$ average and\ the
subspace with only the $0$ Fourier mode are coupled by the term $B(%
\widehat{V}_{\varepsilon },u)$: 
\begin{eqnarray*}
\widetilde{P}B(\widehat{V}_{\varepsilon },u) &=&\widetilde{P}B(\widetilde{P}%
\widehat{V}_{\varepsilon },\widetilde{u})+B(\widetilde{P}\widehat{V}%
_{\varepsilon },P_{0}u)+B(P_{0}\widehat{V}_{\varepsilon },\widetilde{u}) \\
P_{0}B(\widehat{V}_{\varepsilon },u) &=&P_{0}B(\widetilde{P}\widehat{V}%
_{\varepsilon },\widetilde{u}),
\end{eqnarray*}%
so that the estimates on the semi group are more complicated if we wish to
split the subspaces.

%%%%%%%%%%%%%%%%%%%%%%%%%%%%%%%%%%

\subsection{Nonlinear stability}

%%%%%%%%%%%%%%%%%%%%%%%%%%%%%%%%%%%

In the subspace $P_{\varepsilon }\mathcal{X}_{\eta}$ we need to solve the
initial value problem%
\begin{equation}
\frac{du}{dt}=\boldsymbol{A}_{\varepsilon }u+B_{\varepsilon }(u),\text{ }%
u(0)\in \mathcal{Z}_{\eta}  \label{new NS stab}
\end{equation}%
where $u(0)$ is small enough in norm, and where $B_{\varepsilon }$ is analytic
from $\mathcal{Z}_{\eta}$ to $\mathcal{Y}_{\eta}$ with%
\begin{equation*}
B_{\varepsilon }(u)=\varepsilon Q_{\varepsilon }B(u,u)-g_{\varepsilon
}(u)Q_{\varepsilon }\frac{\partial u}{\partial \xi }.
\end{equation*}%
Using that $Q_{\varepsilon }\widetilde{P}=%
\widetilde{P}Q_{\varepsilon },$ $P_{0}Q_{\varepsilon }=P_{0}$ and $%
B(P_{0}u,P_{0}v)\equiv 0$, we have
\begin{eqnarray*}
\widetilde{P}B_{\varepsilon }(\widetilde{u}+P_{0}u) &=&\varepsilon
Q_{\varepsilon }\widetilde{P}B(\widetilde{u},\widetilde{u})+2\varepsilon
Q_{\varepsilon }B(\widetilde{u},P_{0}u)-g_{\varepsilon }(\widetilde{u}%
+P_{0}u)Q_{\varepsilon }\frac{\partial \widetilde{u}}{\partial \xi }, \\
P_{0}B_{\varepsilon }(\widetilde{u}+P_{0}u) &=&\varepsilon P_{0}B(\widetilde{%
u},\widetilde{u}),
\end{eqnarray*}%
hence, there exists $\delta >0$ such that for $\|u\|_{\mathcal{Z}_{\eta}}\leq
\delta $ 
\begin{equation}
\|B_{\varepsilon }(u)\|_{\mathcal{Y}_{\eta,\eta}}\leq c\varepsilon \|u\|_{\mathcal{%
Z}_{\eta}}^{2},  \label{estimBeps}
\end{equation}%
\begin{equation*}
|g_{\varepsilon }(u)|\leq c\varepsilon \|u\|_{\mathcal{Z}_{\eta}},
\end{equation*}%
where%
\begin{equation*}
\mathcal{Y}_{\eta,\eta}=\Bigl\{f\in \mathcal{Y}_{\eta}; \, P_{0}f\in \mathring{W}%
_{\eta}^{1,\infty }\Bigr\}.
\end{equation*}
Using the properties of the operator $\boldsymbol{A}_{\varepsilon }$ we can
prove that the spectrum of $e^{\boldsymbol{A}_{\varepsilon }t}$ is the union
of an essential spectrum and of a bounded set of isolated eigenvalues of finite
multiplicities. More precisely the essential spectrum is included in the
union of the real interval $[0,1]$ with a set included in $\{\lambda \in 
\mathbb{C};\lambda =e^{\sigma t},\sigma \in \Sigma _{U_+,\omega }\}$ where $%
\Sigma _{U_+,\omega }$ is described in Figure \ref{spectrumc}. It results that in the
case when the eigenvalue $\sigma _{\varepsilon }<-\kappa \varepsilon ^{2},$
the spectral radius of $e^{\boldsymbol{A}_{\varepsilon }t}$ equals $1$. Hence,
we have 
\begin{equation*}
\|e^{\boldsymbol{A}_{\varepsilon }t}\|_{\mathcal{L}(\mathcal{Z}_{\eta})}\leq C,%
\text{ for }t\geq 0.
\end{equation*}%
However, this estimate is not sufficient to avoid secular terms in the
solution of the initial value problem (\ref{new NS stab}).

Let us proceed in adapting the proof of Theorem \ref{maintheo1}.
We obtain the following Lemma

\begin{lemma}
\label{lem:estimexp A Qtilde}Let us assume that $\sigma _{\varepsilon
}<-\kappa \varepsilon ^{2},$ then there exists $C(\varepsilon )>0$ such that%
\begin{equation}
\|e^{\boldsymbol{A}_{\varepsilon }t}\|_{\mathcal{L}(\mathcal{Z}%
_{\eta,\eta},\mathcal{Z}_{\eta})}\leq \frac{C(\varepsilon )}{1+t},\text{ }t\geq 0  \label{estim Zkka}
\end{equation}%
\begin{equation}
\|\widetilde{P}e^{\boldsymbol{A}_{\varepsilon }t}f\|_{\mathcal{L}(\mathring{H%
}_{\eta}^{2})} \leq \frac{C(\varepsilon )}{\sqrt{t}
(1 + \sqrt{t})}\|\widetilde{P}f\|_{%
\mathring{H}_{\eta}^{1}}+\frac{C(\varepsilon )}{1+t}\|P_{0}f\|_{%
\mathring{W}_{\eta}^{1,\infty }},\text{ }t>0  \label{estim Kkk} 
\end{equation}
$$
\|P_{0}e^{\boldsymbol{A}_{\varepsilon }t}f\|_{\mathcal{L}(\mathring{W}%
^{2,\infty })} \leq \frac{C(\varepsilon )}{1+t}\|\widetilde{P}f\|_{%
\mathring{H}_{\eta}^{1}}+\frac{C(\varepsilon )}{\sqrt{t}(1 + \sqrt{t})}\|P_{0}f\|_{\mathring{W%
}_{\eta}^{1,\infty }},\text{ }t>0.  \notag
$$
\end{lemma}
The proof is done in Appendix %
\ref{App: estim semigrAeps}. The estimate (\ref{estim Kkk}), valid for all $%
t>0$, shows the loss of regularity as $t\rightarrow 0.$ We are now ready to
prove the following

\begin{theorem} \label{maintheo3}
Let us assume that $\sigma _{\varepsilon }<0$ (i.e.  that 
the bifurcation is supercritical) and let us choose $\eta>0$ and $\varepsilon $ small
enough. Then there exists $\delta >0$ such that, for $\|u(0)\|_{%
\mathcal{Z}_{\eta,\eta}}\leq \delta$, there is a unique solution $u$ of 
\begin{equation}
\frac{du}{dt}=\boldsymbol{A}_{\varepsilon }u+B_{\varepsilon }(u),u(0)\in 
\mathcal{Z}_{\eta,\eta},\text{ }t\geq 0  \label{equ nonlin stab}.
\end{equation}%
Moreover, we have 
\begin{equation*}
\|u(t)\|_{\mathcal{Z}_{\eta}}\leq \frac{2C(\varepsilon )}{1+t}\|u(0)\|_{%
\mathcal{Z}_{\eta,\eta}},\text{ }t\geq 0.
\end{equation*}
\end{theorem}

Moreover, the shift $b(t)$, solution of (\ref{equ b}), satisfies
\begin{equation*}
|b(t)|\leq |b(0)|+ \ln(1 + t) C^{\prime \prime }(\varepsilon)\|u(0)\|_{\mathcal{Z}_{\eta,\eta}},\text{ }t\geq 0.
\end{equation*}

\begin{proof} 
The integral formulation of (\ref{equ nonlin stab}) is%
\begin{equation}
u(t)=e^{\boldsymbol{A}_{\varepsilon }t}u(0)+\int_{0}^{t}e^{\boldsymbol{A}%
_{\varepsilon }(t-s)}B_{\varepsilon } \Bigl( u(s) \Bigr) \, ds. \label{integ form a}
\end{equation}%
Using the estimates (\ref{estimBeps}), (\ref{estim Zkka}),  (\ref{estim Kkk}),
we obtain 
\begin{equation}
\|u(t)\|_{\mathcal{Z}_{\eta}}\leq \frac{C(\varepsilon )}{1+t}\|u(0)\|_{%
\mathcal{Z}_{\eta},_{\eta}}+\int_{0}^{t}\frac{c\varepsilon C(\varepsilon )}{\sqrt{%
t-s}(1 + \sqrt{t-s})}\|u(s)\|_{\mathcal{Z}_{\eta}}^{2} \, ds.  \label{estim nonlin a}
\end{equation}%
Let us define 
\begin{equation*}
|||u|||_{t}=\underset{s\in \lbrack 0,t]}{\sup } \Bigl\|(1+s)u(s) \Bigr\|_{\mathcal{Z}_\eta},
\end{equation*}%
then%
\begin{equation}
\|(1+t^{ })u(t)\|_{\mathcal{Z}_\eta} \leq C(\varepsilon )%
\|u(0)\|_{\mathcal{Z}_{\eta},_{\eta}}+K(\varepsilon )
|||u|||_{t}^{2},  \label{ineg |||u|||}
\end{equation}%
where 
\begin{equation}
\int_{0}^{t}\frac{c\varepsilon C(\varepsilon ) (1 + t)}{\sqrt{t-s}%
(1+\sqrt{t-s})(1 + s^2)} ds \leq K(\varepsilon ),
\label{estim integ}
\end{equation}%
which can be checked by splitting the integral in $s < t/2$ and $s > t/2$.
In view of (\ref{ineg |||u|||}), if $\|u(0)\|_{\mathcal{Z}_{\eta},_{\eta}}$ is small enough, then $\|(1+t^{ })u(t)\|_{\mathcal{Z}_\eta}$ is bounded uniformly in time, by a constant depending on 
$\|u(0)\|_{\mathcal{Z}_{\eta},_{\eta}}$.
The Lemma is proved. 
\end{proof}

%%%%%%%%%%%%%%%%%%%%%%%%%%%%%%%%%%%%%%%%%%%%%%%%%%%%%%%%%%%%%%%%%%%%%%%%%%%%%%%%%%%%%%%%%

\section{Appendix}

%%%%%%%%%%%%%%%%%%%%%%%%%%%%%%%%%%%%%%%%%%%%%%%%%%%%%%%%%%%%%%%%%%%%%%%%%%%%%%%%%%%%%%%%%

%%%%%%%%%%%%%%%%%%%%%%%%%%%%%%%%%%%%%%%%%%%%%%%%%%%%%%%%%%%%%%%%%%%%%%%%%%

\subsection{Proof of Lemma \protect\ref{lem:decomp Hk}\label{App:decompHk}}

%%%%%%%%%%%%%%%%%%%%%%%%%%%%%%%%%%%%%%%%%%%%%%%%%%%%%%%%%%%%%%%%%%%%%%%%%%

The Helmholtz decomposition is very classical, however, as our function spaces are not standard,
we have to detail it. We take the Fourier transform in the horizontal variable of the decomposition 
\begin{align*}
\widetilde{u} =\widetilde{v}+\nabla \phi , \quad
\Delta \phi =\nabla \cdot \widetilde{u}, \quad \frac{\partial \phi }{%
\partial y}|_{y=0}=\widetilde{u}^{y}|_{y=0},
\end{align*}%
which gives
\begin{eqnarray*}
(D^{2}-n^{2}\alpha ^{2})\phi _{n} &=&in\alpha u_{n}^{x}+Du_{n}^{y}=g_{n}, \\
D\phi _{n}(0) &=&\widetilde{u}_{n}^{y}(0)=h_{n},
\end{eqnarray*}%
where, by definition%
\begin{eqnarray}
|u_{n}(y)| \leq M_{n}e^{-\eta y},  \qquad \label{estim u n}
\|u\|_{L_{\eta}^{2}}^{2} =\sum_{|n|\geq 1}M_{n}^{2}<\infty ,  \notag
\end{eqnarray}%
and $Du_{n}^{y}$ is defined in the distribution sense. We easily obtain  (for $n>0$) 
\begin{equation*}
\phi _{n}(y)=-\frac{1}{2n\alpha }\int_{0}^{\infty }g_{n}(\tau )e^{-n\alpha
|\tau -y|}d\tau -\frac{1}{2n\alpha }\int_{0}^{\infty }g_{n}(\tau
)e^{-n\alpha (\tau +y)}d\tau -\frac{h_{n}}{n\alpha }e^{-n\alpha y}.
\end{equation*}%
Replacing $g_{n}$ by its expression gives, after an integration by parts,
\begin{eqnarray*}
\phi _{n}(y) &=&\frac{1}{2}\int_{0}^{y}(u_{n}^{y}-iu_{n}^{x})(\tau
)e^{n\alpha (\tau -y)}d\tau -\frac{1}{2}\int_{y}^{\infty
}(u_{n}^{y}+iu_{n}^{x})(\tau )e^{n\alpha (\tau -y)}d\tau \\
&&-\frac{1}{2}\int_{0}^{\infty }(u_{n}^{y}+iu_{n}^{x})e^{-n\alpha (\tau
+y)}d\tau .
\end{eqnarray*}%
We observe that $\phi _{n}\in C^{1}( \mathbb{R} ^{+})$ with%
\begin{eqnarray*}
D\phi _{n}(y) &=&\frac{\alpha n}{2}\int_{0}^{y}(iu_{n}^{x}-u_{n}^{y})(\tau
)e^{n\alpha (\tau -y)}d\tau -\frac{\alpha n}{2}\int_{y}^{\infty
}(iu_{n}^{x}+u_{n}^{y})(\tau )e^{n\alpha (\tau -y)}d\tau \\
&&+\frac{\alpha n}{2}\int_{0}^{\infty }(u_{n}^{y}+iu_{n}^{x})(\tau
)e^{- n\alpha (\tau +y)}d\tau -u_{n}^{x}(y),
\end{eqnarray*}%
and using (\ref{estim u n}), provided that $\eta<\alpha /2,$ $\phi _{n}$ is
such that the following estimates hold, with $c$ independent of $n$%
\begin{equation*}
|\phi _{n}(y)e^{\eta y}|\leq c\frac{M_{n}}{n},\quad |D\phi _{n}(y)e^{\eta y}|\leq
cM_{n}.
\end{equation*}%
The result is that $\nabla \phi \in L_{\eta}^{2}$ and $\widetilde{v}\in \mathring{L}_{\eta}^{2},$ with
$\|\widetilde{v}\|_{\mathring{L}_{\eta}^{2}}\leq c\|u\|_{L_{\eta}^{2}}$.
The result on the component in $\mathring{L}^{\infty }$ is straightforward.

%%%%%%%%%%%%%%%%%%%%%%%%%%%%%%%%%%

\subsection{Proof of Lemma \protect\ref{lem:spectrumL0 expky}}

%%%%%%%%%%%%%%%%%%%%%%%%%%%%%%%%%%

\label{proof2}

%%%%%%%%%%%%%%

\subsubsection{Resolvent estimate and spectrum in $\mathring{L}^{\infty }$}

%%%%%%%%%%%%%%

For the part of the linear operator in $\mathring{L}^{\infty },$ we need to
solve, for $\lambda \notin (-\infty ,0],$%
\begin{eqnarray*}
\nu D^{2}v_{0}^{x}-\lambda v_{0}^{x} &=&f_{0}\in L^{\infty }(\mathbb{R}^{+})
\\
v_{0}^{x}(0) &=&0,\text{ }v_{0}^{x}\in W^{2,\infty }.
\end{eqnarray*}%
Let us define $s$ such that $\Re s>0$ and $s^{2}=\lambda /\nu _{},$ then%
\begin{eqnarray*}
v_{0}^{x}(y) &=&-\frac{1}{2s\nu }\int_{y}^{\infty }f_{0}(\tau )e^{s(y-\tau
)}d\tau -\frac{1}{2s\nu }\int_{0}^{y}f_{0}(\tau )e^{-s(y-\tau )}d\tau \\
&&+\frac{1}{2s\nu }\int_{0}^{\infty }f_{0}(\tau )e^{-s(y+\tau )}d\tau .
\end{eqnarray*}%
We check that $v_{0}^{x}\in W^{2,\infty },$
and 
\begin{equation*}
|v_{0}^{x}(y)|\leq \frac{\|f_{0}\|_{L^{\infty} }}{|s| \, |\Re s|\nu}.
\end{equation*}%
Moreover, for $\arg \lambda =\theta \in \lbrack 0,\pi )$, if $\Re \lambda <0$,
\begin{eqnarray*}
\frac{1}{|s| \, |\Re s|\nu } = {1 \over \nu |s|^2 \cos (\theta/2)} = {1 \over |\lambda| \cos(\theta/2)}.
\end{eqnarray*}%
If $\Re \lambda \ge 0$, then we have
\begin{eqnarray*}
\frac{1}{|s| \, |\Re s|\nu } 
\leq \frac{\sqrt{2}}{|\lambda |}
\end{eqnarray*}%
since in this case, $0 \le \theta \le \pi / 2$.
This proves the first estimate in Lemma \ref{lem:spectrumL0 expky}.

Let us now concentrate on the case $\lambda <0$. Then we have%
% \begin{eqnarray}
% D^{2}v_{0}^{x}-\frac{\lambda }{\nu }v_{0}^{x} &=&\frac{1}{\nu }f_{0}\in
% L^{\infty }(\mathbb{R}^{+})  \notag \\
% v_{0}^{x}(0) &=&0\text{, }v_{0}^{x}\in W^{2,\infty }(\mathbb{R}^{+}).  \notag
% \end{eqnarray}%
% This gives%
\begin{equation*}
v_{0}^{x}(y)=A\sin sy+\frac{1}{s\nu }\int_{0}^{y}f_{0}(\tau )\sin s(y-\tau) \, d\tau ,
\end{equation*}%
where $s^2 = - \lambda / \nu \ge 0$ and where
 $A$ is arbitrary. Let us now show that the range of $\boldsymbol{L}_{(\nu)}^{(0)}-\lambda \mathbb{I}$ is not closed.
Indeed, let us choose $f_{0}\in L^{\infty }(\mathbb{R}^{+})$ such that $f_{0}(y) = \chi(y)/y$, where 
$\chi(y)$ is the characteristic function of the set%
\begin{equation*}
\cup _{n\geq 1}[\frac{2\pi n}{s},\frac{2\pi n+\pi /2}{s}].
\end{equation*}%
We obtain%
\begin{equation*}
\int_{0}^{y}f_{0}(\tau )\sin s(y-\tau )d\tau =\sin sy\int_{0}^{y}\chi \frac{%
\cos s\tau }{\tau }d\tau -\cos sy\int_{0}^{y}\chi \frac{\sin s\tau }{\tau }%
d\tau ,
\end{equation*}%
and%
\begin{align*}
\sum_{1\leq n\leq N(y)}\frac{s}{(2\pi n+\pi /2)} &< \int_{0}^{y}\chi \frac{%
\cos s\tau }{\tau }d\tau <\sum_{1\leq n\leq N(y)}\frac{s}{(2\pi n)} \\
\sum_{1\leq n\leq N(y)}\frac{s}{(2\pi n+\pi /2)} &< \int_{0}^{y}\chi \frac{%
\sin s\tau }{\tau }d\tau <\sum_{1\leq n\leq N(y)}\frac{s}{(2\pi n)}
\end{align*}%
with 
$\lbrack sy]=2\pi N(y)$, $[\cdot ]$ being the integer part.
As $y\rightarrow \infty$, series on both sides diverge, since the function $%
\int_{0}^{y}f_{0}(\tau )\sin s(y-\tau )d\tau $ behaves as%
\begin{equation*}
(\sin sy-\cos sy)\sum_{1\leq n\leq N(y)}\frac{s}{(2\pi n)},
\end{equation*}%
hence the limit is not in $W^{2,\infty }(\mathbb{R}^{+}),$ showing that $f_{0}$ is not
in the range of $\boldsymbol{L}_{(\nu )}^{(0)}-\lambda \mathbb{I}.$ Now consider
the sequence $\{f^{(N)}\}_{N\in \mathbb{N}}$ defined by%
\begin{equation*}
f^{(N)}(y)=f_{0}(y)\chi _{\lbrack 0,N]}.
\end{equation*}%
It is clear that $\{f^{(N)}\}$ is a Cauchy sequence in $L^{\infty }(\mathbb{R%
}^{+})$ since for $M>N$%
\begin{equation*}
\|f^{(N)}-f^{(M)}\|_{L^{\infty }}\leq \frac{1}{N}
\end{equation*}%
and this series converges in $L^{\infty }(\mathbb{R}^{+})$ towards $f_{0}.$
Moreover the functions $f^{(N)}(y)$ lie in the range of $\boldsymbol{L}%
_{(\nu )}^{(0)}-\lambda \mathbb{I}$ \ since%
 \begin{equation*}
v_{0N}^{x}(y)=\frac{1}{s\nu }\int_{0}^{N}f_{0}(\tau )\sin s(y-\tau
)d\tau,
\end{equation*}%
is the solution of ($\boldsymbol{L}_{(\nu )}^{(0)}-\lambda \mathbb{I)}%
v_{0N}=f^{(N)}$ with $v_{0N}^{x}\in W^{2,\infty }$. Since $f_{0}$ is not in
the range, this shows that the range is not closed.

For $\lambda =0,$ we obtain%
\begin{equation*}
v_{0}^{x}(y)=-\frac{1}{\nu }\Bigl[\int_{0}^{y}\tau f_{0}(\tau )d\tau
+y\int_{y}^{\infty }f_{0}(\tau )d\tau \Bigr]
\end{equation*}%
which shows that $0$ is not eigenvalue. Moreover choosing $f_{0}(y)=1/(1+y)$
in $L^{\infty }$ but not in the range, and choosing a Cauchy sequence $%
\{f^{(N)}=f_{0}\chi _{\lbrack 0,N]}\}$ converging to $f_{0}$ in $L^{\infty }$
and sitting in the range of $\boldsymbol{L}_{(\nu )},$ shows that the range
is not closed. Hence $\lambda =0$ lies in the essential spectrum.

%%%%%%%%%%%%%%%%%%%%

\subsubsection{Resolvent estimates in $\mathring{L}_{\eta}^{2}$}\label{app 8.2.2}

%%%%%%%%%%%%%%%%%%%%

Let us now study the linear system
\begin{equation*}
(\boldsymbol{L}_{(\nu )}^{(0)}-\lambda )v=f\in \mathring{L}_{\eta}^{2},
\end{equation*}%
namely Stokes' equation, where we look for $v\in \mathring{H}_{\eta}^{2}$. 
Using the Fourier series for $f$ and $v$ leads for $n\neq 0$ to%
\begin{eqnarray} 
\nu \Bigl( D^{2}-n^{2}\alpha ^{2}-\frac{\lambda }{\nu } \Bigr) v_{n}+\binom{in\alpha }{D}%
q_{n} &=&\binom{f_{n}}{g_{n}},  \label{equ for resolv} \\
in\alpha v_{n}^{x}+Dv_{n}^{y} &=&0, \notag
\end{eqnarray}%
with boundary condition $v_{n}|_{y=0}=0$, where 
\begin{equation*}
in\alpha f_{n}+Dg_{n}=0.
\end{equation*}%
By definition of $\mathring{L}^2_\eta$, $\|(f_{n},g_{n})\|_{C_{\eta}^{0}}\leq M_{n}$ with $\sum_{n\neq
0}M_{n}^{2}=\|(f,g)\|_{L_{\eta}^{2}}^{2}$.
This leads to%
\begin{eqnarray}
(D^{2}-n^{2}\alpha ^{2}) \Bigl[ D^{2}- \Bigl( n^{2}\alpha ^{2}+\frac{\lambda }{\nu } \Bigr) \Bigr]v_{n}^{y} 
&=&\frac{1}{\nu }(D^{2}-n^{2}\alpha ^{2})g_{n},
\label{equ for resolv2} \\
v_{n}^{y} &=&Dv_{n}^{y}=0\text{ for }y=0,  \notag
\end{eqnarray}%
leading to explicit expressions for $v_{n}^{x}$ and $v_{n}^{y}.$ Let 
\begin{equation*}
s_{n}^{2}=n^{2}\alpha ^{2}+\frac{\lambda }{\nu }.
\end{equation*}% 
We obtain%
\begin{equation}
v_{n}^{y}(y)=w_{n}(y)-\frac{w_{n}(0)}{2} \Bigl[ \frac{n\alpha +s_{n}}{n\alpha -s_{n}%
}(e^{-s_{n}y}-e^{-n\alpha y})+e^{-s_{n}y}+e^{-n\alpha y} \Bigr],  \label{vn(y)}
\end{equation}%
where%
\begin{equation}
w_{n}(y)=-\frac{1}{2\nu s_{n}}\int_{0}^{\infty }g_{n}(\tau )e^{-s_{n}|\tau
-y|}d\tau -\frac{1}{2\nu s_{n}}\int_{0}^{\infty }g_{n}(\tau )e^{-s_{n}(\tau
+y)}d\tau . \label{wn(y)}
\end{equation}%
% and where we use the fact that $w_{n}(y)$ satisfies 
% \begin{eqnarray*}
% \Bigl[ D^{2}-(n^{2}\alpha ^{2}+\frac{\lambda }{\nu }) \Bigr] w_{n} &=&\frac{1}{\nu 
% }g_{n}, \\
% Dw_{n}(0) &=&0.
% \end{eqnarray*}%
% Moreover, $v_{n}^{y}(y)$ satisfies $v_{n}^{y}(0)=Dv_{n}^{y}(0)=0.$
We have the following useful Lemma:

\begin{lemma} \label{lemme30}
There exists $C>0$ such that
for any $\delta$ with $0<\delta <\pi /6$, any $n\neq 0$ and any $\lambda
\in \mathbb{C}$ such that $0\leq \arg (\lambda +\nu \alpha ^{2})\leq 2\pi
/3-\delta$ 
we have
\begin{eqnarray*}
|s_{n}| \geq \Re s_n \geq C \Bigl( \alpha \sqrt{n^{2}-1}+|\lambda +\nu \alpha ^{2}|^{1/2} \Bigr). 
%\Re s_{n} &\geq &C(\alpha \sqrt{n^{2}-1}+|\lambda +\nu \alpha ^{2}|^{1/2}).
\end{eqnarray*}
\end{lemma}

\begin{proof} Using
$$
s_n^2 = (n^2 - 1) \alpha^2 + \frac{\lambda + \nu \alpha^2}{\nu},
$$
we have, with $\theta = \arg (\lambda +\nu \alpha ^{2})$,
\begin{equation*}
|s_{n}|^{4}=\Bigl[(n^{2}-1)\alpha ^{2}+\frac{\cos \theta }{\nu }|\lambda +\nu
\alpha ^{2}|\Bigr]^{2}+\frac{1}{\nu ^{2}}|\lambda +\nu \alpha ^{2}|^{2}\sin
^{2}\theta ,
\end{equation*}%
hence
\begin{equation*}
|s_{n}|^{4}=(n^{2}-1)^{2}\alpha ^{4}+\frac{1}{\nu ^{2}}|\lambda +\nu \alpha
^{2}|^{2}+2(n^{2}-1)\alpha ^{2}\frac{1}{\nu }|\lambda +\nu \alpha ^{2}|\cos
\theta ,
\end{equation*}%
which implies%
\begin{eqnarray*}
|s_{n}|^{4} &\geq &\Bigl[(n^{2}-1)^{2}\alpha ^{4}+\frac{1}{\nu ^{2}}|\lambda
+\nu \alpha ^{2}|^{2}\Bigr]\text{ for }\cos \theta \geq 0, \\
&\geq &(1-|\cos \theta |)\Bigl[(n^{2}-1)^{2}\alpha ^{4}+\frac{1}{\nu ^{2}}%
|\lambda +\nu \alpha ^{2}|^{2}\Bigr] \text{ for }\cos \theta <0.
\end{eqnarray*}%
Hence, in all cases%
\begin{equation}
|s_{n}|^{2}\geq \cos \frac{\theta}{2}\Bigl\{(n^{2}-1)\alpha ^{2}+\frac{1}{\nu }|\lambda
+\nu \alpha ^{2}|\Bigr\},  \label{bound sn^2}
\end{equation}%
and, as $\theta/2 \le \pi/3$,
$$
|s_{n}| \geq \frac{1}{2} \Bigl[(n^{2}-1)^{1/2}\alpha +\frac{1}{\nu ^{1/2}}%
|\lambda +\nu \alpha ^{2}|^{1/2}\Bigr] 
\geq C \Bigl( \alpha \sqrt{n^{2}-1}+|\lambda +\nu \alpha ^{2}|^{1/2} \Bigr),
$$
where the constant $C$ depends on $\nu$.
Now, we also have%
\begin{equation*}
(\Re s_{n})^{2}= \frac{1}{2} \Bigl[ \Re s_n^2 + |s_n|^2 \Bigr] 
= \frac{1}{2}\Bigl[(n^{2}-1)\alpha ^{2}+\frac{1}{\nu }|\lambda
+\nu \alpha ^{2}|\cos \theta +|s_{n}|^{2}\Bigr].
\end{equation*}%
Using (\ref{bound sn^2}), we obtain 
\begin{equation*}
(\Re s_{n})^{2}\geq \frac{1}{2}\Bigl[ \Bigl(1+\cos \frac{\theta}{2}\Bigr)(n^{2}-1)\alpha
^{2}+ \Bigl(\cos \frac{\theta}{2}+\cos \theta \Bigr)\frac{1}{\nu }|\lambda +\nu \alpha ^{2}|\Bigr].
\end{equation*}%
Since $\cos \theta /2+\cos \theta \geq \cos (\pi /3-\delta /2)-\cos (\pi
/3+\delta )>0,$ it results that, as above (adapting the constant $C),$%
\begin{equation*}
\Re s_{n}\geq C \Bigl(\alpha \sqrt{n^{2}-1}+|\lambda +\nu \alpha ^{2}|^{1/2} \Bigr),
\end{equation*}%
which ends the proof.
\end{proof}

Now, from (\ref{wn(y)}) we have the estimate%
\begin{equation*}
|w_{n}(y)e^{\eta y}|\leq \frac{3M_{n}}{2\nu |s_{n}|(\Re s_{n}-\eta)},
\end{equation*}%
and assuming that $\lambda $ is such that $|\lambda +\nu \alpha ^{2}|\geq \varepsilon _{0}^{2}$,
%(we need to check this condition for $\lambda \in \Gamma $ which is the contour for the integral giving the semi-group), 
we choose $\eta>0$ such that%
\begin{equation*}
\eta<C\frac{\varepsilon _{0}}{2},
\end{equation*}%
so that 
\begin{eqnarray*}
\Re s_{n}-\eta > C \Bigl(\alpha \sqrt{n^{2}-1}+|\lambda +\nu \alpha ^{2}|^{1/2}-%
\frac{\varepsilon _{0}}{2} \Bigr) 
>\frac{C}{2} \Bigl(\alpha \sqrt{n^{2}-1}+|\lambda +\nu \alpha ^{2}|^{1/2} \Bigr).
\end{eqnarray*}%
Finally we have (using $|n|\neq 0)$%
\begin{equation}
|w_{n}(y)e^{\eta y}|\leq \frac{cM_{n}}{(n^{2}-1)\alpha +|\lambda +\nu \alpha
^{2}|},  \label{w(y)}
\end{equation}%
hence%
\begin{equation*}
\|w\|_{L_{\eta}^{2}}\leq \frac{c_{1}}{|\lambda +\nu \alpha ^{2}|}%
\|g\|_{L_{\eta}^{2}}.
\end{equation*}%
Coming back to (\ref{vn(y)}) we observe that%
\begin{equation*}
\Bigl|\frac{n\alpha +s_{n}}{n\alpha -s_{n}} \Bigr|
=\frac{|s_{n}+n\alpha |^{2}}{|\alpha
^{2}-|\lambda +\nu \alpha ^{2}||}
\end{equation*}%
and assuming that $|\lambda +\nu \alpha ^{2}|\neq \alpha ^{2}$ when $\lambda 
$ is real , it is clear that%
\begin{equation*}
\frac{|s_{n}+n\alpha |^{2}}{|s_{n}|(\Re s_{n}-\eta)}\leq c_{2},\text{
independently of }n\text{ and }\lambda .
\end{equation*}%
It results that 
\begin{equation*}
\|\widetilde{v}^{y}\|_{L_{\eta}^{2}}\leq \frac{c_{1}}{|\lambda +\nu \alpha ^{2}|%
}\|g\|_{L_{\eta}^{2}}.
\end{equation*}%
Now we have%
\begin{equation*}
v_{n}^{x}=-\frac{Dv_{n}^{y}}{in\alpha },
\end{equation*}%
hence%
\begin{equation}
v_{n}^{x}(y)=-\frac{Dw_{n}(y)}{in\alpha }-iw_{n}(0)\frac{s_{n}}{n\alpha
-s_{n}}(e^{-n\alpha y}-e^{-s_{n}y}).  \label{vn^x(y)}
\end{equation}%
It is easy to check that, after an integration by parts,%
\begin{equation*}
Dw_{n}(y)=-\frac{1}{2\nu s_{n}}\int_{0}^{\infty }Dg_{n}(\tau
)e^{s_{n}|y-\tau |}d\tau +\frac{1}{2\nu s_{n}}\int_{0}^{\infty }Dg_{n}(\tau
)e^{-s_{n}(y+\tau )}d\tau ,
\end{equation*}%
and since%
\begin{equation*}
f_{n}=-\frac{Dg_{n}}{in\alpha },\text{ }|f_{n}(y)e^{\eta y}|\leq M_{n},
\end{equation*}%
we obtain%
\begin{equation}
\Bigl| \frac{Dw_{n}(y)}{in\alpha }e^{\eta y} \Bigr| \leq \frac{3M_{n}}{2\nu |s_{n}|(\Re
s_{n}-\eta)}.  \label{DW/n}
\end{equation}%
Now collecting (\ref{vn^x(y)}), (\ref{DW/n}), (\ref{w(y)}) for $w_{n}(0),$
and%
\begin{equation*}
\frac{1}{|n\alpha -s_{n}|(\Re s_{n}-\eta)}=\frac{|n\alpha +s_{n}|}{(\Re
s_{n}-\eta)|\alpha ^{2}-|\lambda +\nu \alpha ^{2}||}\leq c_{4},
\end{equation*}%
we obtain, with a constant $c$ independent of $\lambda $,%
\begin{equation*}
\|v\|_{L_{\eta}^{2}}\leq \frac{c}{|\lambda +\nu \alpha ^{2}|}%
\Bigl( \|f\|_{L_{\eta}^{2}}+\|g\|_{L_{\eta}^{2}} \Bigr),
\end{equation*}%
which is the result stated in Lemma \ref{lem:spectrumL0 expky}.

%%%%%%%%%%%%%%%%%%%%%%%%%%%%%%%%%%%%%%%%%%%%%%%

\subsection{Proof of Lemma \protect\ref{lem: locationspect}}

%%%%%%%%%%%%%%%%%%%%%%%%%%%%%%%%%%%%%%%%%%%%%%%

\label{proof3} 
We first prove that the linear operator $\boldsymbol{L}^{(1)}$ is relatively
bounded with respect to $\boldsymbol{L}_{(\nu )}^{(0)}$, namely that for $v\in \mathring{H}_{\eta}^{2}$%
\begin{equation}
\|\boldsymbol{L}^{(1)}v\|_{\mathring{L}_{\eta}^{2}}\leq a\|v\|_{\mathring{L}%
_{\eta}^{2}}+b\|\boldsymbol{L}_{(\nu )}^{(0)}v\|_{\mathring{L}_{\eta}^{2}},
\label{relat bound}
\end{equation}%
where $b$ can be chosen as small as we want. Indeed%
\begin{equation*}
(\boldsymbol{L}^{(1)}v)_{n}=-\Pi _{n}\left[ in\alpha Uv_{n}+U^{\prime }\binom{v_{n}^{y}}{0}\right],
\end{equation*}%
thus
\begin{equation*}
\|(\boldsymbol{L}^{(1)}v)_{n}\|_{C_{\eta}^{0}}\leq c(|n|+1)M_{n}
\end{equation*}%
where%
\begin{equation*}
\sup |\widetilde{v}_{n}(y)|e^{\eta  y}=M_{n},
\end{equation*}%
hence%
\begin{equation*}
\|(\boldsymbol{L}^{(1)}v)_{n}\|_{C_{\eta}^{0}}^{2}\leq 2c^{2} \Bigl(\varepsilon
^{2}n^{4}+\frac{1}{4\varepsilon ^{2}} \Bigr) M_{n}^{2},
\end{equation*}%
and since%
\begin{eqnarray*}
\|v\|_{\mathring{L}_{\eta}^{2}}^{2} &=&\sum_{|n|\geq 1}M_{n}^{2}, \\
\|v\|_{\mathring{H}_{\eta}^{2}}^{2} &=&\sum_{|n|\geq
1}n^{4}M_{n}^{2}+n^{2}\|Dv_{n}\|_{C_{\eta}^{0}}^{2}+\|D^{2}v_{n}\|_{C_{\eta}^{0}}^{2},
\end{eqnarray*}%
we obtain%
\begin{equation*}
\|\boldsymbol{L}^{(1)}v\|_{_{\mathring{L}_{\eta}^{2}}}^{2}\leq
2c^{2}\varepsilon ^{2}\|v\|_{\mathring{H}_{\eta}^{2}}^{2}+\frac{c^{2}}{%
2\varepsilon ^{2}}\|v\|_{\mathring{L}_{\eta}^{2}}^{2}.
\end{equation*}%
Thanks to the equivalence between $\|\boldsymbol{L}_{(\nu )}^{(0)}v\|_{%
\mathring{L}_{\eta}^{2}}$ and $\|v\|_{\mathring{H}_{\eta}^{2}}$ 
\begin{equation*}
\|v\|_{\mathring{H}_{\eta}^{2}}\leq K\|\boldsymbol{L}_{(\nu )}^{(0)}v\|_{%
\mathring{L}_{\eta}^{2}},
\end{equation*}%
we obtain (\ref{relat bound}) with $b=Kc\varepsilon $ and $a= c / 2 \varepsilon$
(which may be large).

Now, using (\ref{relat bound}) on $v = (\lambda - \boldsymbol{L}^{(0)}_{(\nu)}) w$,
together with the bounds obtained in Lemma \ref{lem:spectrumL0 expky}, we have%
\begin{eqnarray*}
&&a\Bigl\|(\lambda -\boldsymbol{L}_{(\nu )}^{(0)})^{-1} \Bigr\|_{\mathcal{L}(\mathring{L}%
_{\eta}^{2})}+b \Bigl\| \boldsymbol{L}_{(\nu )}^{(0)}(\lambda -\boldsymbol{L}_{(\nu )}^{(0)})^{-1} \Bigr\|_{\mathcal{L}%
(\mathring{L}_{\eta}^{2})} \\
&\leq &\frac{cC}{2\varepsilon |\lambda +\nu \alpha ^{2}|)}+Kc\varepsilon \Bigl(1+%
\frac{C|\lambda |}{|\lambda +\nu \alpha ^{2}|} \Bigr),
\end{eqnarray*}%
provided 
\begin{equation*}
0\leq \arg (\lambda +\nu \alpha ^{2})\leq 2\pi /3-\delta.
\end{equation*}%
It is clear that 
\begin{equation*}
a \Bigl\|(\lambda -L_{(\nu )}^{(0)})^{-1} \Bigr\|_{\mathcal{L}(\mathring{L}%
_{\eta}^{2})}+b \Bigl\|L_{(\nu )}^{(0)}(\lambda -L_{(\nu )}^{(0)})^{-1} \Bigr\|_{\mathcal{L}%
(\mathring{L}_{\eta}^{2})}<1
\end{equation*}%
for $\varepsilon $ such that 
$$
Kc\varepsilon \Bigl( 2+\frac{\nu\alpha^{2}\varepsilon }{cC} \Bigr) < \frac{1}{2}
$$
and imposing
\begin{equation}
|\lambda +\nu \alpha ^{2}|> \frac{cC}{\varepsilon}.  \label{cond sur lambda}
\end{equation}%
Applying the theorem 3.17 p.214 in \cite{Kato}, we deduce that the set of $%
\lambda $ satisfying (\ref{cond sur lambda}) is included in the resolvent
set of $\boldsymbol{L}_{(\nu )}^{(0)}+\boldsymbol{L}^{(1)}=\boldsymbol{L}%
_{(\nu )}$ when restricted to the subspace $\mathring{L}_{\eta}^{2}.$

Now in the subspace $\mathring{L}^{\infty }$, we have $\boldsymbol{L}%
^{(1)}|_{\mathring{L}^{\infty }}=0$, hence the result holds.
This means that the spectrum of eigenvalues of $\boldsymbol{L}_{(\nu )}$ is
located in a right bounded region as indicated in the Lemma.

%%%%%%%%%%%%%%%%%%%%%%%%%%%%%%%%%%%%

\subsection{Proof of Lemma \protect\ref{relatcomp}}

%%%%%%%%%%%%%%%%%%%%%%%%%%%%%%%%%%%%

\label{proof4} Let us consider a sequence $\{v^{(n)}\}_{n\in \mathbb{N}%
}\subset \mathcal{Z}_{\eta}$ such that $v^{(n)}$ and $\boldsymbol{L}_{(\nu
)}^{(0)}v^{(n)}$ are bounded in $\mathcal{X}_{\eta},$ then the linear operator $%
\boldsymbol{L}^{(1,c)}$ is relatively compact with respect to $\boldsymbol{L}%
_{(\nu )}^{(0)}$ if there exists a subsequence $\{v^{(p_{n})}\}_{p_{n}\in 
\mathbb{N}}$ such that $\boldsymbol{L}^{(1,c)}v^{(p_{n})}$ converges in $%
\mathcal{X}_{\eta}.$

We know that $\boldsymbol{L}^{(1,c)}\in \mathcal{L}(\mathcal{Z}_\eta,\mathcal{Y}_\eta)$
cancels on $\mathcal{L}(\mathring{W}^{2,\infty },\mathring{W}%
^{1,\infty }),$ so it is sufficient to work in $\mathcal{L}(\mathring{H}%
_{\eta}^{2},\mathring{H}_{\eta}^{1}).$

Here we use an additional property of $\boldsymbol{L}^{(1,c)}:$ the function $%
U(y)$ tends exponentially to $U_+$ like $e^{-\gamma y}$ as $y\rightarrow \infty$. 
The $m-th$ Fourier component 
\begin{equation*}
\Bigl[ ((U-U_+)\cdot \nabla )v+(v\cdot \nabla )U \Bigr]_{m}(y)=(U(y)-U_+)im\alpha
v_{m}(y)+v_{m}^{y}(y)U^{\prime }(y)
\end{equation*}%
is bounded by%
\begin{equation*}
Ke^{-(\gamma +\eta)y} \Bigl( 1+|m| \Bigr) |v_{m}(y)e^{\eta y}|\leq Ke^{-(\gamma
+\eta)y} \Bigl( 1+|m| \Bigr) \|v_{m}\|_{C_{\eta}^{0}},
\end{equation*}%
where%
\begin{equation*}
\sum_{|m|\geq 1}(1+m^{2}) \|v_{m}\|_{C_{\eta}^{0}}^{2}=\|v\|_{\mathring{H}%
_{\eta}^{1}}^{2}.
\end{equation*}%
It appears, from the properties of the projection $\Pi $ described in
Appendix \ref{App:decompHk}, that 
\begin{equation*}
\boldsymbol{L}^{(1,c)}\in \mathcal{L}(\mathring{H}_{\eta}^{1},\mathring{L}%
_{\eta+\gamma }^{2})\cap \mathcal{L}(\mathring{H}_{\eta}^{2},\mathring{H}%
_{\eta+\gamma }^{1}).
\end{equation*}%
Let us show that the identity map $\mathring{H}_{\eta+\gamma
}^{1}\hookrightarrow \mathring{L}_{\eta}^{2}$ is compact, which is sufficient
for our purpose.

We define the space $\mathring{H}_{\eta,N}^{1},$ and similarly $\mathring{L}%
_{\eta,N}^{2},$ by%
\begin{eqnarray*}
\mathring{H}_{\eta,N}^{1} =\Bigl\{ \widetilde{v}\in \lbrack H^{1}[(\mathbb{T},C_{\eta}^{0}(0,N)]]^{2}; \,
\int_{0}^{2\pi }\widetilde{v}(x,y)dx=0, \, \nabla \cdot \widetilde{v}=0,%
\, \widetilde{v}^{y}|_{y=0}=0\Bigr\} .
\end{eqnarray*}%
Let us consider a sequence $\{v^{(n)}\}_{n\in \mathbb{N}}$, bounded in $%
\mathring{H}_{\eta+\gamma }^{1},$ and the corresponding sequence $%
\{v^{(n)N}\}_{n\in \mathbb{N}}$ in $\mathring{H}_{\eta+\gamma ,N}^{1},$ where%
\begin{equation*}
v_{m}^{(n)N}(y)=\left\{ 
\begin{array}{c}
v_{m}^{(n)}(y)\text{ for }y\in \lbrack 0,N] \\ 
0\text{ for }y>N%
\end{array}%
\right. .
\end{equation*}%
Then%
\begin{equation*}
m^{2}\|v_{m}^{(n)}-v_{m}^{(n)N}\|_{C_{\eta}^{0}}^{2}+\|Dv_{m}^{(n)}-Dv_{m}^{(n)N}\|_{C_{\eta}^{0}}^{2}\leq 4M_{m}^{2}e^{-2\gamma N},
\end{equation*}%
where 
\begin{equation*}
M_{m}^{2}=m^{2}\|v_{m}^{(n)}\|_{C_{\eta}^{0}}^{2}+\|Dv_{m}^{(n)}\|_{C_{\eta}^{0}}^{2}.
\end{equation*}%
Since the interval $[0,N]$ is bounded, and the functions are bounded and
equicontinuous in $y,$ there is a subsequence $\{v^{(p_{n}^{N})}\}_{n}$
converging in $\mathring{L}_{\eta,N}^{2}$ when $n\rightarrow \infty .$

Now, for any $N$, we have%
\begin{equation*}
\|v^{(p_{n}^{N})}-v^{(p_{q}^{N})}\|_{\mathring{L}_{\eta,N}^{2}}\leq Ke^{-\gamma
N}\varepsilon _{p_{n}}\text{ for }q>p,\text{ }
\end{equation*}%
and $\varepsilon _{p_{n}}\rightarrow 0$ as $n\rightarrow \infty $ since $%
\{v^{(p_{n}^{N})}\}_{_{n}}$ is a Cauchy sequence in $\mathring{L}_{\eta,N}^{2}$. 
Now extracting the diagonal subsequence $\{v^{(p_{N}^{N})}\}_{N\rightarrow
\infty }$, we obtain a Cauchy sequence $\{v^{(p_{N}^{N})}\}$ in $\mathring{L}%
_{\eta}^{2},$ which thus converges.

%%%%%%%%%%%%%%%%%%%%%%%%%%%%%%%%%%%%%%%%%%%%%%%%%%%%%%%%%%

\subsection{Proof of Lemma \ref{newlemma}
\label{newappendix}}

%Spectrum of $\protect\nu \Pi \Delta -\frac{\protect\omega }{%
%\protect\alpha }\frac{\partial }{\partial \protect\xi }$ in $\mathring{L}%
%_{\eta}^{2}$ 

%%%%%%%%%%%%%%%%%%%%%%%%%%%%%%%%%%%%%%%%%%%%%%%%%%%%%%%%%%

To prove Lemma \ref{newlemma}, let us proceed as in Appendix
\ref{app 8.2.2}. We have explicitly the solution $v$ of 
\begin{equation*}
\nu \Pi \Delta v- \Pi \Bigl[ U_+\frac{\partial v}{\partial x } \Bigr]
-\lambda v=f\in \mathring{L}_{\eta}^{2},\text{ }v(0)=0,
\end{equation*}%
with formulas (\ref{equ for resolv}), (\ref{equ for resolv2}) in Appendix \ref{app 8.2.2}. The
result is based on the good estimates given at Lemma \ref{lemme30}, for $%
s_{n}$ and $\Re s_{n}-\eta $ where now
\begin{equation*}
s_{n}^{2}=n^{2}\alpha ^{2}+\frac{\lambda -i n \alpha U_+ }{\nu },\text{ }\Re
s_{n}>0,
\end{equation*}%
and $\eta >0$ is small enough. We first observe that if $\lambda $ is such that $%
s_{n}=0$ or $\Re s_{n}=\eta $ then $\lambda $ is situated in a very specific
region which can be avoided for a finite number of values $1\leq |n|\leq
N_{0}.$ Then it will be sufficient to obtain lower estimates for $|s_{n}|$
and $\Re s_{n}-\eta $ for values of $n$ such that $|n|>N_{0}.$

First consider $\lambda $ such that $s_{n}=0.$ This implies%
\begin{equation*}
\lambda _{r}=-\nu \alpha ^{2}n^{2},\text{ }\lambda _{i}=n\alpha U_+ ,
\end{equation*}%
hence%
\begin{equation}
\lambda _{r}=-\frac{\nu }{U_+^{2}}(\lambda _{i})^{2},
\label{parabolaP1}
\end{equation}%
which means that $\lambda $ belongs to a parabola $P_{1},$ truncated by $%
\lambda _{r}\leq -\nu \alpha ^{2}$ since $|n|\geq 1.$

Now consider $\lambda $ such that $\Re s_{n}=\eta .$ This implies%
\begin{equation*}
A+\sqrt{A^{2}+B^{2}}=2\eta ^{2},
\end{equation*}%
with%
\begin{equation*}
A=n^{2}\alpha ^{2}+\frac{\lambda _{r}}{\nu },\text{ }B=\frac{\lambda
_{i}-n\alpha U_+ }{\nu }.
\end{equation*}%
Hence%
\begin{equation*}
B^{2}=4\eta ^{4}-4A\eta ^{2},
\end{equation*}%
which leads to%
\begin{equation*}
(\lambda _{i}-n\alpha U_+ )^{2}=4\nu ^{2}\eta ^{4}-4\nu \eta ^{2}(\lambda _{r}+\nu
\alpha ^{2}n^{2})
\end{equation*}%
\begin{equation*}
\lambda _{r}=-\frac{(\lambda _{i}-n\alpha U_+ )^{2}}{4\nu \eta ^{2}}-\nu (\alpha
^{2}n^{2}-\eta ^{2})
\end{equation*}%
which is a set of parabolas in $\mathbb{C} $ with the following envelope
(varying the parameter $|n|\geq 1)$%
\begin{equation*}
\lambda _{r}=-\frac{\nu }{U_+^{2}+4\nu ^{2}\eta ^{2}}%
\lambda _{i}^{2}+\nu \eta ^{2}.
\end{equation*}%
It results that if $\lambda $ is such that $\Re s_{n}=\eta $ then 
\begin{equation}
\lambda _{r}\leq -\frac{\nu }{U_+^{2}+4\nu ^{2}
\eta ^{2}}\lambda _{i}^{2}+\nu \eta ^{2},\text{ }\lambda _{r}\leq -\nu (\alpha
^{2}-\eta ^{2}),  \label{parabolaP2}
\end{equation}%
which corresponds again to a parabolic region, truncated by $\lambda
_{r}\leq -\nu (\alpha ^{2}-\eta ^{2})<0$ for $\eta <\alpha .$ It is clear that the
parabola (\ref{parabolaP1}) is included in this region.

Let us choose $\lambda $ outside of region (\ref{parabolaP2}) and such that 
\begin{equation*}
0\leq \arg (\lambda +\nu \alpha ^{2})\leq 2\pi /3-\delta ,
\end{equation*}%
and let us follow the method used in Appendix \ref{app 8.2.2}. We have for $%
\cos \theta >\frac{C}{N_{0}}$ where $\frac{|n U_+ |}{(n^{2}-1)\alpha ^{}}%
\leq \frac{C}{_{N_{0}}}$, and for $|n|\geq N_{0},$ 
\begin{equation*}
|s_{n}|^{4}\geq \Bigl[ (n^{2}-1)^{2}\alpha ^{4}+\frac{1}{\nu ^{2}}|\lambda
+\nu \alpha ^{2}|^{2} \Bigr]+\frac{n^{2}\alpha^2 U_+^{2}}{\nu ^{2}},
\end{equation*}%
and for $\cos (2\pi /3-\delta )\leq \cos \theta \leq \frac{C}{N_{0}},$%
\begin{equation*}
|s_{n}|^{4}\geq (1-|\cos \theta |-\frac{C}{N_{0}}) \Bigl[ (n^{2}-1)^{2}\alpha ^{4}+%
\frac{1}{\nu ^{2}}|\lambda +\nu \alpha ^{2}|^{2} \Bigr]+\frac{n^{2}\alpha^2 U_+^{2}}{%
\nu ^{2}},
\end{equation*}%
so that in all cases%
\begin{equation*}
|s_{n}|^{2}\geq (\cos \theta /2-\delta _{1}) \Bigl[(n^{2}-1)\alpha ^{2}+\frac{1}{%
\nu }|\lambda +\nu \alpha ^{2}| \Bigr],
\end{equation*}%
where $\delta _{1}$ is such that%
\begin{equation*}
(\cos ^{2}\theta /2-\frac{C}{2N_{0}})^{1/2}=\cos \theta /2-\delta _{1},\text{
}\delta _{1}<\frac{C}{N_{0}},
\end{equation*}%
using $\cos \theta /2>\cos (\pi /3-\delta /2)>1/2.$

Now we have 
\begin{equation*}
(\Re s_{n})^{2}\geq \frac{1}{2}\Bigl\{(1+\cos \theta /2-\delta
_{1})(n^{2}-1)\alpha ^{2}+(\cos \theta /2+\cos \theta -\delta _{1})\frac{1}{%
\nu }|\lambda +\nu \alpha ^{2}|\Bigr\}.
\end{equation*}%
Moreover, we have
\begin{eqnarray*}
\cos \theta /2+\cos \theta -\delta _{1} &\geq &\cos (\pi /3-\delta /2)-\cos
(\pi /3+\delta )-\delta _{1} \\
&\geq &\frac{\sqrt{3}\delta }{4}-\delta _{1}\text{ for }\delta \text{ small
enough.}
\end{eqnarray*}%
Hence choosing $N_{0}$ large enough, such that 
\begin{equation*}
\frac{C}{N_{0}}<\frac{\sqrt{3}\delta }{8},
\end{equation*}%
we obtain the required estimates of Lemma \ref{lemme30} in Appendix \ref{app 8.2.2}. 

Now, the inverse of $\boldsymbol{L}_{(\nu )}^{(0)}
+ \boldsymbol{L}^{(1,0)} - \lambda$ is bounded 
in $\mathring{L}^2_\eta$ by $C / |\lambda + \nu \alpha^2|$
provided that we adapt the constant $C$ in order to take
care of the Fourier components with $|n| < N_0$.

The Lemma \ref{newlemma} is then proved.

%%%%%%%%%%%%%%%%%%%%%%%%%%%%%%%%%%%%%%%%%%%%%%%%%%%%

\subsection{Proof of Lemma \protect\ref{lem:estimsemigroup poids}\label{app:
lem semigr poids}}

%%%%%%%%%%%%%%%%%%%%%%%%%%%%%%%%%%%%%%%%%%%%%%%%%%%%

Let $f\in L_{\eta}^{\infty }$. Then $v= ( \lambda -\boldsymbol{L}_{(\nu )}^{(0)})^{-1}f$
is explicitely given by 
\begin{eqnarray}\label{v(y) Linfty}
v(y) &=&\frac{1}{2s\nu }\int_{y}^{\infty }f(\tau )e^{s(y-\tau )}d\tau +\frac{%
1}{2s\nu }\int_{0}^{y}f(\tau )e^{-s(y-\tau )}d\tau  \\
&&-\frac{1}{2s\nu }\int_{0}^{\infty }f(\tau )e^{-s(y+\tau )}d\tau .\notag
\end{eqnarray}%
We want to estimate
$\|e^{\boldsymbol{L}_{(\nu )}^{(0)}t}f\|_{L^{\infty }}$,
using 
\begin{equation}
e^{\boldsymbol{L}_{(\nu )}^{(0)}t}f=\frac{1}{2i\pi }\int_{\Gamma }e^{\lambda
t} \Bigl(\lambda -\boldsymbol{L}_{(\nu )}^{(0)} \Bigr)^{-1}fd\lambda ,
\label{integ for semi group}
\end{equation}%
where the contour $\Gamma $ is detailed on Figure \ref{fig:gamma}. 

The estimate obtained in Lemma \ref%
{lem:spectrumL0 expky} is not sufficient, and does not use the decay at $%
\infty $ of $f(y).$ Now we have, with $s^{2}=\lambda /\nu ,$ $\Re s>0$%
\begin{equation}
\|(\lambda -\boldsymbol{L}_{(\nu )}^{(0)})^{-1}f\|_{L^{\infty }}\leq \frac{3%
}{2\nu |s||\Re s-\eta|}\|f\|_{L_{\eta}^{\infty }}\text{ if }|\Re s-\eta|\neq 0,
\label{estim resolv Linfty}
\end{equation}%
\begin{equation*}
\|(\lambda -\boldsymbol{L}_{(\nu )}^{(0)})^{-1}f\|_{L^{\infty }}\leq \frac{3%
}{2\nu |s|\eta}\|f\|_{L_{\eta}^{\infty }}\text{ for } \frac{\eta}{2}\leq \Re s,
\end{equation*}%
where we notice that $\Re s-\eta$ may be negative. 
In fact, we use estimate (\ref%
{estim resolv Linfty}) only for large $|\lambda |.$ For bounded $|\lambda |$
we use next result.

\subsubsection{Special estimates for bounded $\protect\lambda $}

Let us come back to 
\begin{equation}
\lambda u-\nu D^{2}u=f\in L_{\eta}^{\infty },\text{ }v(0)=0.
\label{equ small lambda}
\end{equation}%
and assume that $\lambda $ is bounded. We first solve%
\begin{equation*}
\nu D^{2}v=-f,
\end{equation*}%
looking for $v$ such that $v\rightarrow 0$ as $y\rightarrow \infty .$ We find%
\begin{equation*}
v(y)=-\int_{y}^{\infty }ds\int_{s}^{\infty }\frac{f(\tau )}{\nu }d\tau ,
\end{equation*}%
so that%
\begin{equation*}
\|v(y)\|\leq \frac{e^{-\eta y}}{\eta ^{2}\nu }\|f\|_{L_{\eta }^{\infty }},\text{ }%
\|Dv(y)\|\leq \frac{e^{-\eta y}}{\eta \nu }\|f\|_{L_{\eta }^{\infty }},\text{ }%
\|D^{2}v(y)\|\leq \frac{e^{-\eta y}}{\nu }\|f\|_{L_{\eta }^{\infty }}.
\end{equation*}%
Let us now solve%
\begin{eqnarray*}
\lambda w-\nu D^{2}w &=&-\lambda v \\
w(0) &=&0,
\end{eqnarray*}%
which leads to
\begin{eqnarray*}
w(y) &=&\frac{-\lambda }{2s\nu }\int_{y}^{\infty }v(\tau )e^{s(y-\tau
)}d\tau -\frac{\lambda }{2s\nu }\int_{0}^{y}v(\tau )e^{-s(y-\tau )}d\tau \\
&&+\frac{\lambda }{2s\nu }\int_{0}^{\infty }v(\tau )e^{-s(y+\tau )}d\tau ,
\end{eqnarray*}%
thus
\begin{eqnarray*}
\|w\|_{L^{\infty }} &\leq &\frac{3|s|}{\eta }\|v\|_{L_{\eta }^{\infty }}\leq \frac{%
3|s|}{\eta ^{3}\nu }\|f\|_{L_{\eta }^{\infty }}, \\
\|Dw\|_{L^{\infty }} &\leq &\frac{3|s|}{\eta }\|Dv\|_{L_{\eta }^{\infty }}\leq \frac{%
3|s|}{\eta ^{2}\nu }\|f\|_{L_{\eta }^{\infty }}, \\
\|D^{2}w\|_{L^{\infty }} &\leq &\frac{3|\lambda |}{\eta \nu }\|Dv\|_{L_{\eta }^{%
\infty }}\leq \frac{3|\lambda |}{\eta ^{2}\nu ^{2}}\|f\|_{L_{\eta }^{\infty }},
\end{eqnarray*}%
i.e. 
\begin{equation*}
\|w\|_{W^{2,\infty }}\leq \frac{3|s|}{\eta ^{3}\nu } \Bigl[ 1+\eta +\eta |s| \Bigr]\|f\|_{L_{\eta }^{%
\infty }}.
\end{equation*}%
Now we solve%
\begin{equation*}
\lambda w_{1}-\nu D^{2}w_{1}=0,\text{ }w_{1}(0)=-v(0),
\end{equation*}%
which leads to%
\begin{equation*}
w_{1}(y)=-v(0)e^{-sy},
\end{equation*}%
hence%
\begin{equation*}
\|w_{1}\|_{W^{2,\infty }}\leq \frac{1+|s|+|s|^{2}}{\eta ^{2}\nu }%
\|f\|_{L_{\eta }^{\infty }},
\end{equation*}%
The sum%
\begin{equation*}
u=v+w+w_{1}
\end{equation*}%
satisfies (\ref{equ small lambda}) and we have the estimate%
\begin{equation}
\Bigl\| (\lambda -L_{(\nu )}^{(0)})^{-1}f \Bigr\|_{W^{2,\infty }}\leq \frac{4(1+|s|)^{2}%
}{\eta ^{3}\nu }\|f\|_{L_{\eta }^{\infty }}\leq c_{0}\|f\|_{L_{\eta }^{\infty }}
\label{estim resol nighb 0}
\end{equation}%
which is used when $\lambda =\nu s^{2}$ is bounded by some large constant $A$.

%%%%%%%%%%%%%%%

\subsubsection{Study of the contour integral}

%%%%%%%%%%%%%%%

On the part $\Gamma _{0}$ of $\Gamma$, we have $\lambda =\delta e^{i\theta
}, $ $-\pi /2\leq \theta \leq \pi /2.$ For any $\delta >0$%
\begin{eqnarray*}
\Bigl\|\frac{1}{2i\pi }\int_{\Gamma _{0}}e^{\lambda t} \Bigl(\lambda -\boldsymbol{L}%
_{(\nu )}^{(0)} \Bigr)^{-1}f d\lambda  \Bigr|_{L^{\infty }} &\leq &\frac{%
c_{0}\|f\|_{L_{k}^{\infty }}}{2\pi }\int_{-\pi /2}^{\pi /2}\delta e^{\delta
t}d\theta \\
&\leq &\frac{c_{0}}{2}\delta e^{\delta t}\|f\|_{L_{k}^{\infty }},
\end{eqnarray*}
which goes to $0$ as $\delta$ goes to $0$, $t$ being fixed.

On the part $\Gamma _{1}$ we have $\lambda =-\tau +i\delta ,$ $\tau \in
\lbrack 0,\gamma ],$ with $\gamma >0$ of order 1, so that $d\lambda =-d\tau $%
. Hence%
\begin{equation*}
\Bigl\|\frac{1}{2i\pi }\int_{\Gamma _{1}}e^{\lambda t} \Bigl(\lambda -\boldsymbol{L}%
_{(\nu )}^{(0)} \Bigr)^{-1}fd\lambda  \Bigr\|_{L^{\infty }}\leq \frac{%
c_{0}\|f\|_{L_{k}^{\infty }}}{2\pi }\int_{0}^{\gamma }e^{-\tau t}d\tau =%
\frac{c_{0}\|f\|_{L_{k}^{\infty }}}{2\pi }\frac{(1-e^{\gamma t})}{t}.
\end{equation*}%
Now on the part $\Gamma _{2}$ of the contour $\Gamma $ we have $\lambda
=-\gamma +i\beta ,$ $\beta \in \lbrack \delta ,A],$ so that $d\lambda
=d\beta ,$ and the estimate (\ref{estim resol nighb 0}) leads to%
\begin{equation*}
\Bigl\|\frac{1}{2i\pi }\int_{\Gamma _{2}}e^{\lambda t} \Bigl(\lambda -\boldsymbol{L}%
_{(\nu )}^{(0)}\Bigr)^{-1}fd\lambda \Bigr\|_{L^{\infty }}\leq \frac{%
c_{0}\|f\|_{L_{k}^{\infty }}e^{-\gamma t}}{2\pi }(A-\delta ).
\end{equation*}%
It is then clear that the estimates for integral on $\overline{\Gamma _{1}}$
and $\overline{\Gamma _{2}}$ may be obtained in the same way and that the
integral on $\Gamma _{0}\cup \Gamma _{1}\cup \overline{\Gamma _{1}}\cup
\Gamma _{2}\cup \overline{\Gamma _{2}}$ has a limit when $\delta \rightarrow
0.$ Completing with the rest of $\Gamma $ which is independent of $\delta ,$
and which is bounded in a standard way for $t\in (0,1),$ and bounded for $%
t>1 $ by $ce^{-\delta _{1}t},$ we obtain an estimate of the form (the
main part comes from the integral on $\Gamma _{1})$%
\begin{equation*}
\|e^{\boldsymbol{L}_{(\nu )}^{(0)}t}f\|_{L^{\infty }}\leq \Bigl\|\frac{1}{2i\pi }%
\int_{\Gamma }e^{\lambda t} \Bigl(\lambda -\boldsymbol{L}_{(\nu
)}^{(0)} \Bigr)^{-1}fd\lambda \Bigr\|_{L^{\infty }}\leq \frac{C}{1+t}%
\|f\|_{L_{\eta}^{\infty }}.
\end{equation*}%

%%%%%%%%%%%%%%%%%%

\subsubsection{End of the proof}

%%%%%%%%%%%%%%%%%%

The rest of the proof of the first part of Lemma \ref{lem:estimsemigroup
poids} results from the fact that $e^{\boldsymbol{L}_{(\nu )}^{(0)}t}$
commutes with $\boldsymbol{L}_{(\nu )}^{(0)}$ and from%
\begin{equation*}
D^{2}(\lambda -\boldsymbol{L}_{(\nu )}^{(0)})^{-1}f=(\lambda -\boldsymbol{L}%
_{(\nu )}^{(0)})^{-1}D^{2}f,
\end{equation*}%
which implies that
\begin{equation*}
\|D^{2}e^{\boldsymbol{L}_{(\nu )}^{(0)}t}f\|_{L^{\infty }}\leq \frac{C}{1+t}%
\|D^{2}f\|_{L^{\infty }}.
\end{equation*}%
Moreover the interpolation estimate 
\begin{equation*}
\|Dv\|_{L^{\infty }}\leq \|D^{2}v\|_{_{L^{\infty }}}+\|v\|_{L^{\infty }},
\end{equation*}%
is obtained in solving%
\begin{equation*}
D^{2}v-v=g\in L^{\infty },\text{ }v(0)=0.
\end{equation*}%
Indeed this leads to%
\begin{eqnarray*}
v(y) &=&-\frac{1}{2}\int_{0}^{\infty }g(\tau )e^{-|\tau -y|}d\tau +\frac{1}{2%
}\int_{0}^{\infty }g(\tau )e^{-(\tau +y)}d\tau , \\
Dv(y) &=&\frac{1}{2}\int_{0}^{y}g(\tau )e^{\tau -y}d\tau -\frac{1}{2}%
\int_{y}^{\infty }g(\tau )e^{y-\tau }d\tau -\frac{1}{2}\int_{0}^{\infty
}g(\tau )e^{-(\tau +y)}d\tau ,
\end{eqnarray*}%
and we obtain 
\begin{equation*}
\|Dv\|_{L^{\infty }}\leq \|g\|_{L^{\infty }}\leq \|D^{2}v\|_{L^{\infty
}}+\|v\|_{L^{\infty }}.
\end{equation*}%
Now, taking $v=(\lambda -\boldsymbol{L}_{(\nu )}^{(0)})^{-1}f$, we obtain the first part of the Lemma.

Let $f\in W_{\eta}^{1,\infty }$. Then the solution $v(y)$ in $%
W_{0}^{2,\infty }$ of%
\begin{equation*}
(\lambda -\boldsymbol{L}_{(\nu )}^{(0)})v=f
\end{equation*}%
with $s^{2}=\lambda /\nu ,$ $\Re s>0,$ is given by (\ref{v(y) Linfty}%
) and, after an integration by parts, we have%
\begin{eqnarray*}
Dv(y) &=&\frac{1}{2s\nu }\int_{0}^{\infty }f^{\prime }(\tau )e^{-s|y-\tau
|}d\tau +\frac{1}{2s\nu }\int_{0}^{\infty }f^{\prime }(\tau )e^{-s(y+\tau
)}d\tau , \\
D^{2}v(y) &=&-\frac{1}{2\nu }\int_{0}^{y}f^{\prime }(\tau )e^{-s(y-\tau
)}d\tau +\frac{1}{2\nu }\int_{y}^{\infty }f^{\prime }(\tau )e^{s(y-\tau
)}d\tau \\
&&-\frac{1}{2\nu }\int_{0}^{\infty }f^{\prime }(\tau )e^{-s(y+\tau )}d\tau .
\end{eqnarray*}%
It is clear that we have the estimates 
\begin{equation}
\|v\|_{W^{1,\infty }}\leq \left\{ 
\begin{array}{c}
c_{0}\|f\|_{W_{\eta}^{1,\infty }}\text{ for }|\lambda |\leq A \\ 
\frac{C}{|\lambda |}\|f\|_{W_{\eta}^{1,\infty }}\text{ for }\lambda \in \Gamma
_{3}%
\end{array}%
\right. ,  \label{estim inv L-lambda 0-mode}
\end{equation}%
\begin{equation*}
\|D^{2}v\|_{L^{\infty }}\leq \left\{ 
\begin{array}{c}
c_{0}\|f\|_{W_{\eta}^{1,\infty }}\text{ for }|\lambda |\leq A \\ 
\frac{C}{\sqrt{|\lambda |}}\|f\|_{W_{\eta}^{1,\infty }}\text{ for }\lambda \in
\Gamma _{3}%
\end{array}%
\right. ,
\end{equation*}%
where $\lambda \in \Gamma _{3}$ is such that $\arg \lambda =2\pi /3-\delta ,$
hence%
\begin{equation*}
\Re \lambda \simeq -\frac{1}{2}|\lambda |,\text{ }|d\lambda |=d|\lambda |,%
\text{ }A\leq \Im \lambda \sim \frac{\sqrt{3}}{2}|\lambda |\leq \infty .
\end{equation*}%
It results that for, $\lambda \in \Gamma$, we have the estimates%
\begin{equation}
\|v\|_{W^{2,\infty }}\leq \frac{C}{(1+|\lambda |)}\|f\|_{W_{\eta}^{2,\infty }},
\label{estim resol L(0) b}
\end{equation}%
\begin{equation}
\|v\|_{W^{2,\infty }}\leq \frac{C}{1+\sqrt{|\lambda |}}\|f\|_{W_{\eta}^{1,%
\infty }}.  \label{estim resol L(0)  a}
\end{equation}%
The estimate (\ref{estim resol L(0) b}) implies the first part of the Lemma %
\ref{lem:estimsemigroup poids} already proved. Estimate (\ref{estim resol
L(0) a}) leads to the second estimate (\ref{estim sg L(0) W1}) in Lemma \ref%
{lem:estimsemigroup poids}, where we notice that the behavior in $1/\sqrt{%
|\lambda |}$ for $|\lambda |$ large in the part $\Gamma _{3}$ of the
integral (\ref{integ for semi group}) gives the factor $1/\sqrt{t}$ for $t$
near $0,$ while the integral on $\Gamma _{1}$ gives a bound in $1/(1+t)$ as $%
t\rightarrow \infty $.

%%%%%%%%%%%%%%%%%%%%%%%%%%%%%%%%%%%%%%%%%%%%%%%%

\subsection{Proof of Lemma \protect\ref{lem perturb hilb}}

%%%%%%%%%%%%%%%%%%%%%%%%%%%%%%%%%%%%%%%%%%%%%%%%%

\label{proof6} 
Let us solve the system%
\begin{equation}
\Bigl( \mathbb{I}-\varepsilon \boldsymbol{L}_{(\nu )}^{(0)} \Bigr)u_{\varepsilon }=u\in 
\mathring{H}_{\eta}^{1},\text{ }u_{\varepsilon }\in \mathring{H}_{\eta}^{2},
\label{basic syst}
\end{equation}%
hence in Fourier components%
\begin{eqnarray*}
(D^{2}-s_{n}^{2})u_{\varepsilon ,n}+\binom{in\alpha }{D}q_{n} &=&\frac{1}{%
\varepsilon \nu }\binom{u_{n}^{x}}{u_{n}^{y}}, \\
in\alpha u_{\varepsilon ,n}^{x}+Du_{\varepsilon ,n}^{y} &=&0,u_{\varepsilon
,n}|_{y=0}=0,
\end{eqnarray*}%
where 
\begin{equation*}
in\alpha u_{n}^{x}+Du_{n}^{y}=0,
\end{equation*}%
\begin{equation*}
s_{n}^{2}=n^{2}\alpha ^{2}+\frac{1}{\varepsilon \nu }.
\end{equation*}%
We obtain%
\begin{eqnarray*}
(D^{2}-n^{2}\alpha ^{2})[D^{2}-s_{n}^{2})]u_{\varepsilon ,n}^{y} &=&\frac{1}{%
\varepsilon \nu }(D^{2}-n^{2}\alpha ^{2})u_{n}^{y}, \\
u_{\varepsilon ,n}^{y} &=&Du_{\varepsilon ,n}^{y}=0\text{ for }y=0,
\end{eqnarray*}%
which allows to solve separately in $u_{\varepsilon ,n}^{y}$ and $%
u_{\varepsilon ,n}^{x}.$ First, we get%
\begin{equation}
u_{\varepsilon ,n}^{y}(y)=w_{n}(y)-Dw_{n}(0)\frac{e^{-s_{n}y}-e^{-n\alpha y}}{s_{n}-n\alpha },  \label{u eps n}
\end{equation}%
where%
\begin{equation*}
w_{n}(y)=-\frac{1}{2\varepsilon \nu s_{n}}\int_{0}^{\infty }u_{n}^{y}(\tau
)e^{-s_{n}|\tau -y|}d\tau +\frac{1}{2\varepsilon \nu s_{n}}\int_{0}^{\infty
}u_{n}^{y}(\tau )e^{-s_{n}(\tau +y)}d\tau .
\end{equation*}%
After an integration by parts, we obtain 
\begin{eqnarray*}
w_{n}(y) &=&\frac{1}{2\varepsilon \nu s_{n}^{2}}\int_{0}^{y}Du_{n}^{y}(\tau
)e^{-s_{n}(y-\tau )}d\tau -\frac{1}{2\varepsilon \nu s_{n}^{2}}%
\int_{y}^{\infty }Du_{n}^{y}(\tau )e^{s_{n}(y-\tau )}d\tau  \\
&&+\frac{1}{2\varepsilon \nu s_{n}^{2}}\int_{0}^{\infty }Du_{n}^{y}(\tau
)e^{-s_{n}(y+\tau )}d\tau -\frac{1}{\varepsilon \nu s_{n}^{2}}u_{n}^{y}(y)
\end{eqnarray*}%
from which we deduce that 
\begin{equation*}
|w_{n}(y)e^{\eta y}|\leq \frac{3M_{n}}{2\nu \varepsilon s_{n}^{2}(s_{n}-\eta)}+%
\frac{M_{n}}{\varepsilon \nu ns_{n}^{2}}
\end{equation*}%
where%
\begin{eqnarray*}
n^{2}|u_{n}^{y}(y)e^{\eta y}|^{2}+|Du_{n}^{y}(y)e^{\eta y}|^{2} &\leq &M_{n}^{2}, \\
\|u\|_{\mathring{H}_{\eta}^{1}}^{2} &=&\sum_{|n|\geq 1}M_{n}^{2}.
\end{eqnarray*}%
We also have%
\begin{eqnarray*}
Dw_{n}(y) &=&-\frac{1}{2\varepsilon \nu s_{n}}\int_{0}^{y}Du_{n}^{y}(\tau
)e^{-s_{n}(y-\tau )}d\tau -\frac{1}{2\varepsilon \nu s_{n}}\int_{y}^{\infty
}Du_{n}^{y}(\tau )e^{-s_{n}(\tau -y)}d\tau  \\
&&-\frac{1}{2\varepsilon \nu s_{n}}\int_{0}^{\infty }Du_{n}^{y}(\tau
)e^{-s_{n}(\tau +y)}d\tau ,
\end{eqnarray*}%
from which we deduce that 
\begin{equation*}
|Dw_{n}(y)e^{\eta y}|\leq \frac{3M_{n}}{2\nu \varepsilon s_{n}(s_{n}-\eta )}
\end{equation*}%
and%
\begin{eqnarray*}
D^{2}w_{n}(y) &=&\frac{1}{2\varepsilon \nu }\int_{0}^{y}Du_{n}^{y}(\tau
)e^{-s_{n}(y-\tau )}d\tau -\frac{1}{2\varepsilon \nu }\int_{y}^{\infty
}Du_{n}^{y}(\tau )e^{-s_{n}(\tau -y)}d\tau  \\
&&+\frac{1}{2\varepsilon \nu }\int_{0}^{\infty }Du_{n}^{y}(\tau
)e^{-s_{n}(\tau +y)}d\tau ,
\end{eqnarray*}%
from which we obtain that 
\begin{equation*}
|D^{2}w_{n}(y)e^{\eta y}|\leq \frac{3M_{n}}{2\nu \varepsilon (s_{n}-\eta )}.
\end{equation*}%
Now, we show that there exists $C>0$ such that for $|n|\geq 1$ 
\begin{equation}
\max \left\{ \frac{3n^{2}}{2\nu \varepsilon s_{n}^{2}(s_{n}-\eta )},\frac{|n|}{%
\nu \varepsilon s_{n}^{2}},\frac{|n|}{2\nu \varepsilon s_{n}(s_{n}-\eta )},\frac{%
3}{2\nu \varepsilon (s_{n}-\eta )}\right\} \leq \frac{C}{\varepsilon ^{1/2}}.
\label{simple estim}
\end{equation}%
Indeed, we note that%
\begin{eqnarray*}
\varepsilon s_{n}^{2} &=&\varepsilon n^{2}\alpha ^{2}+1/\nu >\frac{1}{2}%
(\alpha |n|\sqrt{\varepsilon }+\nu ^{-1/2})^{2}, \\
\sqrt{\varepsilon }(s_{n}-\eta ) &>&\frac{1}{\sqrt{2}} \Bigl(\alpha |n|-\eta \sqrt{2})%
\sqrt{\varepsilon }+\nu ^{-1/2} \Bigr),
\end{eqnarray*}%
so that, for $\eta $ small enough, there exists $c>0$ independent of $n$ and $%
\varepsilon $ such that%
\begin{equation*}
\sqrt{\varepsilon }s_{n}>\sqrt{\varepsilon }(s_{n}-\eta )>c \Bigl( |n|\sqrt{\varepsilon 
}+1 \Bigr).
\end{equation*}%
Then, we have (with $X=|n|\sqrt{\varepsilon }$) 
\begin{equation*}
\frac{3n^{2}}{2\nu \varepsilon s_{n}^{2}(s_{n}-\eta )}\leq \frac{3n^{2}\sqrt{%
\varepsilon }}{2\nu c^{3}(|n|\sqrt{\varepsilon }+1)^{3}}\leq \frac{C_{1}}{%
\sqrt{\varepsilon }}\frac{X^{2}}{(X+1)^{3}}\leq \frac{C}{\varepsilon ^{1/2}},
\end{equation*}%
\begin{equation*}
\frac{|n|}{2\nu \varepsilon s_{n}(s_{n}-\eta )}\leq \frac{1}{2\nu c^{2}\sqrt{%
\varepsilon }}\frac{X}{(X+1)^{2}}\leq \frac{C}{\varepsilon ^{1/2}},
\end{equation*}%
\begin{equation*}
\frac{3}{2\nu \varepsilon (s_{n}-\eta )}\leq \frac{3}{2c\nu \sqrt{\varepsilon }}%
\frac{1}{X+1}\leq \frac{C}{\varepsilon ^{1/2}},
\end{equation*}%
showing that (\ref{simple estim}) holds true. It then results that%
\begin{equation*}
n^{4}\|w_{n}\|_{C_{\eta}^{0}}^{2}+n^{2}\|Dw_{n}\|_{C_{\eta}^{0}}^{2}+\|D^{2}w_{n}\|_{C_{\eta}^{0}}^{2}\leq C_{1}%
\frac{M_{n}^{2}}{\varepsilon }.
\end{equation*}%
Coming back to (\ref{u eps n}), we use (assuming from now on that $n>0$) 
\begin{equation*}
|Dw_{n}(0)|\leq c_{0}\frac{M_{n}}{(X+1)^{2}},\text{ }X=n\sqrt{\varepsilon ,}
\end{equation*}%
and good estimates for the function of $y$ defined by%
\begin{equation*}
b_{n}(y)\overset{def}{=}\frac{e^{-s_{n}y}-e^{-n\alpha y}}{s_{n}-n\alpha }.
\end{equation*}%
Indeed, we have%
\begin{eqnarray*}
|e^{-(n\alpha -\eta )y}-e^{-(s_{n}-\eta )y}| &=&e^{-(n\alpha
-\eta )y}(1-e^{-(s_{n}-n\alpha )y}) \\
&\leq &(s_{n}-n\alpha )ye^{-(n\alpha -\eta )y} \\
&\leq &\frac{(s_{n}-n\alpha )}{(n\alpha -\eta )e},
\end{eqnarray*}%
hence there exists $C>0$ such that%
\begin{equation*}
\|b_{n}\|_{C_{\eta}^{0}}\leq \frac{C}{n},
\end{equation*}%
and 
\begin{equation*}
n^{2}\|Dw_{n}(0)b_{n}\|_{C_{\eta}^{0}}\leq c_{0}CM_{n}\frac{n}{(X+1)^{2}}\leq
c_{0}C\frac{M_{n}}{\sqrt{\varepsilon }}.
\end{equation*}%
Now 
\begin{eqnarray*}
Db_{n}(y) =\frac{n\alpha e^{-n\alpha y}-s_{n}e^{-s_{n}y}}{s_{n}-n\alpha }
=-e^{-s_{n}y}+\frac{n\alpha (e^{-n\alpha y}-e^{-s_{n}y})}{s_{n}-n\alpha }
\end{eqnarray*}%
hence there exists $C>0$ such that%
\begin{equation*}
\|Db_{n}\|_{C_{\eta}^{0}}\leq C
\end{equation*}%
and%
\begin{equation*}
n\|Dw_{n}(0)Db_{n}\|_{C_{\eta}^{0}}\leq C_{1}M_{n}\frac{n}{(X+1)^{2}}\leq C_{2}%
\frac{M_{n}X}{\sqrt{\varepsilon }(X+1)^{2}}\leq C_{2}\frac{M_{n}}{\sqrt{%
\varepsilon }}.
\end{equation*}%
In the same way, we have%
\begin{equation*}
D^{2}b_{n}(y)=(s_{n}+n\alpha )e^{-s_{n}y}+\frac{n^{2}\alpha
^{2}(e^{-s_{n}y}-e^{-n\alpha y})}{s_{n}-n\alpha }
\end{equation*}%
leading to%
\begin{equation*}
\|D^{2}b_{n}\|_{C_{\eta}^{0}}\leq C_{3}\frac{X+1}{\sqrt{\varepsilon }},
\end{equation*}%
hence%
\begin{equation*}
\|Dw_{n}(0)D^{2}b_{n}\|_{C_{\eta}^{0}}\leq c_{0}C_{3}M_{n}\frac{1}{\sqrt{%
\varepsilon }(X+1)}\leq c_{0}C_{3}\frac{M_{n}}{\sqrt{\varepsilon }}.
\end{equation*}%
Finally%
\begin{equation*}
n^{4}\|u_{\varepsilon ,n}^{y}\|_{C_{\eta}^{0}}^{2}+n^{2}\|Du_{\varepsilon
,n}^{y}\|_{C_{\eta}^{0}}^{2}+\|D^{2}u_{\varepsilon
,n}^{y}\|_{C_{\eta}^{0}}^{2}\leq C\frac{M_{n}^{2}}{\varepsilon }.
\end{equation*}
Let us now consider $u_{n,\varepsilon }^{x}.$ We have%
\begin{eqnarray*}
(D^{2}-s_{n}^{2})u_{n,\varepsilon }^{x} &=&-in\alpha q_{n}+\frac{u_{n}^{x}}{%
\nu _{0}\varepsilon },\text{ }u_{n,\varepsilon }^{x}(0)=0, \\
(D^{2}-n^{2}\alpha ^{2})q_{n} &=&0,\text{ }Dq_{n}(0)=-D^{2}u_{\varepsilon
,n}^{y}(0),
\end{eqnarray*}%
hence%
\begin{equation*}
q_{n}(y)=\frac{D^{2}u_{\varepsilon ,n}^{y}(0)}{n\alpha }e^{-n\alpha y},
\end{equation*}%
and 
\begin{eqnarray*}
u_{n,\varepsilon }^{x}(y) &=&\frac{1}{2\varepsilon \nu s_{n}^{2}}%
\int_{0}^{y}Du_{n}^{x}(\tau )e^{-s_{n}(y-\tau )}d\tau -\frac{1}{2\varepsilon
\nu s_{n}^{2}}\int_{y}^{\infty }Du_{n}^{x}(\tau )e^{s_{n}(y-\tau )}d\tau  \\
&&+\frac{1}{2\varepsilon \nu s_{n}^{2}}\int_{0}^{\infty }Du_{n}^{x}(\tau
)e^{-s_{n}(y+\tau )}d\tau -\frac{1}{\varepsilon \nu s_{n}^{2}}u_{n}^{x}(y) \\
&&+\frac{iD^{2}u_{\varepsilon ,n}^{y}(0)}{2\nu s_{n}}\left( -b_{n}(y)+\frac{%
e^{-n\alpha y}+e^{-s_{n}y}}{s_{n}+n\alpha }\right) .
\end{eqnarray*}%
It appears that the estimates on the part without $D^{2}u_{\varepsilon
,n}^{y}(0)$ are the same as the estimates obtained for $w_{n}(y).$ Now we
have in addition, as seen above,%
\begin{equation*}
\left\vert \frac{iD^{2}u_{\varepsilon ,n}^{y}(0)}{2\nu s_{n}}\right\vert
\leq C\frac{M_{n}}{X+1},
\end{equation*}%
\begin{equation*}
n^{2}\|b_{n}\|_{C_{\eta}^{0}}+n\|Db_{n}\|_{C_{\eta}^{0}}+\|D^{2}b_{n}\|_{C_{\eta}^{0}}\leq C_{1}%
\frac{X+1}{\sqrt{\varepsilon }}.
\end{equation*}%
Finally, we easily have 
\begin{align*}
n^{2}\Bigl\|\frac{e^{-n\alpha y}+e^{-s_{n}y}}{s_{n}+n\alpha }\Bigr\|_{C_{\eta}^{0}}
&+n \Bigl\| \frac{n\alpha e^{-n\alpha y}+s_{n}e^{-s_{n}y}}{s_{n}+n\alpha } \Bigr\|_{C_{\eta}^{0}}
+ \Bigl\|\frac{n^{2}\alpha ^{2}e^{-n\alpha y}+s_{n}^{2}e^{-s_{n}y}}{s_{n}+n\alpha } \Bigr\|_{C_{\eta}^{0}} \\
& \qquad \qquad \leq C_{2}\frac{X+1}{\sqrt{\varepsilon }},
\end{align*}%
hence %
\begin{equation*}
n^{4}\|u_{\varepsilon ,n}^{x}\|_{C_{\eta}^{0}}^{2}+n^{2}\|Du_{\varepsilon
,n}^{x}\|_{C_{\eta}^{0}}^{2}+\|D^{2}u_{\varepsilon
,n}^{x}\|_{C_{\eta}^{0}}^{2}\leq C\frac{M_{n}^{2}}{\varepsilon },
\end{equation*}%
which ends the proof of the Lemma \ref{lem perturb hilb}. %

%%%%%%%%%%%%%%%%%%%%%%%%%%%%%%%%%%%%%%%%%%%%%%%%

\subsection{Estimates on $(\boldsymbol{L}_{\protect\nu ,\protect%
\omega }-\protect\lambda )^{-1}\widetilde{P}$}\label{app:complement}

%%%%%%%%%%%%%%%%%%%%%%%%%%%%%%%%%%%%%%%%%%%%%%%%%

Let us consider $v\in \mathring{H}_{\eta}^{2}$, solution of 
\begin{equation*}
(\boldsymbol{L}_{\nu ,\omega }-\lambda )v=f\in \mathring{L}_{\eta}^{2},
\end{equation*}%
for $\lambda \notin \Sigma _{U_+,\omega }$. We have%
\begin{equation*}
\boldsymbol{L}_{\nu ,\omega }=\boldsymbol{L}_{\nu ,\omega }^{(0)}+%
\boldsymbol{L}^{(1)},\text{ with }\boldsymbol{L}_{\nu ,\omega }^{(0)}=\nu
\Pi \Delta -\frac{\omega }{\alpha }\frac{\partial }{\partial \xi },
\end{equation*}%
and, using the Fourier components, we obtain%
\begin{equation*}
\lbrack (\boldsymbol{L}_{\nu ,\omega }^{(0)})_{n}-\lambda ]v_{n}+(%
\boldsymbol{L}^{(1)})_{n}v_{n}=f_{n},
\end{equation*}%
with%
\begin{equation*}
(\boldsymbol{L}^{(1)})_{n}v_{n}=-\Pi _{n}\left\{ in\alpha U\binom{v_{n}^{x}}{%
v_{n}^{y}}+\binom{v_{n}^{y}DU}{0}\right\} ,
\end{equation*}%
hence 
\begin{equation*}
\|(\boldsymbol{L}^{(1)})_{n}v_{n}\|_{C_{\eta}^{0}}\leq
C(1+|n|)\|v_{n}\|_{C_{\eta}^{0}}.
\end{equation*}%
Now, from the computation above, we know that for $\lambda \notin \Sigma
_{U_+,\omega }$ 
\begin{equation*}
\Bigl\|[(\boldsymbol{L}_{\nu ,\omega }^{(0)})_{n}-\lambda ]^{-1} \Bigr\|_{\mathcal{L}%
(C_{\eta}^{0})}\leq \frac{C}{(n^{2}-1)+|\lambda +\nu \alpha ^{2}|},
\end{equation*}%
so that there exists $N_{0}$ such that for $|n|>N_{0}$%
\begin{equation*}
\Bigl\|[(\boldsymbol{L}_{\nu ,\omega }^{(0)})_{n}-\lambda ]^{-1}(\boldsymbol{L}%
^{(1)})_{n} \Bigr\|_{\mathcal{L}(C_{\eta}^{0})}\leq 1/2,
\end{equation*}%
hence, for $|n|>N_{0}$%
\begin{eqnarray*}
\|v_{n}\|_{C_{\eta}^{0}} &\leq &2\|[(\boldsymbol{L}_{\nu ,\omega
}^{(0)})_{n}-\lambda ]^{-1}f_{n}\|_{C_{\eta}^{0}} \\
&\leq &\frac{2C}{(n^{2}-1)+|\lambda +\nu \alpha ^{2}|}\|f_{n}\|_{C_{\eta}^{0}}.
\end{eqnarray*}%
Now we choose a contour $\Gamma$, like the one described at Figure 2, such that there is no
eigenvalue (which form a bounded discrete set) of $(\boldsymbol{L}_{\nu ,\omega })_{n}$
on it for $1\leq |n|\leq N_{0}$. For such a contour there exists $M>0$ such that for $\lambda \in
\Gamma$,
\begin{eqnarray*}
\|(\boldsymbol{L}_{\nu ,\omega }-\lambda )^{-1}\widetilde{P}\|_{\mathcal{L}(%
\mathring{H}_{\eta}^{2},\mathring{H}_{\eta}^{2})} &\leq &\frac{M}{|\lambda +\nu
\alpha ^{2}|}, \\
\|(\boldsymbol{L}_{\nu ,\omega }-\lambda )^{-1}\widetilde{P}\|_{\mathcal{L}(%
\mathring{H}_{\eta}^{1},\mathring{H}_{\eta}^{2})} &\leq &\frac{M}{|\lambda +\nu
\alpha ^{2}|^{1/2}}, \\
\|(\boldsymbol{L}_{\nu ,\omega }-\lambda )^{-1}\widetilde{P}\|_{\mathcal{L}(%
\mathring{L}_{\eta}^{2},\mathring{H}_{\eta}^{2})} &\leq &M.
\end{eqnarray*}%
We may observe that, for $\lambda \in \Gamma ,$ there exists $C>0$ and $a>0$
such that%
\begin{equation*}
|\lambda +\nu \alpha ^{2}|\geq \frac{1}{C}(a^{2}+|\lambda |).
\end{equation*}

%%%%%%%%%%%%%%%%%%%%%%%%%%%%%%%%%%%%%%%%%%%%%

\subsection{Estimates for $(\protect\lambda -\boldsymbol{A}_{\protect%
\varepsilon })^{-1}$ for $\protect\lambda \in \Gamma $, and for $e^{\boldsymbol{A}%
_{\protect\varepsilon }t}$ \label{App: estim semigrAeps}}

%%%%%%%%%%%%%%%%%%%%%%%%%%%%%%%%%%%%%%%%%%%%%

Let us estimate the solution of%
\begin{equation*}
(\lambda -\boldsymbol{A}_{\varepsilon })u=f\in Q_{\varepsilon }\mathcal{X}%
_{\eta},
\end{equation*}%
where we look for $u\in Q_{\varepsilon }\mathcal{Z}_{\eta}.$ We look for 
\begin{eqnarray*}
u &=&\beta u_{\varepsilon }+\widetilde{w}+u_{0},\text{ }\beta =\langle
u,\phi _{1}^{\ast }\rangle ,\text{ }\langle u,\phi _{0}^{\ast }\rangle =0, \\
\widetilde{P}u &=&\beta \widetilde{P}u_{\varepsilon }+\widetilde{w},\text{ }%
P_{0}u=\beta P_{0}u_{\varepsilon }+u_{0},
\end{eqnarray*}%
with%
\begin{eqnarray*}
f &=&\eta u_{\varepsilon }+\widetilde{g}+f_{0},\text{ }\eta =\langle f,\phi
_{1}^{\ast }\rangle ,\text{ }\langle f,\phi _{0}^{\ast }\rangle =0, \\
\widetilde{P}f &=&\eta \widetilde{P}u_{\varepsilon }+\widetilde{g},\text{ }%
P_{0}f=\eta P_{0}u_{\varepsilon }+f_{0},
\end{eqnarray*}%
and%
\begin{eqnarray*}
\boldsymbol{A}_{\varepsilon }(\beta u_{\varepsilon }+\widetilde{w}+u_{0})
&=&\beta \sigma _{\varepsilon }u_{\varepsilon }+\boldsymbol{A}_{\varepsilon }%
\widetilde{w}+\boldsymbol{A}_{\varepsilon }u_{0}, \\
\boldsymbol{A}_{\varepsilon }u_{0} &=&\boldsymbol{L}_{\nu
}^{(0)}u_{0}+2Q_{\varepsilon }B(\widehat{V}_{\varepsilon },u_{0}), \\
\boldsymbol{A}_{\varepsilon }\widetilde{w} &=&Q_{\varepsilon }[\boldsymbol{L}%
_{\nu ,\omega }\widetilde{w}+2B(\widehat{V}_{\varepsilon },\widetilde{w})].
\end{eqnarray*}%
Then we obtain the following system for the unknown $\beta ,\widetilde{w}%
,u_{0}:$%
\begin{equation}
(\lambda -\sigma _{\varepsilon })\beta =\eta + \Bigl\langle \Bigl[ \boldsymbol{L}%
_{\nu ,\omega }-\boldsymbol{L}_{\nu _{0},\omega _{0}}+2B(\widehat{V}%
_{\varepsilon },\cdot ) \Bigr](\widetilde{w}+u_{0}),\phi _{1}^{\ast }\Bigr\rangle ,
\label{equ beta}
\end{equation}%
\begin{eqnarray}
(\lambda -\widetilde{P}\boldsymbol{A}_{\varepsilon })\widetilde{w} &=&%
\widetilde{g}+2\widetilde{P}Q_{\varepsilon }B(\widehat{V}_{\varepsilon
},u_{0}),  \label{equ wtilde} \\
(\lambda -\boldsymbol{L}_{\nu }^{(0)})u_{0} &=&f_{0}+2P_{0}B(\widehat{V}%
_{\varepsilon },\widetilde{w}),  \label{equ u0}
\end{eqnarray}%
and we need to estimate the solution all along the curve $\Gamma $ described
by $\lambda .$ Now we fix the choice of $\Gamma $ in such a way that $\gamma 
$ (which defines the parts $\Gamma _{1}$ and $\Gamma _{2}$) is such that%
\begin{equation*}
\gamma =\kappa ^{2}\varepsilon ^{2}\sim |\sigma _{\varepsilon }|/2,
\end{equation*}%
so that the curves $\Gamma _{0}\cup \Gamma _{1}\cup \Gamma _{2}$ lie on the
right of the real eigenvalue $\sigma _{\varepsilon }=-|\sigma _{\varepsilon
}|.$ We notice that the operator $\widetilde{P}\boldsymbol{A}_{\varepsilon }$
is%
\begin{equation*}
Q_{\varepsilon } \Bigl(\boldsymbol{L}_{\nu ,\omega }+2\widetilde{P}B(\widehat{V}%
_{\varepsilon },\cdot ) \Bigr),
\end{equation*}%
and we notice that for $\varepsilon =0,$ it is the operator $Q_{0}%
\boldsymbol{L}_{\nu _{0},\omega _{0}}$ which has a simple eigenvalue $0$,
eliminated when it is acting in the subspace orthogonal to $\phi _{1}^{\ast
}.$ The remaining spectrum is then a perturbation of order $\varepsilon $ of
the spectrum of $\boldsymbol{L}_{\nu ,\omega }$, except the eigenvalue close
to $0,$ i.e. its spectrum is ``far" on the left from the imaginary axis. It
results from estimates obtained at Appendix \ref{app:complement}, that we have, with a certain $a>0$ the estimates%
\begin{equation*}
\Bigl\|(\lambda -\widetilde{P}\boldsymbol{A}_{\varepsilon })^{-1} \Bigr\|_{\mathcal{L}(%
\mathring{H}_{\eta}^{2})}\leq \frac{C}{a^{2}+|\lambda |},\text{ for }\lambda
\in \Gamma ,
\end{equation*}%
and in the same way 
\begin{equation}
\Bigl\|(\lambda -\widetilde{P}\boldsymbol{A}_{\varepsilon })^{-1} \Bigr\|_{\mathcal{L}(%
\mathring{H}_{\eta}^{1},\mathring{H}_{\eta}^{2})}\leq \frac{C}{a+|\lambda |^{1/2}},%
\text{ for }\lambda \in \Gamma ,  \label{estim  resol Aeps Ptilde}
\end{equation}%
using the following property for $\lambda \in \Gamma $%
\begin{equation*}
\frac{1}{|\lambda +\nu \alpha ^{2}|}\leq \frac{c}{a^{2}+|\lambda |},\text{
for some }c,a>0.
\end{equation*}%
From (\ref{equ wtilde}), (\ref{equ u0}) we obtain immediately, for $\lambda
\in \Gamma ,$%
\begin{equation*}
\widetilde{w}=(\lambda -\widetilde{P}\boldsymbol{A}_{\varepsilon })^{-1}%
\widetilde{g}+2(\lambda -\widetilde{P}\boldsymbol{A}_{\varepsilon })^{-1}%
\widetilde{P}Q_{\varepsilon }B(\widehat{V}_{\varepsilon },u_{0})
\end{equation*}%
\begin{equation*}
u_{0}=(\lambda -\boldsymbol{L}_{\nu }^{(0)})^{-1}f_{0}+2(\lambda -%
\boldsymbol{L}_{\nu }^{(0)})^{-1}P_{0}B(\widehat{V}_{\varepsilon },%
\widetilde{w}).
\end{equation*}%
Finally we obtain%
\begin{eqnarray*}
\widetilde{w}-D_{\varepsilon ,\lambda }\widetilde{w} &=&(\lambda -\widetilde{%
P}\boldsymbol{A}_{\varepsilon })^{-1}\widetilde{g}+2(\lambda -\widetilde{P}%
\boldsymbol{A}_{\varepsilon })^{-1}\widetilde{P}Q_{\varepsilon }B \Bigl( \widehat{V}%
_{\varepsilon },(\lambda -\boldsymbol{L}_{\nu }^{(0)})^{-1}f_{0} \Bigr), \\
u_{0}-E_{\varepsilon ,\lambda }u_{0} &=&(\lambda -\boldsymbol{L}_{\nu
}^{(0)})^{-1}f_{0}+2(\lambda -\boldsymbol{L}_{\nu }^{(0)})^{-1}P_{0}B \Bigl( %
\widehat{V}_{\varepsilon },(\lambda -\widetilde{P}\boldsymbol{A}%
_{\varepsilon })^{-1}\widetilde{g} \Bigr),
\end{eqnarray*}%
where we need to estimate the operator $D_{\varepsilon ,\lambda }$ acting on 
$\widetilde{w}:$ 
\begin{equation}
D_{\varepsilon ,\lambda }\widetilde{w}\overset{def}{=}4(\lambda -\widetilde{P%
}\boldsymbol{A}_{\varepsilon })^{-1}\widetilde{P}Q_{\varepsilon }B \Bigl(\widehat{V%
}_{\varepsilon },(\lambda -\boldsymbol{L}_{\nu }^{(0)})^{-1}P_{0}B(\widehat{V%
}_{\varepsilon },\widetilde{w}) \Bigr),  \label{def Deps}
\end{equation}%
and the operator $E_{\varepsilon ,\lambda }$ acting on $u_{0}:$%
\begin{equation}
E_{\varepsilon ,\lambda }u_{0}\overset{def}{=}4(\lambda -\boldsymbol{L}_{\nu
}^{(0)})^{-1}P_{0}B \Bigl( \widehat{V}_{\varepsilon },(\lambda -\widetilde{P}%
\boldsymbol{A}_{\varepsilon })^{-1}\widetilde{P}Q_{\varepsilon }B(\widehat{V}%
_{\varepsilon },u_{0}) \Bigr).  \label{def Eeps}
\end{equation}

%%%%%%%%%%%%%%%%%%%%%%%%%%%%%%%%%%%%%

\subsubsection{Estimate for $D_{\protect\varepsilon ,\protect\lambda }$}

%%%%%%%%%%%%%%%%%%%%%%%%%%%%%%%%%%%%%

We have%
\begin{equation*}
\|D_{\varepsilon ,\lambda }\|_{\mathcal{L}(\mathring{H}_{\eta}^{2})}\leq
c\varepsilon ^{2}\|(\lambda -\widetilde{P}\boldsymbol{A}_{\varepsilon
})^{-1}\|_{\mathcal{L}(\mathring{H}_{\eta}^{1},\mathring{H}_{\eta}^{2})}\|(\lambda
-\boldsymbol{L}_{\nu }^{(0)})^{-1}\|_{\mathcal{L}(\mathring{W}_{\eta}^{1},%
\mathring{W}^{2,\infty })},
\end{equation*}%
hence, using (\ref{estim resol nighb 0}) and
(\ref{estim resol Aeps Ptilde}),
\begin{align*}
\|D_{\varepsilon ,\lambda }\|_{\mathcal{L}(\mathring{H}_{\eta}^{2})} &\leq \frac{%
cC^{2}\varepsilon ^{2}}{\sqrt{|\lambda |}(a+\sqrt{%
|\lambda |})}\leq  C_1 \varepsilon^2 ,\text{ for }\lambda
\in \Gamma_3 
\\
&\leq
\frac{cC^{2}\varepsilon^{2} c_0}{(a+\sqrt{%
|\lambda |})}\leq  C_1 \varepsilon^2 ,\text{ for }\lambda
\in \Gamma_0 \cup \Gamma_1 \cup \Gamma_2. 
\end{align*}%
It results that for $\varepsilon $ small enough, the operator $\mathbb{I}%
-D_{\varepsilon ,\lambda }$ has a bounded inverse in $\mathcal{L}(\mathring{H%
}_{\eta}^{2})$.

\subsubsection{Estimate for $E_{\protect\varepsilon ,\protect\lambda }$}

In the same way, we have%
\begin{equation*}
\|E_{\varepsilon ,\lambda }\|_{\mathcal{L}(\mathring{W}^{2,\infty })}\leq
c\varepsilon ^{2}\|(\lambda -\boldsymbol{L}_{\nu }^{(0)})^{-1}\|_{\mathcal{L}%
(\mathring{W}_{\eta}^{1},\mathring{W}^{2,\infty })}\|(\lambda -\widetilde{P}%
\boldsymbol{A}_{\varepsilon })^{-1}\|_{\mathcal{L}(\mathring{H}_{\eta}^{1},%
\mathring{H}_{\eta}^{2})},
\end{equation*}%
hence%
\begin{equation*}
\|E_{\varepsilon ,\lambda }\|_{\mathcal{L}(\mathring{W}^{2,\infty })}\leq 
C_1 \varepsilon^2,\text{ for }\lambda \in \Gamma .
\end{equation*}%
It results that, for $\varepsilon $ small enough, the operator $\mathbb{I}%
-E_{\varepsilon ,\lambda }$ has a bounded inverse in $\mathcal{L}(\mathring{W%
}^{2,\infty }).$

\subsubsection{Estimates related to $\widetilde{g}$ and $f_{0}$}

Now we have
\begin{eqnarray*}
\Bigl\|2\widetilde{P}Q_{\varepsilon }B \Bigl( \widehat{V}_{\varepsilon },(\lambda -%
\boldsymbol{L}_{\nu }^{(0)})^{-1}f_{0} \Bigr) \Bigr\|_{\mathring{H}_{\eta}^{1}} &\leq &%
\frac{C\varepsilon }{\sqrt{|\lambda |}(1 +\sqrt{|\lambda |}),%
}\|f_{0}\|_{\mathring{W}_{\eta}^{2,\infty }}, \, \hbox{for } \lambda \in \Gamma_3
\\
&\leq &%
\frac{c_0 C\varepsilon }{(1 +\sqrt{|\lambda |}),%
}\|f_{0}\|_{\mathring{W}_{\eta}^{2,\infty }}, \, \hbox{for } \lambda \in \Gamma_0 \cup \Gamma_1 \cup \Gamma_2,
\\
\Bigl\|2P_{0}B \Bigl(\widehat{V}_{\varepsilon },(\lambda -\widetilde{P}\boldsymbol{A}%
_{\varepsilon })^{-1}\widetilde{g}\Bigr) \Bigr\|_{\mathring{W}_{\eta}^{1}} &\leq &\frac{%
C\varepsilon }{a^{2}+|\lambda |}\|\widetilde{g}\|_{\mathring{H}_{\eta}^{2}}.
\end{eqnarray*}%
For $\widetilde{g}\in \mathring{H}_{\eta}^{1}$ and $f_{0}$ in $\mathring{W}%
_{\eta}^{1,\infty }$, we obtain instead%
\begin{eqnarray*}
\Bigl\|2\widetilde{P}Q_{\varepsilon }B \Bigl(\widehat{V}_{\varepsilon },(\lambda -%
\boldsymbol{L}_{\nu }^{(0)})^{-1}f_{0} \Bigr) \Bigr\|_{\mathring{H}_{\eta}^{1}} &\leq &%
\frac{C\varepsilon }{1 +\sqrt{|\lambda |}}\|f_{0}\|_{%
\mathring{W}_{\eta}^{1,\infty }}, \\
\Bigl\|2P_{0}B \Bigl( \widehat{V}_{\varepsilon },(\lambda -\widetilde{P}\boldsymbol{A}%
_{\varepsilon })^{-1}\widetilde{g} \Bigr) \Bigr\|_{\mathring{W}_{\eta}^{1}} &\leq &\frac{%
C\varepsilon }{a+\sqrt{|\lambda |}}\|\widetilde{g}\|_{\mathring{H}_{\eta}^{1}}.
\end{eqnarray*}

\subsubsection{Estimate related to $\protect\beta $}

For the component $\beta $ of $u$ we obtain%
\begin{eqnarray*}
\beta &=&\frac{\eta }{\lambda -\sigma _{\varepsilon }}+\frac{\varepsilon }{%
\lambda -\sigma _{\varepsilon }}F(\widetilde{w}+u_{0}) \\
|F(\widetilde{w}+u_{0})| &\leq &C(\|\widetilde{w}\|_{\mathring{H}%
_{\eta}^{1}}+\|u_{0}\|_{\mathring{W}^{1,\infty }}),\text{ }
\end{eqnarray*}%
where for $\lambda \in \Gamma $ we have $|\lambda -\sigma _{\varepsilon
}|\geq \frac{1}{2}\kappa ^{2}\varepsilon ^{2}$ by construction.

\subsubsection{Estimates for $e^{\boldsymbol{A}_{\protect\varepsilon }t}$}

We still use the contour $\Gamma $ defined in Figure \ref{fig:gamma}
with $\gamma =\kappa ^{2}\varepsilon ^{2}\sim |\sigma _{\varepsilon }|/2.$
We have  
\begin{eqnarray*}
\widetilde{P}e^{\boldsymbol{A}_{\varepsilon }t}(\eta u_{\varepsilon }+%
\widetilde{g}+f_{0}) &=&\frac{1}{2i\pi }\int_{\Gamma }e^{\lambda t}\left( 
\widetilde{w}+\beta \widetilde{P}u_{\varepsilon }\right) d\lambda \\
P_{0}e^{\boldsymbol{A}_{\varepsilon }t}(\eta u_{\varepsilon }+\widetilde{g}%
+f_{0}) &=&\frac{1}{2i\pi }\int_{\Gamma }e^{\lambda t}(u_{0}+\beta
P_{0}u_{\varepsilon })d\lambda
\end{eqnarray*}%
where the explicit part of $\beta $ gives 
\begin{equation}
\frac{1}{2i\pi }\int_{\Gamma }e^{\lambda t}\frac{\eta }{\lambda -\sigma
_{\varepsilon }}d\lambda =\eta e^{-|\sigma _{\varepsilon }|t}.
\label{estim integ beta}
\end{equation}%
Now we use on $\Gamma_3$%
\begin{eqnarray*}
\|\widetilde{w}\|_{\mathring{H}_{\eta}^{2}} &\leq &\frac{C}{a^{2}+|\lambda |}\|%
\widetilde{g}\|_{\mathring{H}_{\eta}^{2}}+\frac{cC^{2}\varepsilon }{\sqrt{%
|\lambda |}(a+\sqrt{|\lambda |})(1 +\sqrt{|\lambda |})}%
\|f_{0}\|_{\mathring{W}_{\eta}^{2,\infty }}, 
\\
\|u_{0}\|_{\mathring{W}^{2,\infty }} &\leq &\frac{C}{\sqrt{|\lambda |}%
(1 +\sqrt{|\lambda |})}\|f_{0}\|_{\mathring{W}%
_{\eta}^{2,\infty }}+\frac{cC^{2}\varepsilon }{(a^{2}+|\lambda |)(1 +\sqrt{|\lambda |})}\|\widetilde{g}\|_{\mathring{H}_{\eta}^{2}},
\end{eqnarray*}%
on $\Gamma_0 \cup \Gamma_1 \cup \Gamma_2$
\begin{eqnarray*}
\|\widetilde{w}\|_{\mathring{H}_{\eta}^{2}} &\leq &\frac{C}{a^{2}+|\lambda |}\|%
\widetilde{g}\|_{\mathring{H}_{\eta}^{2}}+\frac{c c_0 C^{2}\varepsilon }{(a+\sqrt{|\lambda |})(1 +\sqrt{|\lambda |})}%
\|f_{0}\|_{\mathring{W}_{\eta}^{2,\infty }}, 
\\
\|u_{0}\|_{\mathring{W}^{2,\infty }} &\leq &\frac{c_0 C}{%
(1 +\sqrt{|\lambda |})}\|f_{0}\|_{\mathring{W}%
_{\eta}^{2,\infty }}+\frac{c C^{2}\varepsilon }{(a^{2}+|\lambda |)(1 +\sqrt{|\lambda |})}\|\widetilde{g}\|_{\mathring{H}_{\eta}^{2}},
\end{eqnarray*}%
and, for $\widetilde{g}\in \mathring{H}_{\eta}^{1}$ and $f_{0}$ in $\mathring{W}%
_{\eta}^{1,\infty }$,
\begin{eqnarray*}
\|\widetilde{w}\|_{\mathring{H}_{\eta}^{2}} &\leq &\frac{C}{a +%
\sqrt{|\lambda |}}\|\widetilde{g}\|_{\mathring{H}_{\eta}^{1}}+\frac{%
cC^{2}\varepsilon }{(1 +\sqrt{|\lambda |})(a+\sqrt{|\lambda
|})}\|f_{0}\|_{\mathring{W}_{\eta}^{1,\infty }}, \\
\|u_{0}\|_{\mathring{W}^{2,\infty }} &\leq &\frac{C}{1 +%
\sqrt{|\lambda |}}\|f_{0}\|_{\mathring{W}_{\eta}^{1,\infty }}+\frac{%
cC^{2}\varepsilon }{(a+\sqrt{|\lambda |})(1 +\sqrt{|\lambda
|})}\|\widetilde{g}\|_{\mathring{H}_{\eta}^{1}}.
\end{eqnarray*}%
It results that there exists $M(\varepsilon)>0$ such that 
\begin{eqnarray*}
\Bigl\|\frac{1}{2i\pi }\int_{\Gamma }e^{\lambda t}\widetilde{w}d\lambda \Bigr\|_{%
\mathring{H}_{\eta}^{2}} 
+ &\leq &\frac{M(\varepsilon )}{1+t} \Bigl[\|\widetilde{g}\|_{%
\mathring{H}_{\eta}^{2}}+\|f_{0}\|_{\mathring{%
W}_{\eta}^{2,\infty }} \Bigr], \\
\Bigl\|\frac{1}{2i\pi }\int_{\Gamma }e^{\lambda t}u_{0}d\lambda \Bigr\|_{\mathring{W}%
^{2,\infty }} &\leq &\frac{M(\varepsilon )}{1+t}
\Bigl[\|\widetilde{g}\|_{\mathring{%
H}_{\eta}^{2}} +\|f_{0}\|_{\mathring{W}%
_{\eta}^{2,\infty }} \Bigr],
\end{eqnarray*}%
and for $\widetilde{g}\in \mathring{H}_{\eta}^{1}$ and $f_{0}$ in $\mathring{W}%
_{\eta}^{1,\infty }$%
\begin{eqnarray*}
\Bigl\|\frac{1}{2i\pi }\int_{\Gamma }e^{\lambda t}\widetilde{w}d\lambda \Bigr\|_{%
\mathring{H}_{\eta}^{2}} &\leq &\frac{M(\varepsilon )}{\sqrt{t} (1 + \sqrt{t})}\|\widetilde{g}%
\|_{\mathring{H}_{\eta}^{1}}+\frac{M(\varepsilon )}{1+t}\|f_{0}\|_{\mathring{W}%
_{\eta}^{1,\infty }}, \\
\Bigl\|\frac{1}{2i\pi }\int_{\Gamma }e^{\lambda t}u_{0}d\lambda \Bigr\|_{\mathring{W}%
^{2,\infty }} &\leq &\frac{M(\varepsilon )}{\sqrt{t}(1 + \sqrt{t})}\|f_{0}\|_{\mathring{W}%
_{\eta}^{1,\infty }}+\frac{M(\varepsilon )}{1+t}\|\widetilde{g}\|_{\mathring{H}%
_{\eta}^{1}}.
\end{eqnarray*}%
Now we also have%
\begin{equation*}
\frac{1}{2i\pi }\int_{\Gamma }e^{\lambda t}\beta d\lambda =\eta e^{-|\sigma
_{\varepsilon }|t}+\frac{1}{2i\pi }\int_{\Gamma }e^{\lambda t}\frac{%
\varepsilon }{\lambda -\sigma _{\varepsilon }}F(\widetilde{w}+u_{0})d\lambda
,
\end{equation*}%
with%
\begin{eqnarray*}
&& \Bigl|\frac{1}{2i\pi }\int_{\Gamma }e^{\lambda t}\frac{\varepsilon }{\lambda
-\sigma _{\varepsilon }}F(\widetilde{w}+u_{0})d\lambda \Bigr| \\
&\leq &\frac{c}{\varepsilon }\left( \Bigl\|\frac{1}{2i\pi }\int_{\Gamma
}e^{\lambda t}\widetilde{w}d\lambda \Bigr\|_{\mathring{H}_{\eta}^{1}}+\Bigr\|\frac{1}{%
2i\pi }\int_{\Gamma }e^{\lambda t}u_{0}d\lambda \Bigl\|_{\mathring{W}^{1,\infty
}}\right) .
\end{eqnarray*}%
Since it is sufficient to estimate $(\widetilde{w}+u_{0})$ in $\mathring{H}%
_{\eta}^{1}\oplus \mathring{W}^{1,\infty }$,
%It is clear that we also have%
%\begin{eqnarray*}
%\|\widetilde{w}\|_{\mathring{H}_{\eta}^{1}} &\leq &\frac{C}{a^{2}+|\lambda |}\|%
%\widetilde{g}\|_{\mathring{H}_{\eta}^{2}}+\frac{cC^{2}\varepsilon }{\sqrt{%
%|\lambda |}(\kappa \varepsilon +\sqrt{|\lambda |})}\|f_{0}\|_{\mathring{W}%
%_{\eta}^{2,\infty }}, \\
%\|u_{0}\|_{\mathring{W}^{1,\infty }} &\leq &\frac{C}{\sqrt{|\lambda |}%
%(\kappa \varepsilon +\sqrt{|\lambda |})}\|f_{0}\|_{\mathring{W}%
%_{\eta}^{2,\infty }}+\frac{cC^{2}\varepsilon }{(a^{2}+|\lambda |)}\|\widetilde{g%
%}\|_{\mathring{H}_{\eta}^{2}},
%\end{eqnarray*}%
%and 
%for $\widetilde{g}\in \mathring{H}_{\eta}^{1}$ and $f_{0}$ in $\mathring{W}%
%_{\eta}^{1,\infty }$%
%\begin{eqnarray*}
%\|\widetilde{w}\|_{\mathring{H}_{\eta}^{1}} &\leq &\frac{C}{a^{2}+|\lambda |}\|%
%\widetilde{g}\|_{\mathring{H}_{\eta}^{1}}+\frac{cC^{2}\varepsilon }{(\kappa
%\varepsilon +\sqrt{|\lambda |})(a+\sqrt{|\lambda |})\sqrt{|\lambda |}}%
%\|f_{0}\|_{\mathring{W}_{\eta}^{1,\infty }}, \\
%\|u_{0}\|_{\mathring{W}^{1,\infty }} &\leq &\frac{C}{\sqrt{|\lambda |}%
%(\kappa \varepsilon +\sqrt{|\lambda |})}\|f_{0}\|_{\mathring{W}%
%_{\eta}^{1,\infty }}+\frac{cC^{2}\varepsilon }{(a^{2}+|\lambda |)(\kappa
%\varepsilon +\sqrt{|\lambda |})}\|\widetilde{g}\|_{\mathring{H}_{\eta}^{1}},
%\end{eqnarray*}%
similarly, we have
\begin{equation*}
\Bigl|\frac{1}{2i\pi }\int_{\Gamma }e^{\lambda t}\frac{\varepsilon }{\lambda
-\sigma _{\varepsilon }}F(\widetilde{w}+u_{0})d\lambda \Bigr|\leq \frac{%
cM(\varepsilon) }{\varepsilon (1+t)}\left( 
\|\widetilde{g}\|_{\mathring{H}_{\eta}^{1}}
+\|f_{0}\|_{\mathring{W}%
_{\eta}^{1,\infty }}\right) .
\end{equation*}
Finally, we obtain the following estimates for a certain constant $%
C(\varepsilon )$%
$$
\|e^{\boldsymbol{A}_{\varepsilon }t}u\|_{\mathcal{Z}_\eta}
\leq {C(\varepsilon) \over 1 + t} \| u \|_{\mathcal{Z}_{\eta,\eta}}.
$$
Similarly, for $\widetilde{g}\in \mathring{H}_{\eta}^{1}$ and $f_{0}$ in $%
\mathring{W}_{\eta}^{1,\infty }$ we obtain%
\begin{eqnarray*}
\|\widetilde{P}e^{\boldsymbol{A}_{\varepsilon }t}u\|_{\mathcal{L}(\mathring{H%
}_{\eta}^{2})} &\leq &\frac{C(\varepsilon )}{\sqrt{t} (1 + \sqrt{t})}\|\widetilde{P}u\|_{%
\mathring{H}_{\eta}^{1}}+\frac{C(\varepsilon )}{1+{t}}\|P_{0}u\|_{%
\mathring{W}_{\eta}^{1,\infty }}, \\
\|P_{0}e^{\boldsymbol{A}_{\varepsilon }t}u\|_{\mathcal{L}(\mathring{W}%
^{2,\infty })} &\leq &\frac{C(\varepsilon )}{1+{t}}\|\widetilde{P}u\|_{%
\mathring{H}_{\eta}^{1}}+\frac{C(\varepsilon )}{\sqrt{t}(1 + \sqrt{t})}\|P_{0}u\|_{\mathring{W%
}_{\eta}^{1,\infty }}.
\end{eqnarray*}

%%%%%%%%%%%%%%%%%%%%%%%%%%%%%%%%%%

\subsubsection*{Acknowledgments}  

The authors would like to warmly thank Mariana Haragus
for allowing our collaboration and her valuable help
and comments.
D. Bian is supported by NSFC under the contract 12271032.

\subsubsection*{Conflict of interest}  The authors state that there is no conflict of interest.

\subsubsection*{Data availability}  Data are not involved in this research paper.

%%%%%%%%%%%%%%%%%%%%%%%%%%%%

%%%%%%

\end{document}